\documentclass[11pt]{amsart}

\usepackage{comment}
\usepackage[official]{eurosym}
\usepackage[greek,english]{babel}
\usepackage{bbm}
\usepackage[utf8]{inputenc}
\usepackage[T1]{fontenc}
\usepackage{pdfpages}
\usepackage{tensor}
\usepackage{enumerate}

\usepackage[stable]{footmisc}
\usepackage{mathtools}
\usepackage[mathscr]{eucal}
\usepackage[a4paper, twoside=false, vmargin={2cm,3cm}, includehead]{geometry}
\usepackage{pgfgantt}
\usepackage{graphicx}
\usepackage{xcolor}
\ganttset{group/.append style={orange},
	milestone/.append style={red},
	progress label node anchor/.append style={text=red}}
\theoremstyle{plain}
\newtheorem{theorem}{Theorem}[section]
\newtheorem{lemma}[theorem]{Lemma}
\newtheorem{proposition}[theorem]{Proposition}

\newtheorem{conjecture}{Conjecture}
\newtheorem{remark}{Remark}
\newtheorem{corollary}[theorem]{Corollary}

\newtheorem{definition}[theorem]{Definition}

\newtheorem*{Graal1*}{Special Case of Problem~\ref{Graal1}}
\newtheorem*{Graal2*}{Special Case of Problem~\ref{Graal2}}

\newtheorem*{theorem*}{Theorem}

\newtheorem{innercustomthm}{Theorem}
\newenvironment{customthm}[1]
{\renewcommand\theinnercustomthm{#1}\innercustomthm}
{\endinnercustomthm}
\usepackage{amssymb}

\theoremstyle{remark}
\newtheorem*{remark*}{{\bf Remark}}
\newtheorem*{remarks*}{{\bf Remarks}}
\newtheorem*{comment*}{{\bf Comment}}

\usepackage{amsmath}

\newcommand{\ue}{{\underline{\varepsilon}}}
\newcommand{\uh}{{\underline{h}}}

\newcommand{\wt}{\widetilde}

\newcommand{\bignorm}[1]{\big\lVert #1 \big\rVert}
\newcommand{\Bignorm}[1]{\Big\lVert #1 \Big\rVert}
\newcommand{\norm}[1]{\left\Vert #1\right\Vert}
\newcommand{\N}{\mathbb{N}}
\newcommand{\Q}{\mathbb{Q}}
\newcommand{\R}{\mathbb{R}}
\newcommand{\T}{\mathbb{T}}
\newcommand{\Z}{\mathbb{Z}}

\newcommand{\B}{\mathcal{B}}
\newcommand{\C}{\mathbb{C}}

\newcommand{\D}{\Delta}

\newcommand{\G}{\Gamma}

\newcommand{\X}{\mathcal{X}}

\newcommand{\m}{\mu}

\newcommand{\bv}{{\bf v}}

\newcommand{\La}{\Lambda}

\renewcommand{\P}{\mathbb{P}}

\newcommand{\e}{\varepsilon}

\DeclareMathOperator*{\E}{\mathbb{E}}

\def \colon{{:}\;}
\newcommand{\floor}[1]{{\left \lfloor #1 \right \rfloor}}

\newcommand{\bigabs}[1]{\bigl| #1 \bigr|}
\newcommand{\Bigabs}[1]{\Bigl| #1 \Bigr|}

\usepackage{fancyhdr}
\usepackage{theoremref}

\usepackage{hyperref}

\title{Ergodic averages for sparse sequences along primes}
\author{Andreas Koutsogiannis and Konstantinos Tsinas}

\address[Andreas Koutsogiannis]{
Department of Mathematics, Aristotle University of Thessaloniki, Thessaloniki 54124, Greece}
\email{akoutsogiannis@math.auth.gr}

\address[Konstantinos Tsinas]{Department of Mathematics and Applied Mathematics, University of Crete, Voutes University Campus, Heraklion 70013, Greece}
\email{kon.tsinas@gmail.com}

\thanks{The second author was supported by ELIDEK--Fellowship number 5367 (3rd Call for HFRI Ph.D. Fellowships).}



\begin{document}

\begin{abstract}
We investigate the limiting behavior of multiple ergodic averages along sparse sequences  evaluated at prime numbers. Our sequences arise from smooth and well-behaved functions that have polynomial growth.
Central to this topic is a comparison result between standard Ces\'{a}ro averages along positive integers and averages weighted by the (modified) von Mangoldt function. The main ingredients are a recent result of Matom\"{a}ki, Shao, Tao and Ter\"{a}v\"{a}inen on the Gowers uniformity of the latter function in short intervals, a lifting argument that allows one to pass from actions of integers to flows, a simultaneous (variable) polynomial approximation in appropriate short intervals, and some quantitative
equidistribution results for the former polynomials. We derive numerous applications in multiple recurrence, additive combinatorics, and equidistribution in nilmanifolds along primes. In particular, 
we deduce that any set of positive density contains arithmetic progressions with step $\lfloor p^c \rfloor$, where $c$ is a positive non-integer and $p$ denotes a prime, establishing a conjecture of Frantzikinakis. 
\end{abstract}

\maketitle
\tableofcontents

\section{Introduction and main results}\label{Section-Introduction}
\subsection{Motivation}

The seminal work of Furstenberg \cite{Furstenberg-original} towards Szemer\'{e}di's theorem has initiated significant interest in the study of ergodic theoretic problems and their applications in problems of combinatorial or number-theoretic nature. Dynamical methods have proven extremely effective at tackling problems relating to the combinatorial richness of positive density subsets of integers. Furthermore, there are frequently no alternative methods that can recover the results that ergodic theoretic tools provide. For instance, we have far-reaching generalizations of Szemer\'{e}di's theorem, such as the Bergelson-Leibman theorem \cite{Bergelson-Leibman-polynomial-VDW} that produces polynomial progressions on sets of integers with positive density.

The general structure of our problems is the following: we are
given a collection of sequences $a_1(n),\dots, a_k(n)$ of integers and
a standard probability space $(X,\mathcal{X},\m)$ equipped with invertible, commuting, measure-preserving transformations $T_1,\dots, T_k$ that act on $X,$ and we examine the limiting behavior of the multiple averages \begin{equation}\label{E: general ergodic averages}
  \frac{1}{N}\sum_{n=1}^{N} T_1^{a_1(n)}f_1\cdot\ldots\cdot T_k^{a_k(n)}f_k.
\end{equation}
Throughout the article, these assumptions on the transformations will be implicit; we call the tuple $(X,\X,\mu, T_1,\ldots, T_k)$ a \emph{measure-preserving system} (or just \emph{system}).
Here $f_1,\dots, f_k$ are functions in $L^{\infty}(\m)$ and we concern ourselves with their convergence mainly in the $L^2$-sense.  In view of Furstenberg’s correspondence principle, a satisfactory answer to this problem typically ensures that sets with positive density possess patterns of the form $(m,m+a_1(n),\dots, m+a_k(n))$, where $m,n\in \N$.
Specializing to the case where all the sequences are equal and $T_i=T^i$, we arrive at the averages \begin{equation}\label{E: Furstenberg averages}
    \frac{1}{N}\sum_{n=1}^{N} T^{a(n)}f_1\cdot T^{2a(n)}f_2\cdot\ldots\cdot T^{ka(n)}f_k,
\end{equation}which relate to patterns of arithmetic progressions, whose common difference belongs to the set $\{a(n)\colon n\in \N\}$.

Furthermore, it is particularly tempting to conjecture that results pertaining to mean convergence of the averages in \eqref{E: general ergodic averages} 
should still be valid, if we restrict the range of summation to a sparse set such as the primes. Normalizing appropriately, we contemplate whether or not the averages \begin{equation}\label{E: general ergodic averages along the primes}
    \frac{1}{\pi(N)} \sum_{ p\in \P\colon p\leq N}T_1^{a_1(p)}f_1\cdot\ldots\cdot T_k^{a_k(p)}f_k
\end{equation}converge in $L^2(\m)$ and what is the corresponding limit of these averages.
Here, $\pi(N)$ denotes the number of primes less than or equal to $N$ and $\P$ is the set of primes. 

The first results in this direction were established in the case $k=1$. Namely, S\'{a}rk\"{o}zy \cite{sarkozy3} used methods from analytic number theory to show that sets of positive density contain patterns of the form $(m,m+p-1)$, where $p$ is a prime.\footnote{Throughout this article, it will be a reoccurring theme that in combinatorial applications, certain arithmetic obstructions force one to consider the set of shifted primes $\P-1$ (or $\P+1$) in place of $\P$, when dealing with polynomials. This is a necessary assumption, as in such cases the corresponding results for the set $\P$ are easily seen to be incorrect (see, for example, \cite[Remark 1.4]{Wenbo-primes}). } Additionally, Wierdl \cite{Wierdl-primes} established the even stronger pointwise convergence result for the averages \eqref{E: general ergodic averages along the primes} in the case $k=1$ and $a_1(n)=n$, while Nair generalized this theorem to polynomials evaluated at primes \cite{nair-primes}.

In the setting of several iterates, the first results were provided by Frantzikinakis, Host, and Kra \cite{Fra-Ho-Kra-primes-1}, who established that sets of positive density contain 3-term arithmetic progressions whose common difference is a shifted prime. Furthermore, they demonstrated that the averages in \eqref{E: general ergodic averages along the primes} converge in the case $k=2$, $T_1=T_2$ and $a_i(n)=in,\ i\in\{1,2\}$. 
This was generalized significantly by Wooley and Ziegler \cite{Wooley-Ziegler} to hold in the case that the sequences $a_i(n),\ i\in \{1,\dots,k\}$ are polynomials with integer coefficients and the transformations $T_1,\dots, T_k$ are the same. Following that, Frantzikinakis, Host, and Kra confirmed the validity of the Bergelson-Leibman theorem in \cite{Fra-Host-Kra-primes} along the shifted primes. In addition, they showed that the averages in \eqref{E: general ergodic averages along the primes} converge in norm when $a_i(n)$ are integer polynomials. Furthermore, Sun obtained convergence and recurrence results in \cite{Wenbo-primes} for a single transformation and iterates of the form $i\floor{an},\ i\in \{1,\dots, k\}$ or $\floor{ja n}, \ j\in\{1,\dots,k\},$ with $a$ irrational.
Finally, using the convergence results in \cite{Koutsogiannis-correlations} 
along $\N$ for integer parts of real polynomials and several transformations,
the first author extended the convergence result of \cite{Fra-Host-Kra-primes} 
to real polynomials 
in \cite{koutsogiannis-closest}, obtaining recurrence for polynomials with real coefficients rounded to the closest integer.
In all of the previous cases, combinatorial applications along the shifted primes were derived as well.

In the case of multiple iterates, a shared theme in the methods used has been the close reliance on the deep results provided by the work of Green and Tao in their effort to show that primes contain arbitrarily long arithmetic progressions \cite{Green-Tao-primes-progressions}. For instance, all results\footnote{The methods in \cite{Wooley-Ziegler} do not invoke the full power of this theorem, although their approach draws heavily from the work of Green and Tao.} relied on the Gowers uniformity of the (modified) von Mangoldt function that was established in \cite{Green-Tao-linearequationsprimes} conditional to two deep conjectures, which were subsequently verified in \cite{GreeN-Tao-Ziegler-inversetheorem} and \cite{Green-Tao-Mobius}.

It was conjectured by Frantzikinakis that the polynomial theorems along primes should hold for more general sequences involving fractional powers $n^c$, such as $\floor{n^{3/2}}, \floor{n^{\sqrt{2}}}$ or even linear combinations thereof.
Indeed, it was conjectured in \cite{Fra-Hardy-singlecase} that the 
sequence $\floor{p_n^c}$, where $c$ is a positive non-integer and $p_n$ is the sequence of primes is good for multiple recurrence and convergence. To be more precise, he conjectured that the averages \begin{equation}\label{E: Szemeredi averages for p^c}
    \frac{1}{\pi(N)}\sum_{p\in \P\colon p\leq N}^{} T^{\floor{p^c}}f_1\cdot\ldots\cdot T^{k\floor{p^c}}f_k
\end{equation}converge in $L^2(\m)$ for all positive integers $k$ and all positive non-integers $c$.
Analogously, we have the associated multiple recurrence conjecture, namely that all sets of positive upper density contain $k$-term arithmetic progressions with common difference of the form $\floor{p^c}$.
When $0<c<1$, one can leverage the fact that the range of $\floor{p_n^c}$ contains all sufficiently large integers to establish the multiple recurrence result. Additionally, 
the convergence of the previous averages is known in the case $k=1$ since one can use the spectral theorem and the fact that the sequence $\{p_n^c a \}$ is equidistributed mod 1 for all non-zero $a\in \R$. This last assertion follows from \cite{Stux} or \cite{Wolke} when $c<1$ and \cite{Leitmann} in the case $c>1$.

There were significant obstructions to the solution of this problem.  One approach would be to modify the comparison method from \cite{Fra-Host-Kra-primes} (concerning polynomials), but  the Gowers uniformity of the von Mangoldt functions is insufficient to establish this claim.
The other approach would be to use the method of characteristic factors, which is based on the structure theorem of Host-Kra \cite{Host-Kra-annals}. Informally, this reduces the task  of proving convergence to a specific class of systems with special algebraic structure called nilmanifolds. However, this required some equidistribution results on nilmanifolds for the sequence $\floor{p_n^c}$, which were very difficult to establish.

A similar conjecture by Frantzikinakis was made
for more general averages of the form \begin{equation*}
    \frac{1}{\pi(N)} \sum_{p\in \P\colon p\leq N} T^{\floor{p^{c_1}}}f_1\cdot\ldots\cdot T^{\floor{p^{c_k}}}f_k
\end{equation*}for distinct positive non-integer $c_1,\dots, c_k$.
The recent result of Frantzikinakis \cite{Fra-primes} verifies that these averages converge in $L^2(\m)$ to the product of the integrals of the functions $f_1,\dots, f_k$ in any ergodic system, even in the more general case where the sequences in the iterates are linearly independent fractional polynomials. The number theoretic input required is a sieve-theoretic upper bound for the number of tuples of primes of a specific form, as well as an equidistribution result on fractional powers of primes in the torus that was already known. These methods relied heavily on the use of the joint ergodicity results in \cite{Fra-jointly-ergodic} and, thus, the linear independence assumption on the fractional polynomials was absolutely essential. In the same paper, it was conjectured \cite[Problem]{Fra-primes} that the case of fractional polynomials can be generalized to a significantly larger class of functions of polynomial growth, called Hardy field functions, which we consider below. 
The conjecture asks for necessary and sufficient conditions so that the averages along primes converge to the product of the integrals in ergodic systems.
The arguments in \cite{Fra-primes} cannot cover this larger class of functions,\footnote{A more fundamental obstruction in this more general setting was that the necessary seminorm estimates were unavailable even in the simplest case of averages along $\N$, apart from some known special cases. This was established a few months later by the second author \cite{Tsinas}.} as it was remarked in Subsection 1.3 of that article.

In this article, our objective is to strengthen the convergence results in \cite{Fra-Host-Kra-primes} and \cite{Fra-primes} and resolve the convergence problem of the averages in \eqref{E: Szemeredi averages for p^c}. Actually, there is no advantage in confining ourselves to sequences of the form $\floor{p^c}$, so we consider the more general class of sequences arising from Hardy field functions of polynomial growth (see Section \ref{Section: Background} for the general definition), which, loosely speaking, are functions with pleasant behavior (such as smoothness, for instance). 
The prototypical example of a Hardy field is the field $\mathcal{LE}$ of logarithmico-exponential functions, which are defined by a finite combination of the operations $+,-,\times, \div$ and the functions $\exp,\log$ acting on a real variable $t$ and real constants.  
For instance, the field $\mathcal{LE}$ contains the functions
$\log^{3/2}t,$ $t^{\pi},$ $t^{17}\log t+\exp(\sqrt{t^{\log t}+\log\log t  }).$ The fact that $\mathcal{LE}$ is a Hardy field was established in \cite{Hardy-1} and the reader can
keep this in mind as a model case throughout this article.
We resolve several conjectures involving the convergence of the averages in \eqref{E: general ergodic averages along the primes} along Hardy sequences. Consequently, we derive several applications in recurrence and combinatorics that expand the known results in the literature.
Finally, we also establish an equidistribution result in nilmanifolds for sequences evaluated at primes.

\subsection{Main results} We present here our main theorems. We start by stating our mean convergence results, followed by their applications to  multiple recurrence and combinatorics, and conclude our presentation with the equidistribution results in nilmanifolds.
We will assume below that we are working with a Hardy field $\mathcal{H}$ that contains the polynomial functions. This assumption is not necessary, but it simplifies the proofs of our main theorems. Besides, this restriction is very mild and the most interesting Hardy fields contain the polynomials. A few results impose additional assumptions on $\mathcal{H}$ and we state those when necessary. These extra assumptions are a byproduct of convergence results along $\N$ in the literature that were proved under these hypotheses and we will not need to use the implied additional structure on $\mathcal{H}$ in any of our arguments. 

\subsubsection{Comparison between averaging schemes}

For many number-theoretic problems, a suitable proxy for capturing the distribution of the prime numbers is the von-Mangoldt function, which is defined on $\N$ by \begin{equation}\label{E: Definition of con Mangoldt}
    \La(n)=\begin{cases}
        \log p&, \ \text{if } n=p^k\ \text{for some prime } p  \text{ and } k\in \N\\
        0&, \ \text{otherwise}
    \end{cases}.
\end{equation}
The function $\La$ has mean value 1 by the prime number theorem. Usually, the prime powers with exponents at least 2 contribute a term of significantly lower order in asymptotics, so one can think of $\La$ as being supported on primes.
However, due to the irregularity of the distribution of $\La$ in residue classes to small moduli, one typically considers a modified version of $\La$, called the W-tricked version. To define this, let $w$ be a positive integer and let $W=\prod_{p\leq w,p\in \P} p$. Then, for any integer $1\leq b\leq W$ with $(b,W)=1$, we define the $W$-tricked von Mangoldt function $\La_{w,b}$ by \begin{equation}\label{E: defintion of W-tricked von Mangoldt}
    \La_{w,b}(n)=\frac{\phi(W)}{W}\La(Wn+b),
\end{equation}where $\phi$ denotes the Euler totient function.

Our main result provides a comparison between ergodic averages along primes and averages along natural numbers. This will allow us to transfer mean convergence results for Ces\`{a}ro averages to the prime setting, answering numerous conjectures regarding norm convergence of averages as those in \eqref{E: general ergodic averages along the primes} followed by applications in multiple recurrence and combinatorics. We explain the choice of the conditions on the functions $a_{ij}$ in Subsection~\ref{strategysubsection}. Roughly speaking, the first condition implies that the sequence $a_{ij}$ is equidistributed mod 1 due to a theorem of Boshernitzan (see Theorem \ref{T: Boshernitzan} in Section \ref{Section: Background}).

\begin{theorem}\label{T: the main comparison}
    Let $\ell,k$ be positive integers and, for all $1\leq i\leq k,\ 1\leq j\leq \ell$, let $a_{ij}\in \mathcal{H}$ be functions of polynomial growth such that \begin{equation}\label{E: far away from rational polynomials}
        \lim_{t\to+\infty}
        \left|\frac{a_{ij}(t)-q(t)}{\log t} \right|=+\infty\ \text{ for every polynomial }  q(t)\in \Q[t],
        \end{equation} or \begin{equation}\label{E: essentially equal to a polynomial}
           \lim\limits_{t\to+\infty} |a_{ij}(t)-q(t)|=0\  \text{ for some polynomial }  q(t)\in \Q[t]+\R.
        \end{equation}Then, for any measure-preserving system $(X,\X, \m, T_1,\dots, T_k)$ and functions $f_1,\dots,f_{\ell}\in L^{\infty}(\m)$, we have  \begin{equation*}\label{E: main average in Proposition P: the main comparison}
         \lim_{w\to+\infty} \  \limsup\limits_{N\to+\infty}\  \max_{\underset{(b,W)=1}{1\leq b\leq W}}\ \Bignorm{\frac{1}{N}\sum_{n=1}^{N} \big(\La_{w,b}(n) -1\big) \prod_{j=1}^{\ell}\big(  \prod_{i=1}^{k} T_i^{\floor{a_{ij}(Wn+b)}}       \big)f_j   }_{L^2(\m)}=0.
        \end{equation*}
\end{theorem}

\begin{remark}\label{R: replacing rounding functions} 
    We can easily verify that each of the integer parts can be individually replaced by other rounding functions, such as the ceiling function (which we denote by $\lceil\cdot\rceil$) or the closest integer function (denoted by $[[\cdot]]$). This is an immediate consequence of the identities $\lceil  x\rceil=-\floor{-x}$ and $[[x]]=\floor{x+1/2},$ for all $x\in\R$ and the fact that the affine shifts (by rationals) $q_1 a_{ij}+q_2, q_1,q_2\in \Q$, still satisfy \eqref{E: far away from rational polynomials} or \eqref{E: essentially equal to a polynomial} if $a_{ij}$ does.
\end{remark}

Theorem \ref{T: the main comparison} is the main tool that we use to derive all of our applications.
The bulk of the article is aimed towards establishing it and everything else is practically a corollary (in combination with known norm convergence theorems for Ces\`{a}ro averages). We remark that unlike several of the theorems below, there are no ``independence'' assumptions between the functions $a_{ij}$, although, in applications, we will need to impose analogous assumptions to ensure convergence of the averages, firstly along $\N$, and then along $\P$. In order to clarify how the comparison works, we present the following theorem, which is effectively a corollary of Theorem \ref{T: the main comparison} and which shall be proven in Section \ref{Section-Proofs of remaining theorems}.

\begin{theorem}\label{T: criterion for convergence along primes}
  Let $\ell, k $ be positive integers, $(X,\X, \m, T_1,\dots, T_k)$ be a measure-preserving system and $f_1,\dots, f_k\in L^{\infty}(\m)$. 
  Assume that for all $1\leq i\leq k,\ 1\leq j\leq \ell$, $a_{ij}\in\mathcal{H}$ are functions of polynomial growth such that the following conditions are satisfied:\\
  (a) 
  Each one of the functions $a_{ij}(t)$ satisfies either $\eqref{E: far away from rational polynomials}$ or \eqref{E: essentially equal to a polynomial}.\\
  (b) For all positive integers $W,b$, the averages \begin{equation}\label{E: averages along Wn+b}
      \frac{1}{N}\sum_{n=1}^{N} \big(  \prod_{i=1}^{k} T_i^{\floor{a_{i1}(Wn+b)}}       \big)f_1\cdot\ldots \cdot  \big(  \prod_{i=1}^{k} T_i^{\floor{a_{i\ell}(Wn+b)}}       \big)f_{\ell}
  \end{equation}converge in $L^2(\m)$. 
  
  Then, the averages \begin{equation}\label{E: averages along primes converge}
      \frac{1}{\pi(N)}\sum_{p\in \P\colon p\leq N} \big(  \prod_{i=1}^{k} T_i^{\floor{a_{i1}(p)}}       \big)f_1\cdot\ldots \cdot  \big(  \prod_{i=1}^{k} T_i^{\floor{a_{i\ell}(p)}}       \big)f_{\ell}
  \end{equation}converge in $L^2(\m)$.

  Furthermore, if the averages in \eqref{E: averages along Wn+b} converge to the function $F\in L^{\infty}(\m)$ for all positive integers $W,b$, then the limit in $L^2(\m)$ of the averages \eqref{E: averages along primes converge} is equal to $F$. 
\end{theorem}

In the setting of Hardy field functions, the fact that we require convergence for sequences 
along arithmetic progressions is typically harmless. Indeed, convergence results along $\N$ typically follow from a growth condition on the implicit functions $a_{ij}$ (such as \eqref{E: far away from rational polynomials}) and it is straightforward to check that the function $a_{ij}(Wt+b)$ satisfies a similar growth condition as well. Therefore, one can think of the second condition morally as asking to establish convergence in the case $W=1$. 

The final part of Theorem \ref{T: criterion for convergence along primes} allows us to compute the limit of averages along primes in cases where we have an expression for the limit of the standard Ces\`{a}ro averages. This is possible, in rough terms, whenever the linear combinations of the functions $a_{ij}$ do not contain polynomials or functions that are approximately equal to a polynomial. The reason for that is that there is no explicit description of the limit of polynomial ergodic averages in a general measure preserving system (although one can get a simplified expression in special cases, or under some total ergodicity assumptions on the system).

\subsubsection{Convergence of ergodic averages along primes}

The foremost application is that the averages in \eqref{E: Furstenberg averages} converge when $a(n)$ is a Hardy sequence
and when we average along primes. This will also lead to generalizations of Szemer\'{e}di's theorem in our applications.
The following theorem is a corollary of our comparison and the convergence results in \cite{Fra-Hardy-singlecase} (specifically, Theorems 2.1 and 2.2 of that paper). In conjunction with the corresponding recurrence result of Theorem \ref{T: multiple recurrence for szemeredi type patterns} below, we get an affirmative answer to a stronger version of \cite[Problem 7]{Fra-Hardy-singlecase} (this problem also reappeared in \cite[Problem 27]{Fra-open}), 
which was stated only for sequences of the form $n^c, c\in \R^{+}\setminus \N$.

\begin{theorem}\label{T: convergence of Furstenberg averages}
    Let $a\in \mathcal{H}$ be a function of polynomial growth that satisfies either\begin{equation} \label{E: far away from real multiples of integer polynomials}
    \lim_{t\to+\infty}
        \left|\frac{a(t)-cq(t)}{\log t} \right|=+\infty \text{ for every } c\in \R \text{ and every } q\in \Z[t],\end{equation}or \begin{equation}\label{E: equal to a real multiple of integer polynomial}
            \lim\limits_{t\to+\infty}|a(t)-cq(t)|=d\ \text{for some } c,d\in \R \ \text{and some } q\in \Z[t].
        \end{equation}
     Then, for any positive integer $k$, any measure-preserving system $(X,\X,\m,T)$ and functions $f_1,\dots,f_k\in L^{\infty}(\m)$, we have that the averages \begin{equation}\label{E: aek ole}
         \frac{1}{\pi(N)}\sum_{p\in \mathbb{P}\colon p\leq N} T^{\floor{a(p)}}f_1\cdot\ldots\cdot T^{k\floor{a(p)}}f_k
    \end{equation}converge in $L^2(\m)$.
    
    In particular, if $a$ satisfies \eqref{E: far away from real multiples of integer polynomials}, the limit of the averages in \eqref{E: aek ole} is equal to the limit in $L^2(\m)$ of the averages \begin{equation*}
        \frac{1}{N}\sum_{n=1}^{N} T^nf_1\cdot\ldots\cdot T^{kn} f_k.
    \end{equation*}
\end{theorem}

\begin{comment*}
    We can replace the floor function in \eqref{E: aek ole} with either the function $\lceil\cdot\rceil$ or the function  $[[\cdot ]]$. The assumption that the iterates are Hardy field functions can also be relaxed. We discuss this more in Section \ref{Section: more general iterates}.
 \end{comment*}

Observe that there is only one function appearing in the statement of the previous theorem.
The following convergence results concern the case where we may have several different Hardy field functions. In both cases, there are some "independence" assumptions between the functions involved, which has the advantage of providing an exact description of the limit for the averages along $\N$. Thus, we can get a description for the limit along $\P$ as well.

The following theorem concerns the ``jointly ergodic'' case for one transformation, which refers to the setting when we have convergence to the product of the integrals in ergodic systems. Theorem~\ref{T: the main comparison} combines with \cite[Theorem 1.2]{Tsinas} to provide the next result.
This generalizes the theorem of Frantzikinakis \cite[Theorem 1.1]{Fra-primes} and gives a positive answer to \cite[Problem]{Fra-primes}. Unlike the previous theorem, we have to impose here an additional assumption on $\mathcal{H}$, since the respective convergence result along $\N$ is established under this condition. The field $\mathcal{LE}$ does not have the property appearing in the ensuing theorem, but it is contained in the Hardy field of Pfaffian functions, which does (for the definition, see \cite[Section 2]{Tsinas}).

\begin{theorem}\label{T: jointly ergodic case}
Let $\mathcal{H}$ be a Hardy field that contains $\mathcal{LE}$ and is closed under composition and compositional inversion of functions, when defined.\footnote{This means that if $f,g\in \mathcal{H}$ are such that $g(t)\to+\infty$, then $f\circ g\in \mathcal{H}$ and $g^{-1}\in \mathcal{H}$.} For a positive integer $k,$ let  $a_1,\dots,a_k$ be functions of polynomial growth and assume that every non-trivial linear combination $a$ of them satisfies \begin{equation}\label{E: jointly ergodic condition}
   \lim\limits_{t\to+\infty} \Bigabs{\frac{a(t)-q(t)}{\log t} }=+\infty \text{ for every } q(t)\in \Z[t].\footnote{Actually, we can assume here the more general condition that 
$$\frac{1}{N}\sum_{n=1}^N e(t_1\floor{a_1(n)}+\ldots +t_k\floor{a_k(n)})\to 0, $$ for every $(t_1,\ldots,t_k)\in [0,1)^k\setminus\{(0,\ldots,0)\},$ where $e(x)=e^{2\pi i x},$ $x\in \mathbb{R}$ (see the remark under \cite[Theorem 1.2]{Tsinas}). This condition is necessary and sufficient in order for \eqref{E: jointly ergodic averages along primes} to hold. }
\end{equation}

Then, for any measure-preserving system $(X,\X,\m,T)$ and functions $f_1,\dots, f_k\in L^{\infty}(\m)$, we have 
    that \begin{equation}\label{E: jointly ergodic averages along primes}
   \lim\limits_{N\to+\infty}    \frac{1}{\pi(N)} \sum_{p\in \P\colon p\leq N} T^{\floor{a_1(p)}}f_1\cdot \ldots \cdot T^{\floor{a_k(p)}}f_k = \tilde{f}_1\cdot\ldots\cdot \tilde{f}_k, 
    \end{equation}where $\tilde{f}_i:=\mathbb{E}(f_i|\mathcal{I}(T))=\lim_{N\to+\infty}\frac{1}{N}\sum_{n=1}^N T^n f_i$ and the convergence is in $L^2(\m)$.
\end{theorem}

\begin{remark}
    We remark that we can also transfer the convergence result appearing in \cite[Theorem 1.3]{tsinas-pointwise} to primes, although we do not have useful information on the limiting behavior of the associated averages to deduce recurrence results.
\end{remark}

In the case of several commuting transformations, knowledge of the limiting behavior for averages along $\N$ is sparse. This is naturally a barrier to proving multidimensional analogs of our recurrence results below along primes.
Nonetheless, we have the following convergence theorem, which adapts the convergence result in \cite[Theorem 2.3]{Fra-Hardy-multidimensional} to the prime setting. By a shift-invariant Hardy field, we are referring to a Hardy field such that $a(t+h)\in \mathcal{H}$ for any $h\in \Z$ and function $a(t)\in \mathcal{H}$.

\begin{theorem}\label{T: nikos result to primes}
Let $k\in \N,$ $\mathcal{H}$ be a shift-invariant Hardy field, $a_1,\dots,a_k$ be functions in $\mathcal{H}$ with pairwise distinct growth rates and such that there exist integers $d_i\geq 0$ satisfying \begin{equation*}
    \lim\limits_{t\to+\infty} \Bigabs{\frac{a_i(t)}{t^{d_i} \log t}}=\lim\limits_{t\to+\infty} \Bigabs{ \frac{t^{d_i+1}}{a_i(t)}}=0.
\end{equation*}
Then, for any system $(X,\mathcal{X},\m,T_1,\dots, T_k)$ and functions $f_1,\dots, f_k\in L^{\infty}(\m)$, we have \begin{equation*}
     \lim\limits_{N\to+\infty}    \frac{1}{\pi(N)} \sum_{p\in \P\colon p\leq N} T_1^{\floor{a_1(p)}}f_1\cdot \ldots \cdot T_k^{\floor{a_k(p)}}f_k = \tilde{f}_1\cdot\ldots\cdot \tilde{f}_k,
\end{equation*}where $\tilde{f}_i:=\mathbb{E}(f_i|\mathcal{I}(T_i))=\lim_{N\to+\infty}\frac{1}{N}\sum_{n=1}^N T_i^n f_i$ and the convergence is in $L^2(\m)$.
\end{theorem}

While there are more restrictions compared to Theorem \ref{T: jointly ergodic case}, we note that Theorem \ref{T: nikos result to primes} shows that we have norm convergence in the case when $a_i(t)=t^{c_i}$ for distinct, positive non-integers $c_i$.

\begin{comment*}
    In the previous two theorems, we can replace the integer part with any other rounding function in each iterate individually (see Remark \ref{R: replacing rounding functions}).
\end{comment*}

\subsubsection{Applications to multiple recurrence and combinatorics}

In this subsection, we will translate the previous convergence results to multiple recurrence results and then combine them with Furstenberg's correspondence principle to extrapolate combinatorial applications. Due to arithmetic obstructions arising from polynomials, we have to work with the set of shifted primes in some cases. In addition, it was observed in \cite{koutsogiannis-closest} that in the case of real polynomials,  one needs to work with the rounding to the closest integer function instead of the floor function. Indeed, even in the case of sequences of the form $\floor{ap(n)+b}$, explicit conditions that describe multiple recurrence are very complicated (cf. \cite[Footnote 4]{Fra-Hardy-singlecase}).

Our first application relates to the averages of the form as in \eqref{E: general ergodic averages along the primes}. We have the following theorem.

\begin{theorem}\label{T: multiple recurrence for szemeredi type patterns}
    Let $a\in \mathcal{H}$ be a function of polynomial growth. Then, for any measure-preserving system $(X,\X,\m,T),$ $k\in \N,$ and set $A$ with positive measure we have the following:\\
    (a) If $a$ satisfies \eqref{E: far away from real multiples of integer polynomials}, we have \begin{equation*}
        \lim\limits_{N\to+\infty} \frac{1}{\pi(N)} \sum_{p\in \P\colon p\leq N} \m(A\cap T^{-\floor{a(p)}}A\cap \dots\cap T^{-{k\floor{a(p)}}}A )> 0.
\end{equation*}
  (b) If $a$ satisfies \eqref{E: equal to a real multiple of integer polynomial} with $cp(0)+d=0$,\footnote{ Notice here the usual necessary assumption that we have to postulate on the polynomial, i.e., to have no constant term, in order to obtain a recurrence, and, consequently, a combinatorial result.} then for any set $A$ with positive measure, the set \begin{equation*}
        \left\{n\in\N:\; \m\big(A\cap T^{-[[a(n)]]}A \cap \dots \cap T^{-k[[a(n)]]}A \big)>0\right\}
    \end{equation*}has non-empty intersection with the sets $\P-1$ or $\P+1$. 
\end{theorem}


 We recall that for a subset $E$ of $\N$, its upper density $\bar{d}(E)$ is defined by  $$\bar{d}(E):=\limsup\limits_{N\to+\infty}\frac{|E\cap\{1,\ldots,N\}|}{N}.$$ 
\begin{corollary}\label{C: Szemeredi corollary}
    For any set $E\subseteq \N$ of positive upper density, $k\in \N,$ and function $a\in \mathcal{H}$ of polynomial growth, the following holds:\\
    (a) If $a$ satisfies \eqref{E: far away from real multiples of integer polynomials}, we have\begin{equation*}
        \liminf\limits_{N\to+\infty} \frac{1}{\pi(N)}\sum_{p\in \P\colon p\leq N} \bar{d}\big(E\cap (E-\floor{a(p)})\cap \dots\cap (E-k\floor{a(p)})\big)>0.
    \end{equation*}
  (b) If $a$ satisfies \eqref{E: equal to a real multiple of integer polynomial} with $cp(0)+d=0$, then the set \begin{equation*}
        \left\{n\in\N:\; \bar{d}\big(E\cap (E-[[a(n)]]) \cap \dots \cap (E-k[[a(n)]]) \big)>0\right\}
    \end{equation*}has non-empty intersection with the sets $\P-1$ or $\P+1$.\footnote{ In this case only,  $\bar{d}(E)$ can be replaced by $d^\ast(E):=\limsup\limits_{|I|\to+\infty}\frac{|E\cap I|}{|I|}$ following the arguments from \cite{koutsogiannis-closest}, where the $\limsup$ is taken along all intervals $I\subseteq \Z$ with lengths tending to infinity.}
\end{corollary}


    Specializing to the case where $a(n)=n^c$ where $c$ is a positive non-integer, Theorem~\ref{T: convergence of Furstenberg averages} and part (a) of Theorem \ref{T: multiple recurrence for szemeredi type patterns} provide an affirmative answer to \cite[Problem 27]{Fra-open}.


\begin{remark}\label{R: remark for shifts of primes}
    In part (a) of both Theorem \ref{T: multiple recurrence for szemeredi type patterns} and Corollary \ref{C: Szemeredi corollary}, one can evaluate the sequences along $p+u$ instead of $p$, for any $u\in \Z$, or even more generally along the affine shifts $ap+b$ for $a, b\in \Q$ with $a\neq 0$.
    This follows from the fact that the function $a_i(at+b)$ satisfies \eqref{E: far away from real multiples of integer polynomials} as well.
    However, the shifts $p-1$ and $p+1$ are the only correct ones in part (b) of Theorem \ref{T: multiple recurrence for szemeredi type patterns}.  Notice also that the function $\floor{\cdot}$ can be replaced by $\lceil\cdot\rceil$ or $[[\cdot]]$ in part (a) of the two previous statements.
\end{remark}

Now, we state the recurrence result obtained by Theorem \ref{T: jointly ergodic case}.
\begin{theorem}\label{T: multiple recurrence in the jointly ergodic case}
    Let $k\in\N,$ $\mathcal{H}$ be a Hardy field that contains $\mathcal{LE}$ and is closed under composition and compositional inversion of functions, when defined, and suppose $a_1,\dots,a_k\in \mathcal{H}$ are functions of polynomial growth whose non-trivial linear combinations satisfy \eqref{E: jointly ergodic condition}. Then, for any measure-preserving system $(X,\X,\m,T),$ and set $A$ with positive measure, we have that \begin{equation*}
        \lim\limits_{N\to+\infty} \frac{1}{\pi(N)}\sum_{p\in \P\colon p\leq N} \m\big(A\cap T^{-\floor{a_1(p)}}A \cap \dots \cap T^{-\floor{a_k(p)}}A \big)\geq \big(\m(A)\big)^{k+1}.
    \end{equation*}
\end{theorem}

\begin{corollary}
   For any $k\in\N,$ set $E\subseteq \N$ of positive upper density, Hardy field $\mathcal{H}$ and functions  $a_1,\dots, a_k\in \mathcal{H}$ as in Theorem \ref{T: multiple recurrence in the jointly ergodic case}, we have \begin{equation*}
       \liminf\limits_{N\to+\infty} \frac{1}{\pi(N)}\sum_{p\in \P\colon p\leq N} \bar{d}\big(E\cap (E-\floor{a_1(p)})\cap \dots \cap (E-\floor{a_k(p)})\big)\geq \big(\bar{d}(E)\big)^{k+1}.
   \end{equation*}
\end{corollary}
In particular, we conclude that for any set $E\subseteq \N$ with positive upper density and $a_1,\dots, a_k$ as above, the set
    \begin{equation*}
        \{n\in \N\colon \text{ there exists } m\in\N \text{ such that } m, m+\floor{a_1(n)},\ldots, m+\floor{a_k(n)} \in  E \}
    \end{equation*} has non-empty intersection with the set $\P$.

The following is a multidimensional analog of Theorem~\ref{T: multiple recurrence in the jointly ergodic case} and relies on the convergence result of Theorem~\ref{T: nikos result to primes}.
\begin{theorem}\label{T: multidimensional recurrence for primes}
Let $k\in\N,$ $\mathcal{H}$ be a shift-invariant Hardy field and suppose that $a_1,\dots,a_k\in \mathcal{H}$ are functions of polynomial growth that satisfy the hypotheses of Theorem \ref{T: nikos result to primes}. Then, for any system $(X,\X,\m,T_1,\dots, T_k)$ and set $A$ with positive measure, we have that \begin{equation*}
    \lim\limits_{N\to+\infty} \frac{1}{\pi(N)} \sum_{p\in \P \colon p\leq N} \m\big(A\cap T_1^{-\floor{a_1(p)}}A\cap  \dots  \cap T_k^{-\floor{a_k(p)}}A  \big)\geq \big(\m(A)\big)^{k+1}.
\end{equation*}
\end{theorem}


Lastly, we present the corresponding combinatorial application of our last multiple recurrence result.
We recall that for a set $E\subseteq \Z^d,$ its \emph{upper density} is given by $$\Bar{d}(E):=\limsup_{N\to+\infty}\frac{|E\cap\{-N,\ldots,N\}^d|}{(2N)^d}.$$

\begin{corollary}
   For any $k\in\N,$ set $E\subseteq \Z^d$ of positive upper density, Hardy field $\mathcal{H}
   $ and functions  $a_1,\dots, a_k\in \mathcal{H}$ as in Theorem~\ref{T: multidimensional recurrence for primes} and vectors $\bv_1,\ldots,\bv_k\in \Z^d$, we have 
   \begin{equation*}
       \liminf\limits_{N\to+\infty} \frac{1}{\pi(N)}\sum_{p\in \P\colon p\leq N} \bar{d}\big(E\cap (E-\floor{a_1(p)}\bv_1)\cap \dots \cap (E-\floor{a_k(p)}\bv_k)\big)\geq \big(\bar{d}(E)\big)^{k+1}.
   \end{equation*}
   

\end{corollary}

\begin{comment*}
    Once again, we remark that in the recurrence results in both Theorem \ref{T: multiple recurrence in the jointly ergodic case} and Theorem \ref{T: multidimensional recurrence for primes} and the corresponding corollaries, one can replace $p$ with any other affine shift $ap+b$ with $a, b\in \Q$ $(a\neq 0),$  as we explained in Remark \ref{R: remark for shifts of primes}. 
In addition, one can replace the floor functions with either $\lceil \cdot  \rceil$ or $[[\cdot]]$. 
\end{comment*}



\subsubsection{Equidistribution in nilmanifolds}

In this part, we present some results relating to pointwise convergence in nilmanifolds along Hardy sequences evaluated at primes. 
We have the following theorem that is similar in spirit to Theorem~\ref{T: criterion for convergence along primes}.

\begin{theorem}\label{T: criterion for pointwise convergence along primes-nil version}
Let $k$ be a positive integer. Assume that $a_1,\dots, a_k\in \mathcal{H}$ are functions of polynomial growth, such that the following conditions are satisfied:\\
 (a) For every $1\leq i\leq k$, the function $a_{i}(t)$ satisfies either $\eqref{E: far away from rational polynomials}$ or \eqref{E: essentially equal to a polynomial}.\\
 (b) For all positive integers $W,b$, any nilmanifold $Y=H/\Delta$, pairwise commuting elements $u_1,\dots,u_k$ and points $y_1,\dots, y_k\in Y$, the sequence \begin{equation*}
    \Big( u_1^{\floor{a_1(Wn+b)}}y_1, \ldots, u_k^{\floor{a_k(Wn+b)}}   y_k\Big)
 \end{equation*} is equidistributed on the nilmanifold $\overline{(u_1^{\Z} y_1)}\times\dots \times~ \overline{(u_k^{\Z} y_k)}$.

  
  Then, for any nilmanifold $X=G/\G$, pairwise commuting elements $g_1,\dots, g_k\in G$ and points $x_1,\dots, x_k\in X$, the sequence \begin{equation*}
       \Big(g_1^{\floor{a_1(p_n)}}x_1,\dots, g_k^{\floor{a_k(p_n)}}x_k   \Big)_{n\in \N},
  \end{equation*}where $p_n$ denotes the $n$-th prime, is equidistributed on the nilmanifold 
 $\overline{(g_1^{\Z} x_1)}\times\dots \times~ \overline{(g_k^{\Z} x_k)}$.
\end{theorem}

Instead of the ``pointwise convergence'' assumption (b), one can replace it with a weaker convergence (i.e. in the $L^2$-sense) hypothesis. However, we will not benefit from this in applications, so we opt to not state our results in that setup.

 In the case of a polynomial function, a convergence result along primes follows by combining \cite[Theorem 7.1]{Green-Tao-Mobius} (which is the case of linear polynomials)
and the fact that any polynomial orbit on a nilmanifold can be lifted to a linear orbit of a unipotent affine transformation on a larger nilmanifold (an argument due to Leibman \cite{Leibman-nil-polynomial-equidistribution}). Nonetheless, in this case, we do not have a nice description for the orbit of this polynomial sequence.

On the other hand, equidistribution results in higher-step nilmanifolds (along primes) for sequences such as $\floor{n^c}$, with $c$ a non-integer ($c>1$), are unknown even in the simplest case of one fractional power. Theorem \ref{T: criterion for pointwise convergence along primes-nil version} will allow us to obtain the first results in this direction from the corresponding results along $\N$.
Equidistribution results for Hardy sequences along $\N$ were obtained originally by Frantzikinakis in \cite{Fra-equidsitribution}, while more recently new results were established by Richter \cite{Richter} and the second author \cite{tsinas-pointwise}.
In view of the structure theory of Host-Kra \cite{Host-Kra-annals}, results of this nature are essential to demonstrate that the corresponding multiple ergodic averages along $\N$ converge in $L^2(\m)$.
All of the pointwise convergence theorems that we mentioned above can be transferred to the prime setting. As an application, we state the following sample corollary of Theorem~\ref{T: criterion for pointwise convergence along primes-nil version}. The term invariant under affine shifts refers to a Hardy field $\mathcal{H}$ for which $a(Wt+b)\in \mathcal{H}$ whenever $a\in \mathcal{H}$, for all $W,b\in \N$.

\begin{corollary}\label{C: equidistribution nilmanifolds}
    Let $k$ be a positive integer, $\mathcal{H}$ be a Hardy field invariant under affine shifts, and suppose that $a_1,\dots, a_k\in \mathcal{H}$ are functions of polynomial growth, for which there exists an $\e>0$, so that every non-trivial linear combination $a$ of them satisfies \begin{equation}\label{E: t^e away from polynomials}
        \lim\limits_{t\to+\infty} \Bigabs{\frac{a(t)-q(t)}{ t^{\e}} }=+\infty \text{ for every } q(t)\in \Z[t].
    \end{equation}  Then, for any collection of nilmanifolds $X_i=G_i/\G_i$ $i=1,\dots, k$, elements  $g_i\in G_i$ and points $x_i\in X_i$, the sequence \begin{equation*}
        \big( g_1^{\floor{a_1(p_n)}}x_1,\ldots,g_k^{\floor{a_k(p_n)}}x_k    \big)_{n\in \N},
    \end{equation*}where $p_n$ denotes the $n$-th prime, is equidistributed on the nilmanifold $\overline{(g_1^{\Z} x_1)}\times\dots \times~ \overline{(g_k^{\Z} x_k)}$.
\end{corollary}

The assumption in \eqref{E: t^e away from polynomials} is a byproduct of the corresponding equidistribution result along $\N$ proven in \cite{tsinas-pointwise}.
Also, the assumption on $\mathcal{H}$ can be dropped since the arguments in \cite{tsinas-pointwise} rely on some growth assumptions on the functions $a_{i}$ which translate over to their shifted versions. We choose not to remove the assumption here since the results in \cite{tsinas-pointwise} are not stated in this setup.

Our corollary implies that the sequence  \begin{equation*}
        \big( g_1^{\floor{p_n^{c_1}}}x_1,\ldots,g_k^{\floor{p_n^{c_k}}}x_k    \big)
    \end{equation*}is equidistributed on the subnilmanifold $\overline{(g_1^{\Z} x_1)}\times\dots \times~ \overline{(g_k^{\Z} x_k)}$ of $X_1\times\dots\times X_k$, for any distinct positive non-integers $c_1,\dots, c_k$ and for all points $x_i\in X_i$. This is stronger than the result of Frantzikinakis \cite{Fra-primes} that establishes convergence in the $L^2$-sense (for linearly independent fractional polynomials). This result is novel even in the simplest case $k=1$. Furthermore, we remark that in the case $k=1$ we can actually replace \eqref{E: t^e away from polynomials} with the optimal condition that $a(t)-q(t)$ grows faster than $\log t$, for all $q(t)$ that are real multiples of integer polynomials, using the results from \cite{Fra-equidsitribution}.

\subsection{Strategy of the proof and organization }\label{strategysubsection}

The bulk of the paper is spent on establishing the asserted comparison between the $W$-tricked averages and the standard Ces\`{a}ro averages (Theorem \ref{T: the main comparison}). The main trick is to recast our problem to the setting where our averages range over a short interval of the form $[N, N+L(N)]$, where $L(t)$ is a function of sub-linear growth chosen so that Hardy sequences are approximated sufficiently well by polynomials in these intervals.
Naturally, the study of the primes in short intervals requires strong number theoretic input and this is provided by the recent result in \cite{MSTT} on the Gowers uniformity of several arithmetic functions in short intervals (this is Theorem~\ref{T: Gowers uniformity in short intervals} in the following section). The strategy of restricting ergodic averages to short intervals was first used by Frantzikinakis in \cite{Fra-Hardy-singlecase} to demonstrate the convergence of the averages in \eqref{E: Furstenberg averages} when $a(n)$ is a Hardy sequence and then amplified further by the second author in \cite{Tsinas} to resolve the problem in the more general setting of the averages in \eqref{E: general ergodic averages} (for one transformation).  Certainly, the uniformity estimate in Theorem \ref{T: Gowers uniformity in short intervals} requires that the interval is not too short, but it was observed in \cite{Tsinas} that one can take the function $L(t)$ to grow sufficiently fast, as long as one is willing to tolerate polynomial approximations with much larger degrees.

After this step has been completed, one typically employs a complexity reduction argument (commonly referred to as PET induction in the literature) that relies on repeated applications of the van der Corput inequality. Using this approach, one derives iterates that are comprised of several expressions with integer parts, which are then assembled together using known identities for the floor function (with an appropriate error).
This approach was used for the conventional averages over $\N$ in \cite{Fra-Hardy-singlecase} and \cite{Tsinas},
because one can sloppily combine integer parts in the iterates at the cost of inserting a bounded weight in the corresponding averages. To be more precise, this weight is actually the characteristic function of a subset of $\N$. However, we cannot afford to do this blindly in our setting, since there is no guarantee that this subset of $\N$ does not correlate very strongly with $\La_{w,b}(n)-1$, which could imply that the resulting average is large. The fact that the weight $\La_{w,b}-1$ is unbounded complicates this step as well.
Nonetheless, it was observed in  \cite{koutsogiannis-closest}  (using an argument from \cite{Koutsogiannis-correlations}),\footnote{  This argument was first used for $k=1$ in \cite{Boshernitzan-Jones-Wierdl-book} and \cite{Lesigne-single}  to prove that when a sequence of
real positive numbers is good for (single term) pointwise convergence, then its
floor value is also good. The method was later adapted to the $k=2$ setting by Wierdl (personal communication with the first author, 2015).} that if the fractional parts of the sequences in the iterates do not concentrate heavily around 1, then one can pass to an extension of the system $(X,\X,\m, T_1,\dots, T_k)$, wherein the actions $T_i$ are lifted to $\R$-actions (also called measure-preserving flows) and the integer parts are removed. Since there are no nuisances with combining rounding functions in the iterates, one can then run the complexity reduction argument in the new system and obtain the desired bounds.

Unfortunately, there is still an obstruction in this approach arising from the fact that the flows in the extension are not continuous. To be more precise, let us assume that we derived an approximation of the form $a(n)=p_N(n)+\e_N(n)$, where $n\in [N, N+L(N)]$, $p_N(n)$ is a Taylor polynomial and $\e_N(n)$ is the remainder term. The PET induction can eliminate the polynomials $p_N(n)$, by virtue of the simple observation that taking sufficiently many ``discrete derivatives'' makes a polynomial vanish.
However, this procedure cannot eliminate the error term $\e_N(n)$ at all and the fact that the flow is not continuous prohibits us from replacing them with zero. Thus, we take action to discard $\e_N(n)$ beforehand. This is done by studying the equidistribution properties of the polynomial $p_N(n)$ in the prior approximation, using standard results from the equidistribution theory of finite polynomial orbits due to Weyl. Practically, we show that for ``almost all'' values of $n$ in the interval $[N, N+L(N)]$, we can write $\floor{p_N(n)+\e_N(n)}=\floor{p_N(n)}$, so that the error $\e_N(n)$ can be removed from the expressions in the iterates.

In our approach, some equidistribution assumptions on our original functions are required. This clarifies the conditions on Theorem \ref{T: the main comparison}. Indeed, \eqref{E: far away from rational polynomials} implies that the sequence $\big(a_{ij}(n)\big)_{n}$ is equidistributed modulo 1 (due to Theorem \ref{T: Boshernitzan}), while condition \eqref{E: essentially equal to a polynomial} implies that the function $a_{ij}(t)$ is essentially equal to a polynomial with rational coefficients (thus periodic modulo 1).

\subsubsection{A simple example}
We demonstrate the methods discussed above in a basic case that avoids most complications that appear in the general setting. 
Even this simple case, however, is not covered by prior methods in the literature. We will use some prerequisites from the following section, such as Theorem \ref{T: Gowers uniformity in short intervals}.  

We consider the averages
\begin{equation}\label{E: three Aps example primes}
\frac{1}{\pi(N)} \sum_{p\in\P:\; p\leq N}T^{\floor{p^{3/2}}}f_1\cdot T^{2\floor{p^{3/2}}}f_2,
\end{equation}where $(X,\mathcal{X},\mu,T)$ is a system and $f_1, f_2\in L^\infty(\mu)$.
For every $1\leq b\leq W$ with $(b,W)=1$, we study the averages
\begin{equation}\label{E: three Aps example}
    \frac{1}{N}\sum_{n=1}^N \big(\La_{w,b}(n)-1\big)  T^{\floor{n^{3/2}}}f_1\cdot T^{2\floor{n^{3/2}}}f_2,
\end{equation}which is the required comparison for the averages in \eqref{E: three Aps example primes}. We will show that as $N\to+\infty$ and then $w\to+\infty$, the norm of this average converges to 0 uniformly in $b$.

We set $L(t)=t^{0.65}$. Notice that $L(t)$ grows faster than $t^{5/8}$, which is a necessary condition to use Theorem~\ref{T: Gowers uniformity in short intervals}. In order to establish the required convergence
for the averages in \eqref{E: three Aps example}, it suffices to show that
\begin{equation}\label{E: second equation in first example}
\limsup\limits_{r\to+\infty} \Bignorm{\E_{r\leq n\leq r+L(r)} \big(\La_{w,b}(n)-1\big) T^{\floor{n^{3/2}}}f_1 \cdot T^{2\floor{n^{3/2}}} f_2   }_{L^2(\m)}=o_w(1)
\end{equation}uniformly in $b$.
We remark that in the more general case that encompasses several functions, we will need to average over the parameter $r$ as well and, thus, we are dealing with a double-averaging scheme.
This reduction is the content of Lemma \ref{L: long averages to short averages}.

Using the Taylor expansion around $r$, we can write for every $0\leq h\leq L(r)$:\begin{equation*}
   (r+h)^{3/2}=r^{3/2}+\frac{3r^{1/2}h}{2}+\frac{3h^2}{8r^{1/2}} -\frac{3h^3}{48\xi^{3/2}_{h}}, \ \text{ where } \xi_{h}\in [r,r+h].
\end{equation*}
Observe that the error term is smaller than a constant multiple of  $$\frac{\big(L(r)\big)^{3}}{r^{3/2}}=o_r(1).$$
We show that we have \begin{equation*}
\floor{(r+h)^{3/2}}=\floor{r^{3/2}+\frac{3r^{1/2}h}{2}+\frac{3h^2}{8r^{1/2}}}
\end{equation*}for ``almost all'' $0\leq h \leq L(r)$, in the sense that the number of $h$'s that do not obey this relation is
bounded by a constant multiple of $L(r)\log^{-100} r$ (say). Thus, their contribution on the average is negligible, since the sequence $\La_{w,b}$ has size comparable to $\log r$. 

Let us denote by $p_r(h)$ the quadratic polynomial in the Taylor expansion above.
In order to establish this assertion, we will investigate the discrepancy of the finite sequence $\big(p_r(h)\big)_{0\leq h \leq L(r)}$, using some exponential sum estimates and the Erd\H{o}s-Tur\'{a}n inequality (Theorem \ref{T: Erdos-Turan}). This is the content of Proposition \ref{P: remove error term for fast functions}.

Assuming that all the previous steps were completed, we shall ultimately reduce our problem to showing that  \begin{equation*}
    \limsup\limits_{r\to+\infty} \Bignorm{\E_{0\leq h\leq L(r)} \big(\La_{w,b}(r+h)-1\big) T^{\floor{p_r(h)}}f_1 \cdot T^{2\floor{p_r(h)}}  f_2  }_{L^2(\m)}=o_w(1)
\end{equation*}uniformly for $1\leq b\leq W$ coprime to $W$.
Now that the error terms have been eliminated, we are left with an average that involves polynomial iterates. Next, we use an argument from \cite{Koutsogiannis-correlations} that allows us to pass to an extension of the system $(X,\X,\m, T)$. 
To be more precise, there exists an $\R$-action (see the definition in Section \ref{Section: Background}) $(Y,\mathcal{Y},v,S)$ and functions $\wt{f}_1,\wt{f}_2$, such that we have the equality \begin{equation*}
    T^{\floor{p_r(h)}}f_1 \cdot T^{2\floor{p_r(h)}}  f_2= S_{p_r(h)}\wt{f}_1\cdot S_{2p_r(h)}\wt{f}_2.
\end{equation*}This procedure can be done because the polynomial $p_r(h)$ has good equidistribution properties (which we analyze in the previous step) and thus the fractional parts of the finite sequence $\big(p_r(h)\big)_{0\leq h\leq L(r)}$ fall inside a small interval around 1 with the correct frequency. This is a necessary condition in order to use Proposition \ref{P: Gowers norm bound on variable polynomials}, which provides a bound for the inner average. To be more specific, we have the expression \begin{equation*}
      \limsup\limits_{r\to+\infty} \Bignorm{\E_{0\leq h\leq L(r)} \big(\La_{w,b}(r+h)-1\big) S_{p_r(h)}\wt{f}_1\cdot S_{2p_r(h)}\wt{f}_2 }_{L^2(\m)}.
\end{equation*}
The inner  average involves polynomials and can be bounded uniformly by  the Gowers norm of the sequence $\La_{w,b}(n)-1$ by Proposition \ref{P: Gowers norm bound on variable polynomials} (modulo some constants and error terms that we ignore for the sake of this discussion).
In particular, we have that the average in \eqref{E: second equation in first example} is bounded by   \begin{equation*}
     \bignorm{\La_{w,b}(n)-1}_{U^s(r,r+L(r)] }
\end{equation*}for some $s\in \N$. Finally, Theorem \ref{T: Gowers uniformity in short intervals} implies that for sufficiently large values of $r$ we have $\bignorm{\La_{w,b}(n)-1}_{U^s(r,r+L(r)] }$ is $o_w(1)$ uniformly in $b$. Finally, sending $w$ to $+\infty$, we reach the desired conclusion.

This argument is quite simpler than the general case since it involves only one function. As we commented briefly, one extra complication is that we are dealing with a double averaging, unlike the model example.
During the proof of Theorem \ref{T: the main comparison}, we will also need to split the functions $a_{ij}$ into several distinct classes, which are handled with different methods. For example, the argument above works nicely for the function $t^{3/2}$ but has to be modified in the case of the function $\log^2 t$, because the latter cannot be approximated by polynomials of degree 1 or higher on our short intervals. 
Namely, the Taylor polynomial corresponding to $\log^2 t$ is constant and the previous method is orendered ineffective. 
Thus, we present an additional, more elaborate model example in Section \ref{Section-Removing error terms section}, which exemplifies the possible cases that arise in the main proof.

\subsection{Open problems and further directions}\label{SUBSection-Open problems}

We expect that condition \eqref{E: far away from real multiples of integer polynomials} in Theorem \ref{T: multiple recurrence for szemeredi type patterns} can be relaxed significantly and still provide a multiple recurrence result. Motivated by \cite[Theorem 2.3]{Fra-Hardy-singlecase}, we make the following conjecture.\begin{conjecture}\label{Conjecture 1}
    Let $a\in \mathcal{H}$ be a function of polynomial growth which satisfies \begin{equation*}
        \lim\limits_{t\to+\infty} \bigabs{a(t)-cp(t)}=+\infty \text{ for every } c\in \R  \text{ and }  p(t)\in \Z[t].
    \end{equation*} Then, for any $k\in\N,$ measure-preserving system $(X,\X,\m,T)$ and set $A$ of positive measure, the set 
$$\{n\in \N:\; \m\big(A\cap T^{-\floor{a(n)}} A \cap \dots \cap T^{-k\floor{a(n)}}A  \big)>0\}$$
has non-empty intersection with $\P$.
\end{conjecture}

Comparing the assumptions on the function $a$ to those in Theorem \ref{T: multiple recurrence for szemeredi type patterns}, we see that we are very close to establishing Conjecture \ref{Conjecture 1}. However, there are examples that our work does not encompass, such as the function $ t^4 +\log t$ or $t^2+\log\log (5t)$.
In the setting of multiple recurrence along $\N$, the corresponding result was established in \cite{Fra-Hardy-singlecase} and was generalized for more functions in \cite{BMR-Hardy2}. 
In view of \cite[Corollary B.3, Corollary B.4]{BMR-Hardy2}, we also make the following conjecture:

\begin{conjecture}
    Let $k\in\N$ and $a_1,\dots, a_k\in \mathcal{H}$ be functions of polynomial growth. Assume that 
     every non-trivial linear combination $a$ of the functions $a_1,\dots,a_k$, $a,$ has the property \begin{equation*}
        \lim_{t\to +\infty}|a(t)-p(t)|=+\infty \text{ for all } p(t)\in \Z[t].
    \end{equation*}
   Then, for any measure-preserving system $(X,\X,\m,T)$ and set $A$ of positive measure,
the set 
$$\{n\in \N:\; \m\big(A\cap T^{-\floor{a_1(n)}} A \cap \dots \cap T^{-\floor{a_k(n)}}A  \big)>0\}$$
has non-empty intersection with $\P$.
\end{conjecture}

We remark that if one wants to also include functions that are essentially equal to a polynomial, then there are more results in this direction in \cite{BMR-Hardy2}, where it was shown that a multiple recurrence result for functions that are approximately equal to jointly-intersective polynomials is valid. Certainly, one would need to work with the sets $\P+1$ or $\P-1$ in this setting to transfer this result from $\N$ to the primes.

It is known that a convergence result along $\N
$ with typical Ces\`{a}ro averages cannot be obtained, if one works with the weaker conditions of the previous two conjectures.
Indeed, the result would fail even for rotations on tori, because the corresponding equidistribution statement is false.
The main approach employed in \cite{BMR-Hardy2} was to consider a weaker averaging scheme than  Ces\`{a}ro averages. Using a different averaging method, one can impose some equidistribution assumption on functions that are not equidistributed in the standard sense. For instance, it is well-known that the sequence $(\log n)_{n\in \N}$ is not equidistributed mod 1 using Ces\`{a}ro averages, but it is equidistributed under logarithmic averaging. Thus, it is natural to expect that an analog of Theorem \ref{T: the main comparison} for other averaging schemes would allow someone to relax the conditions \eqref{E: far away from rational polynomials} and \eqref{E: essentially equal to a polynomial} in order to tackle the previous conjectures. A comparison result similar to Theorem \ref{T: the main comparison} (but for other averaging schemes) appears to be a potential first step in this problem.

We expect that, under the same hypotheses, the analogous result in the setting of multiple commuting transformations will also hold. In particular, aside from the special cases established in \cite{Fra-Hardy-multidimensional}, convergence results along $\N$ for Hardy sequences and commuting transformations are still open. For instance, it is unknown whether the averages in Theorem \ref{T: nikos result to primes} converge when the functions $a_i$ are linear combinations of fractional powers.
In view of Theorem \ref{T: the main comparison} and Theorem \ref{T: criterion for convergence along primes}, any new result in this direction can be transferred to the setting of primes in a rather straightforward fashion, since conditions \eqref{E: far away from rational polynomials} and \eqref{E: essentially equal to a polynomial} are quite general to work with. 

\subsection{Acknowledgements}
 We thank Nikos Frantzikinakis for helpful discussions.

\subsection{Notational conventions} Throughout this article, we denote with $\mathbb{N}=\{1,2,\ldots\},$ $\mathbb{Z},$ $\mathbb{Q},$ $\mathbb{R},$ and $\mathbb{C}$ the sets of natural, integer, rational, real, and complex numbers respectively. We denote the one dimensional torus $\T=\R/\Z$, the exponential phases $e(t)=e^{2\pi it}$, while $\norm{x}_{\T}=d(x,\Z),$ $[[x]],$ $\floor{x},$ $\lceil x\rceil,$ and $\{x\}$ are the distance of $x$ from the nearest integer, the nearest integer to $x,$ the greatest integer which is less or equal to $x,$ the smallest integer which is greater or equal to $x,$ and the fractional part of $x$ respectively.  We also let ${\bf 1}_A$ denote the characteristic function of a set $A$ and $|A|$ is its cardinality.

For any integer $Q$ and $0\leq a\leq Q-1$, we use the symbol $a\; (Q)$ to denote the residue class $a$ modulo $Q$. Therefore, the notation ${\bf 1}_{a\; (Q)}$ refers to the characteristic function of the set of those integers, whose residue when divided by $Q$ is equal to $a$.

For two sequences $a_n,b_n$, we say that $b_n$ {\em dominates} $a_n$ and write $a_n\prec b_n$ or $a_n=o(b_n)$, when $a_n/b_n$ goes to 0, as $n\to+\infty$. In addition, we write $a_n\ll b_n$ or $a_n=O(b_n)$, if there exists a positive constant $C$ such that $|a_n|\leq C|b_n|$ for large enough $n$. When we want to denote the dependence of the constant $C$ on some parameters $h_1,\dots,h_k$, we will use the notation $a_n=O_{h_1,\dots,h_k}(b_n)$. In the case that $b_n\ll a_n\ll b_n$, we shall write $a_n\sim b_n$. We say that $a_n$ and $b_n$ have the same growth rate when the limit of $\frac{a_n}{b_n}$, as $n\to+\infty$ exists and is a non-zero real number.
We use a similar notation and terminology for asymptotic relations when comparing functions of a real variable $t$.

Under the same setup as in the previous paragraph,
we say that the sequence $a_n$ {\em strongly dominates} the sequence $b_n$ if there exists $\delta>0$ such that \begin{equation*}
    \frac{a_n}{b_n}\gg n^{\delta}.
\end{equation*}In this case, we write $b_N \lll a_N$, or $a_N\ggg b_N$.\footnote{This notation is non-standard, so we may refer back to this part quite often throughout the text.} We use similar terminology and notation for functions on a real variable $t$.

Finally, for any sequence $(a(n))$, we 
employ the notation \begin{equation*}
    \E_{n\in S} a(n)=\frac{1}{|S|} \sum_{ n\in S}^{} a(n)
\end{equation*}
to denote averages over a finite non-empty set $S$. We will typically work with averages over the integers in a specified interval, whose endpoints will generally be non-integers. We shall avoid using this notation for the Ces\`{a}ro averages.

\section{Background}\label{Section: Background}

\subsection{Measure-preserving actions}

Let $(X,\mathcal{X},\mu)$ be a Lebesgue probability space. A transformation $T:X\to X$ is {\em measure-preserving} if $\m(T^{-1}(A))=\m(A)$ for all $A\in \X$. It is called {\em ergodic} if all the $T$-invariant functions are constant. If $T$ is invertible, then $T$ induces a $\Z$-action on $X$ by $(n,x)=T^n x$, for every $n\in \Z$ and $x\in X$.

More generally, let $G$ be a group.
A {\em measure-preserving $G$-action} on a Lebesgue probability space $(X,\X,\m)$ is an action on $X$ by measure-preserving maps $T_g$ for every $g\in G$ such that, for all $g_1,g_2\in G$, we have $T_{g_1g_2}=T_{g_1}\circ T_{g_2}$. For the purposes of this article, we will only need to consider actions by the additive groups of $\Z$ or $\R$. Throughout the following sections, we will also refer to $\R$-actions as {\em measure-preserving flows}. In the case of $\Z$-actions, we follow the usual notation and write $T^n$ to indicate the map $T_n$.

\subsection{Hardy fields}

Let $(\mathcal{B},+,\cdot)$ denote the ring of germs at infinity of real-valued functions defined on a half-line $(t_0,+\infty)$. A sub-field $\mathcal{H}$ of $\mathcal{B}$ that is closed under differentiation is called a {\em Hardy field}.
 For any two functions $f,g\in\mathcal{H}$, with $g$ not identically zero, the limit \begin{equation*}
    \lim\limits_{t\to+\infty} \frac{f(t)}{g(t)}
\end{equation*} exists in the extended line and thus we can always compare the growth rates of two functions in $\mathcal{H}$. In addition, every non-constant function in $\mathcal{H}$ is eventually monotone and has a constant sign eventually. 
We define below some notions that will be used repeatedly throughout the remainder of the paper.
\begin{definition}\label{D: growthdefinitions}
Let $a$ be a function in $\mathcal{H}$.
We say that the function $a$ has {\em polynomial growth} if there exists a positive integer $k$ such that $a(t)\ll t^k$. The smallest positive integer $k$ for which this holds will be called the {\em degree} of $a$. The function $a$ is called {\em sub-linear} if $a(t)\prec t$. It will be called {\em sub-fractional} if $a(t)\prec t^{\e}$, for all $\e>0$. Finally, we will say that $a$ is {\em strongly non-polynomial} if, for all positive integers $k$, we have that the functions $a(t)$ and $t^k$ have distinct growth rates. 
\end{definition}

Throughout the proofs in the following sections, we will assume that we have fixed a Hardy field $\mathcal{H}$. Some of the theorems impose certain additional assumptions on $\mathcal{H}$, but this is a byproduct of the arguments used to establish the case of convergence of Ces\`{a}ro averages in \cite{Tsinas} and we will not need to use these hypotheses in any of our arguments.

\subsection{Gowers uniformity norms on intervals of integers}

Let $N$ be a positive integer and let $f:\Z_N\to \C$ be a function. For any positive integer $s$, we define the {\em Gowers uniformity norm} $\norm{f}_{U^s(\Z_N)}$ inductively by \begin{equation*}
    \bignorm{f}_{U^1(\Z_N)}=\bigabs{\E_{n\in \Z_N}  \ f(n)}
\end{equation*}and for $s\geq 2$,\begin{equation*}
\bignorm{f}_{U^{s}(\Z_N)}^{2^s}=\E_{h\in \Z_N} \bignorm{\overline{f(\cdot )}f(\cdot+h)}_{U^{s-1}(\Z_N)}^{2^{s-1}}.
\end{equation*} A straightforward computation implies that \begin{equation*}
    \bignorm{f}_{U^{s}(\Z_N)}=\Big( \E_{\uh\in \Z_N^s}\E_{n\in \Z_N} \prod_{\ue\in \{0,1\}^s}\mathcal{C}^{|\ue|} f(n+\uh\cdot \ue)\Big)^{\frac{1}{2^s}}.
\end{equation*} Here, the notation $\mathcal{C}$ denotes the conjugation map in $\C$, whereas for $\ue\in \{0,1\}^s$, $|\ue|$ is the sum of the entries of $\ue$ (the number of coordinates equal to $1$).

It can be shown that for $s\geq 2$, $\norm{\cdot}_{U^s(\Z_N)}$ is a norm and that $$\norm{f}_{U^s(\Z_N)}\leq \norm{f}_{U^{s+1}(\Z_N)}$$ 
for any function $f$ on $\Z_N$ \cite[Chapter 6]{Host-Kra-structures}. For the purposes of this article, it will be convenient to consider similar expressions that are not necessarily defined only for functions in an abelian group $\Z_N$. Therefore, for any $s\geq 1$ and a finitely supported sequence $f(n), n\in \Z$, we define the {\em unnormalized Gowers uniformity norm} \begin{equation}\label{E: unnormalized gowers norm}
    \bignorm{f}_{U^s(\Z)}=\Big( \sum_{\uh\in \Z^s}\ \sum_{n\in \Z} \ \prod_{\ue\in \{0,1\}^s}\mathcal{C}^{|\ue|} f(n+\uh\cdot \ue)\Big)^{\frac{1}{2^s}}
\end{equation}and for a bounded interval $I\subset{\R}$, we define \begin{equation}\label{Definition: Gowers norms along intervals}
    \bignorm{f}_{U^s(I)} =\frac{ \bignorm{f\cdot {\bf 1}_{I}  }_{U^s(\Z)}}{\bignorm{{\bf 1}_{I}}_{U^s(\Z) }}.
\end{equation}

First of all, observe that a simple change of variables in the summation in \eqref{Definition: Gowers norms along intervals} implies that for $X\in \Z$ \begin{equation*}
\bignorm{f}_{U^s(X,X+H]}=\bignorm{f(\cdot+X)}_{U^s[1,H]}.
\end{equation*} Evidently, we want to compare uniformity norms on the interval $[1, H]$ with the corresponding norms on the abelian group $\Z_H$. To this end, we will use the following lemma, whose proof can be found in \cite[Chapter 22, Proposition 11]{Host-Kra-structures}.
\begin{lemma}\label{L: relations for gowers norms defined on intervals and on groups}
    Let $s$ be a positive integer and $N,N'\in \N$ with $N'\geq 2N$. Then, for any sequence $\big(f(n)\big)_{ n\in \Z}$, we have \begin{equation*}
        \bignorm{f}_{U^s[1,N]  } =\frac{\bignorm{f\cdot 1_{[1,N]}}_{U^s(\Z_{N'})} }{\bignorm{1_{[1,N] }}_{U^s(\Z_{N'})}}.
    \end{equation*}
\end{lemma}

We will need a final lemma that implies that the Gowers uniformity norm is smaller when the sequence is evaluated along arithmetic progressions.
\begin{lemma}\label{L: Gowers uniformity norms evaluated at arithmetic progressions}
    Let $u(n)$ be a sequence of complex numbers. Then, for any integer $s\geq 2$ and any positive integers $0\leq a\leq Q-1$, we have \begin{equation*}
        \bignorm{u(n){\bf 1}_{a\;( Q)}(n)}_{U^s(X,X+H]}\leq \bignorm{u(n)}_{U^s(X,X+H]},
    \end{equation*}for all integers $X\geq 0$ and all $H\geq 1$.
\end{lemma}
\begin{proof}
    We set $u_X(n)=u(X+n)$, so that we can rewrite the norm on the left-hand side as $\bignorm{u_X(n){\bf 1}_{a\;(Q)}(X+n)}_{U^s[1,H]}$. Observe that the function ${\bf 1}_{a\; (Q)}(n)$ is periodic modulo $Q$. Thus, treating it as a function in $\Z_Q$, we have the Fourier expansion \begin{equation*}
        {\bf 1}_{a\; (Q)}(n)=\sum_{\xi\in \Z_q} \widehat{{\bf 1}}_{a\; (Q)}(\xi)e\Big(\frac{n\xi}{Q}\Big),  
    \end{equation*}for every $0\leq n\leq Q-1$, and this can be extended to hold for all $n\in \Z$ due to periodicity.
    Furthermore, we have the bound \begin{equation*}
       \bigabs{  \widehat{{\bf 1}}_{a\; (Q)}(\xi)}=\frac{1}{Q}\Bigabs{e\bigg(\frac{a\xi}{Q}\bigg)}\leq \frac{1}{Q}.
    \end{equation*}Applying the triangle inequality, we deduce that \begin{equation*}
        \bignorm{u_X(n){\bf 1}_{a\;(Q)}(X+n)}_{U^s[1,H]}\leq \sum_{\xi\in \Z_Q}  \bigabs{  \widehat{{\bf 1}}_{a\; (Q)}(\xi)}\cdot \Bignorm{u_X(n)e\Big(\frac{(X+n)\xi}{Q}\Big)}_{U^s[1,H]}.
    \end{equation*}However, it is immediate from \eqref{E: unnormalized gowers norm} that the $U^s$-norm is invariant under multiplication by linear phases, for every $s\geq 2$. Therefore, we conclude that \begin{equation*}
          \bignorm{u_X(n){\bf 1}_{a\;(Q)}(X+n)}_{U^s[1,H]}\leq   \bignorm{u_X(n)}_{U^s[1,H]}= \bignorm{u(n)}_{U^s(X,X+H]},
    \end{equation*}which is the desired result.
\end{proof}

The primary utility of the Gowers uniformity norms is the fact that they arise naturally in complexity reduction arguments that involve multiple ergodic averages with polynomial iterates. In particular, Proposition \ref{P: PROPOSITION THAT WILL BE COPYPASTED FROM FRA-HO-KRA} below implies that polynomial ergodic averages weighted by a sequence $(a(n))_{n\in \N}$  can be bounded in terms of the Gowers norm of $a$ on the abelian group $\Z_{sN}$ for some positive integer $s$ (that depends only on the degrees of the underlying polynomials).

\begin{proposition}
    \cite[Lemma 3.5]{Fra-Host-Kra-primes}\label{P: PROPOSITION THAT WILL BE COPYPASTED FROM FRA-HO-KRA}
Let $k, \ell\in \N,$ $(X,\X,\mu,T_1,\ldots, T_k)$ be a system of commuting $\Z$ actions, $p_{i,j}\in\Z[t]$ be polynomials for every $1\leq i\leq k,$ $1\leq j\leq \ell,$  $f_1,\ldots,f_{\ell}\in L^{\infty}(\mu)$ and $a:\N\to\C$ be a sequence. Then, there exists $s\in\N,$ depending only on the maximum degree of the polynomials $p_{i,j}$ and the integers $k,\ell$, and a constant $C_s$ depending on $s,$ such that 
\begin{equation}\label{E: bound of polynomial ergodic averages in terms of Gowers norm of the weight}
  \Bignorm{\E_{1\leq n\leq N} a(n)\cdot  \prod_{j=1}^{\ell} \prod_{i=1}^k T_i^{p_{i,j}(n)} f_j
}_{L^2(\mu)}\leq C_{s} \left( \norm{a\cdot {\bf 1}_{[1,N]} }_{U^s(\Z_{sN})}+\frac{\max \{1,\norm{a}^{2s}_{\ell^{\infty}[1,sN]}\}}{N} \right).   
\end{equation}
\end{proposition}
    \begin{remark}
       $(i)$ The statement presented in \cite{Fra-Host-Kra-primes} asserts that the second term in the prior sum is just $o_N(1)$, under the assumption that $a(n)\ll n^c$ for all $c>0$. However, a simple inspection of the proof gives the error term presented above. Indeed, the error terms appearing in the proof of Proposition \ref{P: PROPOSITION THAT WILL BE COPYPASTED FROM FRA-HO-KRA} are precisely of the form \begin{equation*}
            \frac{1}{N}\E_{n\in [1,N]}\ \E_{ \uh\in [1,N]^k} \Bigabs{\ \prod_{\ue\in \{0,1\}^k}  \mathcal{C}^{|\ue|}\ a(n+\uh\cdot \ue)}
        \end{equation*}for $k\leq s-1$, which are the error terms in the van der Corput inequality. Deducing the error term on \eqref{E: bound of polynomial ergodic averages in terms of Gowers norm of the weight} is then straightforward. \\
        $(ii)$ The number $s-1$ is equal to the number of applications of the van der Corput inequality in the associated PET argument and we may always assume that $s\geq 2$. In that case, Lemma~\ref{L: relations for gowers norms defined on intervals and on groups} and the bound $\norm{{\bf 1}_{[1,N]} }_{U^s(\Z_{sN})}\leq 1$ implies that we can replace the norm in \eqref{E: bound of polynomial ergodic averages in terms of Gowers norm of the weight} with the term $\norm{a}_{U^s[1,N]}$.
    \end{remark}

For polynomials $p_{i,j}(t)\in\mathbb{R}[t]$ of the form \begin{equation*}
        p_{i,j}(t)=a_{ij,d_{ij}}t^{d_{ij}}+\dots+a_{ij,1}t+a_{ij,0},
    \end{equation*} and $(T_{i,s})_{s\in \mathbb{R}}$ $\mathbb{R}$-actions, we have \begin{equation*}
T_{i,p_{i,j}(n)}=\Big(T_{i,a_{ij,d_{ij}}}\Big)^{n^{d_{ij}}}\cdot\ldots \cdot  \Big(T_{i,a_{ij,1}}\Big)^{n}\cdot \Big(T_{i,a_{ij,0}}\Big).
    \end{equation*} Thus, Proposition \ref{P: PROPOSITION THAT WILL BE COPYPASTED FROM FRA-HO-KRA} implies the following.

\begin{corollary}\label{C: Gowers norm bound adapted to flows}
    Let $k, \ell\in \N,$ $(X,\mathcal{X},\mu,S_1,\ldots, S_k)$ be a system of commuting $\R$-actions, $p_{i,j}\in\Z[t]$ be polynomials for all $1\leq i\leq k,$ $1\leq j\leq \ell,$  $f_1,\ldots,f_{\ell}\in L^{\infty}(\mu)$  and $a:\N\to\C$ be a sequence. Then, there exists $s\in\N,$ depending only on the maximum degree of the polynomials $p_{i,j}$ and the integers $k,\ell$ and a constant $C_s$ depending on $s,$ such that 
\begin{equation}
  \Bignorm{\E_{1\leq n\leq N} a(n)\cdot  \prod_{j=1}^{\ell} \prod_{i=1}^k S_{i,p_{i,j}(n)} f_j
}_{L^2(\mu)}\leq C_{s} \left( \norm{a\cdot {\bf 1}_{[1,N]} }_{U^s(\Z_{sN})}+\frac{\max \{1,\norm{a}^{2s}_{\ell^{\infty}[1,sN]}\}}{N} \right).  
\end{equation}
\end{corollary}

\subsection{Number theoretic tools}

 The following lemma is a standard consequence of the prime number theorem and the sparseness of prime powers (actually, we use this argument in the proof of Corollary \ref{C: Brun-Titchmarsh inequality for von Mangoldt sums} below).
For a proof, see, for instance, \cite[Chapter 25]{Host-Kra-structures}.
\begin{lemma}\label{L: indicator of primes to von-Mangoldt}
    For any bounded sequence $(a(n))_{n\in \N}$ in a normed space, we have \begin{equation}
       \lim\limits_{N\to+\infty} \Bignorm{\frac{1}{\pi(N)}\sum_{p\in \P\colon p\leq N} a(p) -\frac{1}{N}\sum_{n=1}^{N} \La(n) a(n)}=0.
    \end{equation}
\end{lemma}

Therefore, in order to study ergodic averages along primes, we can replace them with the ergodic averages over $\N$ weighted by the function $\La(n)$.

For the modified von Mangoldt function, we have the following deep theorem, which was recently established in \cite{MSTT}.
\begin{customthm}{A}\cite[Theorem 1.5]{MSTT}\label{T: Gowers uniformity in short intervals} Let $\e>0$ and assume $L(N)$  is a positive sequence that satisfies the bounds $N^{\frac{5}{8}+\e}\leq L(N)\leq N^{1-\e} $. Let $s$ be a fixed integer and let $w$ be a positive integer. Then, if $N$ is large enough in terms of $w$, we have that 
\begin{equation}
    \label{E: W-trick}
    \norm{\La_{w,b}-1}_{U^s(N,N+L(N)]}=o_w(1)
\end{equation}for every $1\leq b\leq W$ with $(b,W )=1$.
\end{customthm}

We will need to use the orthogonality of $\La_{w,b}$ to polynomial phases in short intervals.
This is an immediate consequence of the $U^d$ uniformity in Theorem \ref{T: Gowers uniformity in short intervals} in conjunction with an application of  the van der Corput inequality $d$ times until the polynomial phase is eliminated. Alternatively, one can use Proposition \ref{P: PROPOSITION THAT WILL BE COPYPASTED FROM FRA-HO-KRA} for a rotation on the torus $\T$ to carry out the reduction to Theorem \ref{T: Gowers uniformity in short intervals}.\footnote{Evidently, both statements rely on similar complexity reduction arguments, though Proposition \ref{P: PROPOSITION THAT WILL BE COPYPASTED FROM FRA-HO-KRA} is stated in much larger generality involving numerous polynomials.}  We omit its proof. 
\begin{lemma}\label{L: discorrelation of W-tricked with polynomial phases}
    Let $L(N)$ be a positive sequence satisfying $N^{\frac{5}{8}+\e}\prec L(N)\prec N^{1-\e}$ for some $\e>0$. Then, we have that \begin{equation}\label{E: orthogonality to polynomial phases}
   \max_{\underset{(b,W)=1}{1\leq b\leq W}}   \sup_{\underset{\deg p=d}{p\in \R[t]}} \Bigabs{\E_{N\leq n\leq N+L(N)}\big(\La_{w,b}(n)-1\big)e(p(n))   }=o_w(1).
    \end{equation}for every $N $ large enough in terms of $w$.
\end{lemma}
\begin{remark}
    $(i)$ The error term $o_w(1)$ depends on the degree $d$, but since this will be fixed in applications, we suppressed that dependence above.\\
    $(ii)$ Quantitative bounds for similar expressions (involving the more general class of nilsequences, as well) were the main focus in \cite{MSTT}, though in that setting the authors used a different weight of the form $\La-\La^{\#}$, where $\La^{\#}$ is a carefully chosen approximant for the von Mangoldt function arising from considerations of the (modified) Cramer random model for the primes.
\end{remark}

Finally, we will also use a corollary of the Brun-Titchmarsh inequality to bound
the contribution of bad residue classes in our ergodic averages by a constant term. For $q\geq 2$ and $(a,q)=1$, we denote by $\pi(x,q,a)$ the number of primes $\leq x$ that are congruent to $a$ modulo $q$. Alternatively, one could also use the asymptotics for averages of $\La$ in short intervals that were established by Huxley \cite{Huxley}, since $L(N)$ will be chosen to grow sufficiently fast in our applications.

\begin{customthm}{B}[Brun-Titchmarsh inequality]\label{T: Brun-Titchmarsh inequality}
   We have \begin{equation}\label{E: Brun-Titchmarsh}
        \pi(x+y,q,a)-\pi(x,q,a)\leq \frac{2y}{\phi(q)\log(\frac{y}{q})}
    \end{equation}for every $x\geq y>q$. 
\end{customthm}
While we referred to this as the Brun-Titchmarsh inequality, 
the previous theorem was established in \cite{Montgomery-Vaughan} by Montgomery and Vaughan (prior results contained the term $2+o(1)$ in the numerator).
We will need a variant of this theorem adapted to the von Mangoldt function. This follows easily from the previous theorem and a standard partial summation argument.
\begin{corollary}\label{C: Brun-Titchmarsh inequality for von Mangoldt sums}
    For every $q\leq y\leq x$, we have \begin{equation*}
        \sum_{\underset{n\equiv a\;(q)}{x\leq n\leq x+y}}\La(n)\leq \frac{2y\log x}{\phi(q)\log(\frac{y}{q})}+O\big(\frac{y}{\log x}\big)+O\big(x^{\frac{1}{2}}\log x\big).
    \end{equation*}
\end{corollary}
\begin{proof}
    Consider the function $$\pi(x,q,a)=\sum_{\underset{n\equiv a\;(Q)}{1\leq n\leq x} } 1_{\P}(n)$$ as in the statement of Theorem \ref{T: Brun-Titchmarsh inequality}, defined for all $x\geq 3/2$.
    Let $$\theta(x,q,a)=\sum_{\underset{n\equiv a\;( Q)}{1\leq n\leq x} } 1_{\P}(n)\log n,\ \ \psi(x,q,a)=\sum_{\underset{n\equiv a\;( Q)}{1\leq n\leq x} } \La(n).$$

    It is evident that \begin{equation}\label{E: prime powers suck}
        \Bigabs{\theta(x,q,a)-\psi(x,q,a)}\leq \sum_{p^k\leq x\colon p\in \P, k\geq 2} \log p\leq x^{1/2}\log x,
    \end{equation}since there are at most $x^{1/2}$ prime powers $\leq x$ and each one of them contributes at most $\log x$ in this sum.
 Now, we use summation by parts to deduce that \begin{multline*}
        \theta(x+y,q,a)-\theta(x,q,a)=\sum_{\underset{n\equiv a\;( Q)}{x<  n\leq x+y} } 1_{\P}(n)\log n+O(1)=\pi(x+y,q,a)\log(x+y)-\\
        \pi(x,q,a)\log(x+1)+\sum_{\underset{n\equiv a\;( Q)}{x<  n\leq x+y} } \pi(n,q,a)\Big(\log n -\log(n+1)\Big)+O(1).
    \end{multline*}Using the inequalities $\log n-\log(n+1)\leq- (n+1)^{-1}$ and $\log(x+y)\leq \log x +y/x$, we deduce that \begin{multline*}
          \theta(x+y,q,a)-\theta(x,q,a)\leq \log x\Big(   \pi(x+y,q,a)-\pi(x,q,a)\Big)+\frac{\pi(x+y,q,a)y}{x}-\\
          \sum_{\underset{n\equiv a\;( Q)}{x<  n\leq x+y} } \frac{\pi(n,q,a)}{n+1}+O(1).
    \end{multline*}Using the estimate $\pi(x,q,a)\ll \frac{x}{\phi(q)\log x}$ and Theorem \ref{T: Brun-Titchmarsh inequality}, we bound the sum in the previous expression by \begin{equation*}
        \log x\frac{2y}{\phi(q)\log (\frac{y}{q})}+ O\Big(\frac{(x+y)y}{\phi(q)x\log(x+y)}\Big)+ O\Big(\sum_{\underset{n\equiv a\;( Q)}{x<  n\leq x+y} } \frac{1}{\phi(q)\log n}\Big)+O(1).
    \end{equation*}Since \begin{multline*}\sum_{\underset{n\equiv a\;( Q)}{x<  n\leq x+y} } \frac{1}{\log n}\leq \int_{x}^{x+y}\frac{dt}{\log t}+O(1)=\frac{x+y}{\log(x+y)}-\frac{x}{\log x}+\int_{x}^{x+y}\frac{dt}{\log^2 t}   +O(1)\leq\\ \frac{y}{\log x}+O(\frac{y}{\log^2 x})+O(1),
    \end{multline*}we conclude that \begin{equation}\label{E: estimate for sums of von Mangoldt on primes}
         \theta(x+y,q,a)-\theta(x,q,a)\leq \frac{2y\log x}{\phi(q)\log(\frac{y}{q})}+O(\frac{y}{\log x})+O(1).
    \end{equation}Consequently, if we combine \eqref{E: prime powers suck} and \eqref{E: estimate for sums of von Mangoldt on primes}, we arrive at \begin{equation*}
        \psi(x+y,q,a)-\psi(x,q,a)\leq  \frac{2y\log x}{\phi(q)\log(\frac{y}{q})}+O(\frac{y}{\log x})+O(x^{\frac{1}{2}}\log x),
    \end{equation*} as was to be shown.\end{proof}
\begin{remark}
    We will apply this corollary for $q=W$ and $y\gg x^{5/8+\e}$. Note that for $y$ in this range, the second error term can be absorbed into the first one.
\end{remark}

\subsection{Quantitative equidistribution mod 1}~

\begin{definition}
Let $(x_n)_{n\in \mathbb{N}}$ be a real valued sequence. We say that $(x_n)_{n\in\mathbb{N}}$ is
\begin{itemize}
    \item[$\bullet$] {\em equidistributed $mod \;1$}  if for all $0\leq a< b \leq 1,$ we have
\begin{equation}\label{Equi}
\lim_{N\to+\infty}\frac{\big|\big\{ n\in \{1,\ldots, N\}:\;\{x_n\}\in [a,b)\big\}\big|}{N}=b-a.
\end{equation} 
\item[$\bullet$] {\em well distributed $mod \;1$}  if for all $0\leq a<b\leq 1,$ we have
\begin{equation}\label{Well_Equi}
\lim_{N\to+\infty}\frac{\big|\big\{n\in \{1,\ldots,N\}:\;\{x_{k+n}\}\in [a,b)\big\}\big|}{N}=b-a, \text{ uniformly in } k=0,1,\ldots.
\end{equation}
\end{itemize}
\end{definition}

In the case of polynomial sequences, their equidistribution properties are well understood. If the polynomial has rational non-constant coefficients, it is straightforward to check that the sequence of its fractional parts is periodic. On the other hand, for polynomials with at least one non-constant irrational coefficient, we have the following theorem.

\begin{customthm}{C}[Weyl]\label{T: Weyl}
Let $p\in \R[t]$ be a polynomial with at least one non-constant irrational coefficient. Then, the sequence $(p(n))_{n\in \N}$ is well-distributed $mod\: 1$.
\end{customthm}

This theorem is classical and for a proof, we refer the reader to \cite[Chapter 1, Theorem 3.2]{Kuipers-Niederreiter}.\footnote{While this theorem concerns the case of equidistribution, the more general result follows easily by a straightforward adaptation of van der Corput's difference theorem to the case of well-distribution. The authors of \cite{Kuipers-Niederreiter} discuss this in the notes of Section 5 in Chapter 1.} 
In the case of Hardy field functions, we have a complete characterization of equidistribution modulo 1 due to Boshernitzan. We recall here \cite[Theorem 1.3]{Boshernitzan-equidistribution}.
\begin{customthm}{D}[Boshernitzan]\label{T: Boshernitzan}
    Let $a\in\mathcal{H}$ be a function of polynomial growth. Then, the sequence $(a(n))_{n\in \N}$ is equidistributed $mod \; 1$ if and only if $|a(t)-p(t)|\succ \log t$ for every $p\in \Q[t]$.
\end{customthm}
This theorem explains the assumptions in Theorem \ref{T: the main comparison} and, in particular, condition \eqref{E: far away from rational polynomials}. Indeed, since we need equidistribution assumptions for our method to work, this condition appears to be vital.
We will invoke Boshernitzan's theorem only in the case of sub-fractional functions. Indeed, we will investigate the equidistribution properties of fast-growing functions by studying their exponential sums in short intervals. This leads to a proof of the previous theorem indirectly, at least in the case that the function involved is not sub-fractional.

For our purposes, we will need a quantitative version of the equidistribution phenomenon.
For a finite sequence of real numbers $(u_n)_{1\leq n\leq N}$ and an interval $[a,b]\subseteq [0,1]$, we define the {\em discrepancy} of the sequence $u_n$ with respect to $[a,b]$ by 
\begin{equation}\label{E: discrepancy}
    \D_{[a,b]}(u_1,\dots,u_N)=\Bigg|\frac{\bigabs{\big\{n\in \{1,\ldots,N\}\colon \{u_n\}\in [a,b]     \big\}    }}{N}-(b-a)\Bigg|.
\end{equation}
The discrepancy of a sequence is a quantitative measure of how close a sequence of real numbers is to being equidistributed modulo 1. For example, it is immediate that for an equidistributed sequence $u_n$, we have that \begin{equation*}
    \lim\limits_{N\to+\infty}  {\D_{[a,b]} (u_1,\dots,u_N)  } =0,
\end{equation*}for all $0\leq a\leq b\leq 1.$
 For an in-depth discussion on the concept of discrepancy and the more general theory of equidistribution on $\T$, we refer the reader to \cite{Kuipers-Niederreiter}. Our only tool will be an upper bound of Erd\H{o}s and Tur\'{a}n on the discrepancy
of a finite sequence. For a proof of this result, see \cite[Chapter 2, Theorem 2.5]{Kuipers-Niederreiter}.\footnote{In this book, the theorem is proven for measures of the form  $\nu= \frac{1}{N}\sum_{i=1}^{N}\delta_{x_i}$, although the more general statement follows by noting that every Borel probability measure is a weak limit of measures of the previous form.}

\begin{customthm}{E}[Erd\H{o}s-Tur\'{a}n]\label{T: Erdos-Turan}
    There exists an absolute constant $C$, such that for any positive integer $M$ and any Borel probability measure $\nu$ on $\T$, we have  \begin{equation*}
        \sup_{A\subseteq \T} |\nu(A)-\lambda(A)|\leq C\Big(\frac{1}{M}+\sum_{m=1}^{M}\frac{|\widehat{\nu}(m)|}{m}\Big),
    \end{equation*}where $\lambda$ is the Lebesgue measure on $\T$ and the supremum is taken over all arcs $A$ of $\T$.

    In particular, specializing to the case that $\nu= N^{-1}\sum_{i=1}^{N}\delta_{\{u_i\}}$,  where $u_1,\dots,u_N$ is a finite sequence of real numbers, we have \begin{equation}
        \D_{[a,b]}(u_1,\dots,u_N)\leq C\Big( \frac{1}{M}+\sum_{m=1}^{M}\frac{1}{m} \Bigabs{\frac{1}{N} \sum_{n=1}^{N} e(mu_n)      } \Big)
    \end{equation} for all positive integers $M$ and all $0\leq a\leq b<1.$
\end{customthm}

It is clear that in order to get the desired bounds on the discrepancy in our setting, we will need some estimates for exponential sums of Hardy field sequences in short intervals. Due to the Taylor approximation, this is morally equivalent to establishing estimates for exponential sums of polynomial sequences. There are several well-known estimates in this direction, the most fundamental of these being a result of Weyl that shows that an exponential sum along a polynomial sequence is small unless all non-constant coefficients of the polynomial are ``major-arc''. In the case of strongly non-polynomial Hardy field functions, we will only need to study the leading coefficient of the polynomial in its Taylor approximation, which will not satisfy such a major-arc condition.
To this end, we require the following lemma. 
\begin{lemma}\label{L: Weyl-type estimate}
Let $0<\delta <1$ and $d\in \N$. There exists a positive constant $C$ depending only on $d$, such that if $p(x)=a_dx^d+\dots +a_1x+a_0$ is a real polynomial that satisfies \begin{equation*}
    \Bigabs{\frac{1}{N}\sum_{n=1}^{N} e(p(n))  }>\delta,
\end{equation*}then, for every $1\leq k\leq d$, there exists $q\in \Z$  with $|q|\leq \delta ^{-C}$, such that $ N^k \norm{qa_k }_{\T}\leq \delta^{-C}   $.
\end{lemma}
Note that there is no dependency of the constant on the length of the averaging interval, or on the implicit polynomial $p$ (apart from its degree). For a proof of this lemma, see \cite[Proposition 4.3]{Green-Tao-quantitative}, where a more general theorem is established in the setting of nilmanifolds as well.

\subsection{Nilmanifolds and correlation sequences}

Let $G$ be a nilpotent Lie group with nilpotency degree $s$ and let $\Gamma$ be a discrete and cocompact subgroup. The space $X=G/\G$ is called an $s$-step {\em nilmanifold}. The group $G$ acts on the space $X$ by left multiplication and the measure on $X$ that is invariant under this action is called the {\em Haar measure} of $X$, which we shall denote by $m_X$. 

Given a  sequence of points $x_n\in X$, we will say that the sequence $x_n$ is {\em equidistributed} on $X$, if for any continuous function $F:X\to \C$ we have that \begin{equation*}
    \lim\limits_{N\to+\infty} \frac{1}{N}\sum_{n=1}^{N} F(x_n)=\int F \ d\, m_X.
\end{equation*}

A {\em subnilmanifold} of $X=G/\G$ is a set of the form $Hx$, where $H$ is a closed subgroup of the Lie group $G$, $x\in X$ and such that $Hx$ is closed in $X$.  

Let $g$ be any element on the group $G$. Then, for any $x\in X$, the closed orbit of the action of $g$ on $x$ will be denoted by $\overline{(g^{\Z}x)}$. 
It is known that this set is a subnilmanifold of $X=G/\G$ and that the sequence $g^n x$ is equidistributed in the subnilamnifold $\overline{(g^{\Z}x)}$ (see, for example, \cite[Chapter 11, Theorem 9]{Host-Kra-structures}). 

We now present the following definition for nilsequences in several variables.
\begin{definition}\label{D: definition of nilsequence}
    Let $k,s$ be positive integers and let $X=G/\G$ be an $s$-step nilmanifold. Assume that $g_1,\dots, g_k$ are pairwise commuting elements of the group $G$, $F: X\to\C$ is a continuous function on $X$ and $x\in X$. Then, the sequence $$\psi(n_1,\dots, n_k)=F(g_1^{n_1}\cdot\ldots\cdot g_k^{n_k}x), \ \text{where } n_1,\dots,n_k\in \Z$$ is called an {\em $s$-step nilsequence in $k$-variables}.
\end{definition}

The main tool that we will need is an approximation of general nilsequences by multi-correlation sequences in the $\ell^{\infty}$-sense. The following lemma is established in \cite[Proposition 4.2]{Fra-Host-weighted}.

\begin{lemma}\label{L: approximation by nilsequences}
    Let $k,s$ be positive integers and $\psi:\Z^k\to \C$ be a $(s-1)$-step nilsequence in $k$ variables. Then, for every $\e>0$, there exists a system $(X,\X,\m,T_1,\dots, T_k)$ and functions $F_1,\dots, F_s$ on $L^{\infty}(\m)$, such that the sequence $b(n_1,\dots,n_k)$ defined by \begin{equation*}
        b(n_1,\dots,n_k)=\int \prod_{j=1}^{s}\big(  T_1^{\ell_j n_1}\cdot\ldots\cdot T_k^{\ell_jn_k} \big)F_j\ d\,\m, \ (n_1,\dots, n_k)\in \Z^k
    \end{equation*}with $\ell_j=s!/j$ satisfies \begin{equation*}
        \norm{\psi-b}_{\ell^{\infty}(\Z^k)}\leq \e.
    \end{equation*}
\end{lemma}

\begin{comment*}
    The definition of nilsequences used in \cite{Fra-Host-weighted} imposed that $x=\G$ and that ${\bf n}\in \N^k$. However, their arguments generalize in a straightforward manner to the slightly more general setting that we presented above.
 \end{comment*}

\section{Lifting to an extension flow}\label{Section-Lifting section}

In this section, we use a trick that allows us to replace the polynomial ergodic averages  with similar ergodic averages over $\R$ actions on an extension of the original probability space, removing the rounding functions in the process.. This argument is implicit in \cite{koutsogiannis-closest} for Ces\`{a}ro averages, so we adapt its proof to the setting of short intervals. 

\begin{proposition}\label{P: Gowers norm bound on variable polynomials}
    Let $k,\ell,d$ be positive integers and let $L(N)$ be a positive sequence satisfying $ N^{\frac{5}{8}+\e}\ll L(N)\ll N^{1-\e}$. Let $(X,\X,\m,T_1,\dots,T_k)$ be a system of commuting transformations. Then, there exists a positive integer $s$ depending only on $k,\ell,d$, such that for any variable family $\mathcal{P}=\{p_{i,j,N}\colon 1\leq i\leq k, 1\leq j\leq \ell\}$ of polynomials with degrees at most $d$ that, for all $i,j,$ satisfy \begin{equation}\label{E: non-concetration around 1}
        \lim_{\delta\to 0^+}\lim_{N\to+\infty}\frac{\left|\{N\leq n\leq N+L(N):\; \{p_{i,j,N}(n)\}\in [1-\delta,1)\}\right|}{L(N)}=0,
    \end{equation} we have that for any $0<\delta<1$ and functions $f_1,\dots, f_{\ell}\in L^{\infty}(\m)$ \begin{multline*}
        \Bignorm{\E_{N\leq n\leq N+L(N)} \ \big(\La_{w,b}(n)-1\big) \prod_{j=1}^{\ell} \prod_{i=1}^{k} T_i^{\floor{p_{i,j,N}(n) }} f_j    }_{L^2(\m)}\ll_{k,\ell,d}\\ \frac{1}{\delta^{k\ell }}\Big(\bignorm {\La_{w,b}(n)-1}_{U^s(N,N+sL(N)]}+o_w(1)\Big)+
        o_{\delta}(1)(1+o_w(1)),
    \end{multline*}
     for all $1\leq b\leq W,\ (b,W)=1$,
    where $W=\prod_{p\in \mathbb{P}\colon p\leq w}p$.
\end{proposition}

\begin{proof}
Let $\lambda$ denote the Lebesgue measure on $[0,1)$ and we define (as in \cite{koutsogiannis-closest}) the measure-preserving $\R^{k \ell}$-action $\displaystyle \prod_{i=1}^{k}S_{i,s_{i,1}}\cdot \ldots\cdot \prod_{i=1}^k S_{i,s_{i,\ell}}$  on the space $Y:=X\times [0,1)^{k\ell},$ endowed with the measure $\nu:=\mu\times \lambda^{k\ell}$, by
$$\prod_{j=1}^{\ell} \prod_{i=1}^k S_{i,s_{i,j}}(x,a_{1,1},\ldots,a_{k,1},a_{1,2},\ldots,a_{k,2},\ldots,a_{1,\ell},\ldots,a_{k,\ell})=$$
$$\left(\prod_{j=1}^{\ell} \prod_{i=1}^k T_i^{[s_{i,j}+a_{i,j}]}x,\{s_{1,1}+a_{1,1}\},\ldots,\{s_{k,1}+a_{k,1}\},\ldots,\{s_{1,\ell}+a_{1,\ell}\},\ldots,\{s_{k,\ell}+a_{k,\ell}\}\right). $$
 If $f_1,\ldots,f_{\ell}$ are bounded functions on $X,$ we define the $Y$-extensions of $f_j,$ setting for every element
  $(a_{1,1},\ldots,a_{k,1},a_{1,2},\ldots,a_{k,2},\ldots,a_{1,\ell},\ldots,a_{k,\ell})\in [0,1)^{k\ell}$:  $$\hat{f}_j(x,a_{1,1},\ldots,a_{k,1},a_{1,2},\ldots,a_{k,2},\ldots,a_{1,\ell},\ldots,a_{k,\ell})=f_j(x),\;\;1\leq j\leq \ell;\;\;$$ and we also define the function
  $$\hat{f}_0(x,a_{1,1},\ldots,a_{k,1},a_{1,2},\ldots,a_{k,\ell})= 1_{[0,\delta]^{k\ell }}(a_{1,1},\ldots,a_{k,1},a_{1,2},\ldots,a_{k,\ell}).$$ 
  For every $N\leq n\leq N+L(N),$ we consider the functions (on the original space $X$)
   $$b_N(n):= (\prod_{i=1}^k T_i^{[p_{i,1,N}(n)]})f_1\cdot\ldots\cdot(\prod_{i=1}^k T_i^{[p_{i,\ell,N}(n)]})f_\ell$$
  as well as the functions $$\tilde{b}_N(n):= \hat{f}_0\cdot(\prod_{j=1}^{\ell}\prod_{i=1}^k S_{i,\delta_{j 1}\cdot p_{i,1,N}(n)})\hat{f}_1\cdot \ldots\cdot (\prod_{j=1}^{\ell}\prod_{i=1}^k S_{i,\delta_{j \ell}\cdot p_{i,\ell,N}(n)})\hat{f}_{\ell}$$ defined on the extension $Y$. Here, $\delta_{ij}$ denotes the Kronecker $\delta$, meaning that the only terms that do not vanish are the diagonal ones (i.e., when $i=j$).
  For every $x\in X$, we also let
  $$b'_N(n)(x):=\int_{[0,1)^{k\ell }}\tilde{b}_N(n)(x,a_{1,1},\ldots,a_{k,1},a_{1,2},\ldots,a_{k,2},\ldots,a_{1,\ell},\ldots,a_{k,\ell})\,d\lambda^{k\ell },$$ where the integration is with respect to the variables $a_{i,j}.$

Using the triangle and Cauchy-Schwarz inequalities, we have 
\begin{multline}\label{E: triangle Andreas 2}
\delta^{k\ell }\Bignorm{\E_{N\leq n\leq N+L(N)} \big(\La_{w,b}(n)-1\big) b_N(n)}_{L^2(\mu)}  \leq \\ \Bignorm{\E_{N\leq n\leq N+L(N)} \big(\La_{w,b}(n)-1\big)\cdot (\delta^{k\ell }b_N(n)-b'_N(n))}_{L^2(\mu)}
 + \Bignorm{\E_{N\leq n\leq N+L(N)} \big(\La_{w,b}(n)-1\big)\tilde{b}_N(n)}_{L^2(\nu)}.
\end{multline}

Using Proposition \ref{P: PROPOSITION THAT WILL BE COPYPASTED FROM FRA-HO-KRA}, we find an integer $s\in \N,$ depending only on the integers $k, \ell, d,$ and a constant $C_s$ depending on $s,$ such that
\begin{equation}\label{E: bound for the first term}
    \Bignorm{\E_{N\leq n\leq N+L(N)} \big(\La_{w,b}(n)-1\big)\tilde{b}_N(n)}_{L^2(\nu)}\leq C_s\left(\bignorm{\La_{w,b}-1}_{U^s(N,N+sL(N)]}+o_N(1)\right),
\end{equation}
where the $o_N(1)$ term depends only on the integer $s$ and the sequence $\La_{w,b}(n).$


Now we study the first term $$\Bignorm{\E_{N\leq n\leq N+L(N)} \big(\La_{w,b}(n)-1\big)\cdot (\delta^{k\ell }b_N(n)-b'_N(n))}_{L^2(\mu)}$$ in \eqref{E: triangle Andreas 2}.
For every $x\in X$ and $N\leq n\leq N+L(N),$ we have
\begin{equation*}
\begin{split}
	&\quad \Big|\delta^{k\ell } b_N(n)(x)-b'_N(n)(x)\Big|= 
	\\& \quad\quad\quad\quad\quad\quad\quad \left|\int_{[0,\delta]^{k\ell }}  \left(\prod_{j=1}^\ell f_j(\prod_{i=1}^k T_i^{[p_{i,j,N}(n)]}x)-\prod_{j=1}^\ell f_j(\prod_{i=1}^k T_i^{[p_{i,j,N}(n)+a_{i,j}]}x) \right)\, d \lambda^{k\ell }\right|. 
	\end{split}
	\end{equation*}
Since all the integrands $a_{i,j}$ are less than or equal than $\delta,$ we deduce that if all of the implicit polynomials satisfy $\{p_{i,j,N}(n)\}<1-\delta,$ we have $T_i^{[p_{i,j,N}(n)+a_{i,j}]}=T_i^{[p_{i,j,N}(n)]}$ for all $1\leq i\leq k,$ $1\leq j\leq \ell.$  To deal with the possible case where $\{p_{i,j,N}(n)\}\geq 1-\delta$ for at least one of our polynomials, we define, for every $1\leq i\leq k,$ $1\leq j\leq \ell,$ the set
$$E_{\delta,N}^{i,j}:=\{n\in [N,N+L(N)]\colon \{p_{i,j,N}(n)\}\in [1-\delta,1)\}.$$ Then, by using the fact that $${\bf 1}_{E_{\delta,N}^{1,1}\cup \ldots\cup E_{\delta,N}^{1,\ell}\cup E_{\delta,N}^{2,1}\cup\ldots\cup E_{\delta,N}^{k,\ell}}\leq \sum_{(i,j)\in[1,k]\times[1,\ell]} {\bf 1}_{E_{\delta,N}^{i,j}} $$ and that ${\bf 1}_{E_{\delta,N}^{i,j}}(n)={\bf 1}_{[1-\delta,1)}(\{p_{i,j,N}(n)\}),$  
 we infer that
$$\Big|\delta^{k\ell }b_N(n)(x)-b'_N(n)(x)\Big|\leq 2\delta^{k\ell }\sum_{(i,j)\in[1,k]\times[1,\ell]}   {\bf 1}_{[1-\delta,1)}(\{p_{i,j,N}(n)\}) $$for every $x\in X$. In view of the above, using the inequality $|\La_{w,b}(n)-1|\leq \La_{w,b}(n)+1$, we deduce that
\begin{multline*}
    \E_{N\leq n\leq N+L(N)}\big|\big(\La_{w,b}(n)-1\big)\big|\cdot {\bf 1}_{[1-\delta,1)}(\{p_{i,j,N}(n)\}) \leq \ \ \ \   \\ 
    \E_{N\leq n\leq N+L(N)} \big(\La_{w,b}(n)-1\big)\cdot{\bf 1}_{[1-\delta,1)}(\{p_{i,j,N}(n)\})+
    2\E_{N\leq n\leq N+L(N)}{\bf 1}_{[1-\delta,1)}(\{p_{i,j,N}(n)\})\leq \\
       \E_{N\leq n\leq N+L(N)}\big(\La_{w,b}(n)-1\big)\cdot{\bf 1}_{[1-\delta,1)}(\{p_{i,j,N}(n)\})+2\cdot\frac{|E^{i,j}_{\delta,N}|}{L(N)}.
\end{multline*}
Since each polynomial $p_{i,j,N}$ satisfies \eqref{E: non-concetration around 1} for large $N$ and small enough $\delta,$ the term (and the sum of finitely many terms of this form)  $\frac{|E_{\delta,N}^{i,j}|}{L(N)}$ is as small as we want.

It remains to show that the term $$\E_{N\leq n\leq N+L(N)} \big(\La_{w,b}(n)-1\big)\cdot{\bf 1}_{[1-\delta,1)}(\{p_{i,j,N}(n)\})$$ goes to zero as $N\to\infty,$ then $w\to\infty$ and finally $\delta\to 0^+.$ To this end, it suffices to show $$\E_{N\leq n\leq N+L(N)}\ \big(\La_{w,b}(n)-1\big)e^{2\pi i m p_{i,j,N}(n)}\to 0$$ as $N\to\infty$ and then $w\to\infty$ for all $m\in \Z\setminus\{0\},$\footnote{ This follows by the fact that if $f$ is Riemann integrable on $[0,1)$ with $\int_{[0,1)}f(x)\,dx=c,$  then,  for every $\varepsilon>0,$ we can find trigonometric polynomials $q_1,\;q_2,$ with no constant terms, with $q_1(t)+c-\varepsilon\leq f(t)\leq q_2(t)+c+\varepsilon.$ We use this for the function $f={\bf 1}_{[1-\delta,1)}.$} which follows from Lemma \ref{L: discorrelation of W-tricked with polynomial phases}. 
\end{proof}

\section{Equidistribution in short intervals}\label{Section-Removing error terms section}

We gather here some useful propositions that describe the behavior of a Hardy field function when restricted to intervals of the form $[N, N+L(N)]$, where $L(N)$ grows slower compared to the parameter $N$. In our applications, we will typically need the function $L(N)$ to grow faster than $N^{5/8}$ in order to be able to use the uniformity results in short intervals, but we will not need to work under this assumption throughout most of this section, the only exception being Proposition \ref{P: remove error terms for polynomial functions} below. We will also present an example that illustrates the main points in the proof of Theorem \ref{T: the main comparison} in the following section.

\subsection{Details on the proof}
In the case of  strongly non-polynomial functions that also grow faster than some fractional power, we show that the associated Taylor polynomial $p_N(n)$ has ideal equidistribution properties. Indeed, by picking the length $L(N)$ a little more carefully, one gains arbitrary logarithmic powers over the trivial bound in the exponential sums of $p_N$. Consequently, we demonstrate that the number of integers in $[N, N+L(N)]$ for which $\floor{a(n)}\neq \floor{p_N(n)}$ is less than $L(N)(\log N)^{-100}$ (say) and, thus,  their contribution to the average is negligible. Therefore, for all intents and purposes, one can suppose that the error terms are identically zero.

The situation is different when a function that grows slower than all fractional powers is involved since these functions are practically constant in these short intervals. For instance, if one has the function $p(t)+\log^2 t$, where $p$ is a polynomial, the only feasible  approximation is of the form $p(n)+\log^2 n=p(n)+\log^2 N +e_N(n)$,  where $e_N(n)$ converges to 0.
While it seems that we do have a polynomial as the main term in the approximation (at least when $p$ is non-constant),
quantitative bounds on the exponential sums of the polynomial component cannot be established in this case at all. The main reason is that such bounds depend heavily on the diophantine properties of the coefficients of $p$, for which we have no data.

In the case that $p$ is a constant polynomial, we can use the equidistribution (mod 1) of the sequence $\log^2 n$ to show that in most short intervals $[N, N+L(N)]$, we have $\floor{\log^2 n}=\floor{\log^2 N}$ for all $n\in [N, N+L(N)]$. The contribution of the bad short intervals is then bounded using the triangle inequality and Corollary \ref{C: Brun-Titchmarsh inequality for von Mangoldt sums}.

Suppose that the polynomial $p$ above is non-constant.
In the case that $p$ has rational non-constant coefficients, we split our averages to suitable arithmetic progressions so that the resulting polynomials have integer coefficients (aside from the constant term), and the effect of $e_N(n)$  will be eliminated when we calculate the integer parts. In the case that $p$ has a non-constant irrational coefficient, we can invoke the well-distribution of $p(n)$ to conclude that the number of integers of the set \begin{equation*}
    E_N=\{n\in [N, N+L(N)]\colon \floor{p(n)+\log^2 n}\neq \floor{p(n)+\log^2 N}\}
\end{equation*}is $O(\e L(N))$, for a fixed small parameter $\e$ and $N$ large. However, in order to bound the total contribution of the set $E_N$, we can only use the triangle inequality in the corresponding ergodic averages, so we are forced to extract information on how large the quantity $$\frac{1}{L(N)}\sum_{N\leq n\leq N+L(N)} \La_{w, b}(n){\bf 1}_{E_N}(n)$$
can be. This can be bounded effectively if the corresponding exponential sums $$\frac{1}{L(N)}\sum_{N\leq n\leq N+L(N)} \La_{w,b}(n)e\big(p(n)\big)$$
are small. This is demonstrated by combining the fact that the exponential sums of $p(n)$ are small (due to the presence of an irrational coefficient) with the fact that exponential sums weighted by $\La_{w,b}(n)-1$ are small due to the uniformity of the $W$-tricked von Mangoldt function. The conclusion follows again by an application of the Erd\H{o}s-Tur\`{a}n inequality, this time for a probability measure weighted by $\La_{w,b}(n)$.

\subsection{A model example}
We sketch the main steps in the case of the ergodic averages \begin{equation}\label{E: example averages}
    \frac{1}{N}\sum_{n=1}^{N} \big(\La_{w,b}(n)-1\big)T^{\floor{n\log n}}f_1\cdot T^{\floor{an^2+\log n }}f_2\cdot T^{\floor{\log^2 n}}f_3.
\end{equation}where $a$ is an irrational number.
We will show that the $L^2$-norm of this expression converges to 0, as $N\to+\infty$ and then $w\to+\infty$.
Note that the three sequences in the iterates satisfy our hypotheses. In addition, we remark that the arguments below are valid in the setting where we have three commuting transformations, but we consider a simpler case for convenience. Additionally, we do not evaluate the sequences at $Wn+b$ (as we should in order to be in the setup of Theorem \ref{T: the main comparison}), since the underlying arguments remain identical apart from changes in notation.

We choose $L(t)=t^{0.66}$ (actually, any power $t^c$ with $5/8<c<2/3$ works here) and claim that it suffices to show that \begin{equation}\label{E: example averages in short intervals}
    \E_{1\leq r \leq R} \Bignorm{\E_{r\leq n\leq r+L(r)  }\big(\La_{w,b}(n)-1\big)T^{\floor{n\log n}}f_1\cdot T^{\floor{an^2+\log n }}f_2\cdot T^{\floor{\log^2 n}}f_3        }_{L^2(\m)}=0.
\end{equation}This reduction is the content of Lemma \ref{L: long averages to short averages}. Now, we can use the Taylor expansion around $r$ to write \begin{align*}
    n\log n&=r\log r+(\log r+1)(n-r)+\frac{(n-r)^2}{2r}-\frac{(n-r)^3}{6\xi_{1,n,r}^2}\\
    \log n&=\log r+\frac{n-r}{\xi_{2,n,r}}\\
    \log^2  n&=\log^2 r +\frac{2(n-r)\log \xi_{3,n,r}}{\xi_{3,n,r}},
\end{align*}for some real numbers $\xi_{i,n,r}\in [r,n]$ ($i=1,2,3$). Our choice of $L(t)$ implies that \begin{equation*}
    \Bigabs{\frac{(n-r)^3}{6\xi_{1,n,r}^2}}\leq \frac{r^{3\cdot 0.65}}{6r^2}\ll 1,
\end{equation*}and similarly for the other two cases. To be more specific, there exists a $\delta>0$, such that all the error terms (the ones involving the quantities $\xi_{i,n,r}$) are $O(r^{-\delta})$.

Let us fix a small $\e>0$.
Firstly, we shall deal with the third iterate, since this is the simplest one. Observe that if $r$ is chosen large enough and such that it satisfies $\{\log^2 r\}\in (\e,1-\e)$, then for all $n\in [r,r+L(r)]$, we will have $$\floor{\log^2 n}=\floor{\log^2 r},$$since the error terms in the expansion are $O(r^{-\delta})$, which is smaller than $\e$ for large $r$. In addition, the sequence $\log^2 n$ is equidistributed modulo 1, so our prior assumption can fail for at most $3\e R$ (say) values of $r\in [1, R]$, provided that $R$ is sufficiently large. For the bad values of $r$, we use the triangle inequality for the corresponding norm to deduce that their contribution on the average is $O(\e R)$, which will be acceptable if $\e$ is small. Actually, in order to establish this, we will need to use Corollary \ref{C: Brun-Titchmarsh inequality for von Mangoldt sums}, though we will ignore that in this exposition. In conclusion, we can rewrite the expression in \eqref{E: example averages in short intervals} as \begin{equation}\label{E: example first reduction}
     \E_{1\leq r\leq R} \Bignorm{\E_{r\leq n\leq r+L(r)  }\big(\La_{w,b}(n)-1\big)T^{\floor{n\log n}}f_1\cdot T^{\floor{an^2+\log n }}f_2\cdot T^{\floor{\log^2 r}}f_3        }_{L^2(\m)}+O(\e ).
\end{equation}

Now, we deal with the first function. We claim that the discrepancy of the finite sequence $$\Big(\{r\log r+(\log r+1)(n-r)+\frac{(n-r)^2}{2r}\}\Big)_{r\leq n\leq r+L(r)}$$is $O_A(\log^{-A }r)$ for any $A>0$. We will establish this in Proposition \ref{P: remove error term for fast functions} using Lemma~\ref{L: Weyl-type estimate} and Theorem \ref{T: Erdos-Turan}. As a baby case, we show the following estimate for some simple trigonometric averages: \begin{equation*}
   \Bigabs{ \E_{r\leq n\leq r+L(r)} e\Big(\frac{(n-r)^2}{2r}\Big)}\leq \frac{1}{\log^A r}
\end{equation*}for $r$ large enough. Indeed, if that inequality fails for some $r\in \N$, there exists an integer $|q_r|\leq \log^{O(A)}r$, such that \begin{equation*}
    \Bignorm{\frac{q_r}{2r}}_{\T}\leq \frac{\log^{O(A)}r}{(L(r))^2}.
\end{equation*}Certainly, if $r$ is large enough, we can replace the norm with the absolute value, so that the previous inequality implies that \begin{equation*}
    \big(L(r)\big)^2 \leq \frac{2r\log^{O(A)}r}{|q_r|}. 
\end{equation*}However, the choice $L(t)=t^{0.66}$  implies that this inequality is false for large $r$.

In our problem, we can just pick $A=2$.
Using the definition of discrepancy, we deduce that the number of integers in $[r,r+L(r)]$, for which we have $$\{r\log r+(\log r+1)(n-r)+\frac{(n-r)^2}{2r}\}\in [0,r^{-\delta/2}]\cup [1-r^{-\delta/2},1)$$ is $O(L(r)\log^{-2} r)$. However, if $n$ does not belong to this set of bad values, we conclude that \begin{equation*}
    \floor{n\log n}=\floor{ r\log r+(\log r+1)(n-r)+\frac{(n-r)^2}{2r}}
\end{equation*}since the error terms are $O(r^{-\delta})$. Furthermore, since $\La_{w,b}(n)=O(\log r)$ for $n\in [r,r+L(r)]$, we conclude that the contribution of the bad values is $o_r(1)$ on the inner average.
Therefore, we can rewrite the expression in \eqref{E: example first reduction} as \begin{equation}\label{E: example second reduction}
       \E_{1\leq r\leq R} \Bignorm{\E_{r\leq n\leq r+L(r)  }\big(\La_{w,b}(n)-1\big)T^{\floor{p_r(n)}}f_1\cdot T^{\floor{an^2+\log n }}f_2\cdot T^{\floor{\log^2 r}}f_3        }_{L^2(\m)}+O(\e )+o_R(1),
\end{equation}where $p_r(n)=r\log r+(\log r+1)(n-r)+\dfrac{(n-r)^2}{2r}$.
\smallskip

Finally, we deal with the second iterate. We consider the parameter $\e$ as above and set $M=1/\e$. Once again, we shall assume that $r$ is very large compared to $M$. Since $a$ is irrational, we have that the sequence $an^2$ is well-distributed modulo 1, so we would expect the number of $n$ for which $\{an^2+\log r\}\not\in [\e,1-\e]$ to be small. Note that for the remaining values of $n$, we have $\floor{an^2+\log n}=\floor{an^2+\log r}$, since the error term in the approximation is $O(r^{-\delta})$. Therefore, we estimate the size of the set $$\B_{r,\e}:=\{n\in [r,r+L(r)]\colon \{an^2+\log  r\}\in [0,\e]\cup [1-\e,1)   \}$$

Using Weyl's theorem, we conclude that \begin{equation}\label{E: exponential sums in example}
    \max_{1\leq m\leq M}  \Bigabs{ \E_{r\leq n\leq r+L(r)} e\big(m(an^2+\log r)\big)  }=o_r(1).
\end{equation}
Here, the $o_r(1)$ term depends on $M=1/\e$, but since we will send $r\to+\infty$ and then $\e\to 0$, this will not cause any issues. We suppress these dependencies in this exposition. 

An application of Theorem \ref{T: Erdos-Turan} implies that \begin{equation}\label{E: size of bad set in example}
    \frac{|\B_{r,\e}|}{L(r)}\ll 2\e +\frac{1}{M} +\sum_{m=1}^{M}\frac{1}{m}\Bigabs{ \E_{r\leq n\leq r+L(r)} e\big(m(ar^2+\log r)\big)  },
\end{equation}so that $|\mathcal{B}_{r,\e}|\ll (\e+o_r(1))L(r)$. Additionally, we will need to estimate \begin{equation*}
    \frac{1}{L(r)} \sum_{r\leq n\leq r+L(r)} \La_{w,b}(n) {\bf 1}_{\B_r}(n),
\end{equation*}which will arise when we apply the triangle inequality to bound the contribution of the set $\B_r$. However, we have that  \begin{equation}
     \max_{1\leq m\leq M}  \Bigabs{ \E_{r\leq n\leq r+L(r)} \La_{w,b}(n)e\big(m(an^2+\log r)\big)  }=o_w(1)+o_r(1),
\end{equation}which can be seen by splitting $\La_{w,b}(n)=(\La_{w,b}(n)-1)+1$, applying the triangle inequality and using Lemma \ref{L: discorrelation of W-tricked with polynomial phases} and \eqref{E: exponential sums in example}, respectively, to treat the resulting exponential averages.
In view of this, we can apply the Erd\H{o}s-Tur\'{a}n inequality (Theorem \ref{T: Erdos-Turan}) for the probability measure $$\nu(S)=\dfrac{\sum\limits_{r\leq n\leq r+L(r)}^{} \La_{w,b}(n)\delta_{\{an^2+\log r\}}(S)  }{\sum\limits_{r\leq n\leq r+L(r)}^{}  \La_{w,b}(n)}$$ 
as well as Corollary \ref{C: Brun-Titchmarsh inequality for von Mangoldt sums} (to bound the sum in the denominator) to conclude that \begin{equation*}
    \frac{1}{L(r)} \sum_{r\leq n\leq r+L(r)} \La_{w,b}(n) {\bf 1}_{\B_r}(n)\ll \e +o_w(1)\log \frac{1}{\e}+o_r(1),
\end{equation*}Therefore, if we apply the triangle inequality, we conclude that the contribution of the set $\mathcal{B}_{r,\e}$ on the average over $[r,r+L(r)]$ is at most $O(\e+o_w(1)\log \frac{1}{\e} +o_r(1))$. This is acceptable if we send $R\to +\infty$, then $w\to +\infty$, and then $\e\to 0$ at the end.

Ignoring the peculiar error terms that turn out to be satisfactory, we can rewrite the expression in \eqref{E: example second reduction} as 
\begin{equation}
      \E_{1\leq r\leq R} \Bignorm{\E_{r\leq n\leq r+L(r)  }\big(\La_{w,b}(n)-1\big)T^{\floor{p_r(n)}}f_1\cdot T^{\floor{an^2+\log r }}f_2\cdot T^{\floor{\log^2 r}}f_3        }_{L^2(\m)}.
\end{equation}
Now, the iterates satisfy the assumptions of Proposition \ref{P: Gowers norm bound on variable polynomials}. This is true for the first iterate since we have a good bound on the discrepancy and it is also true for the second iterate because
the polynomial $an^2$ has an irrational coefficient (so we can use its well-distribution modulo 1). For the third one, our claim is obvious because we simply have an integer in the iterate. Therefore, we can bound the inner average by a constant multiple of the norm $$\norm{\La_{w,b}-1}_{U^s(r,r+L(r)]}$$
with some error terms that we will ignore here. Finally, we invoke Theorem \ref{T: Gowers uniformity in short intervals} to show that the average \begin{equation*}
      \E_{1\leq r\leq R}\norm{\La_{w,b}-1}_{U^s(r,r+L(r)]}
\end{equation*}converges to 0, which leads us to our desired conclusion.

\subsection{Some preparatory lemmas}
Let us fix a Hardy field $\mathcal{H}$. Firstly, we will need a basic lemma that relates the growth rate of a Hardy field function of polynomial growth with the growth rate of its derivative. To do this, we recall a lemma due to Frantzikinakis \cite[Lemma 2.1]{Fra-equidsitribution}, as  well as \cite[Proposition A.1]{Tsinas}.

\begin{lemma}\label{L: Frantzikinakis growth inequalities}
    Let $a\in \mathcal{H}$ satisfy $t^{-m}\prec a(t)\prec t^m$ for some positive integer $m$ and assume that $a(t)$ does not converge to a non-zero constant as $t\to+\infty$. Then,  \begin{equation*}
        \frac{a(t)}{t\log^2 t}\prec a'(t)\ll \frac{a(t)}{t}.
    \end{equation*}
\end{lemma}

Observe that if a function $a(t)$ satisfies the growth inequalities in the hypothesis of this lemma, then the function $a'(t)$ satisfies $\frac{t^{-1-m}}{\log^2 t} \prec a'(t)\prec t^{m-1}$. Therefore, we deduce the relations $ t^{-m-2}  \prec a'(t)\prec t^{m+2}$, which implies that the function $a'(t)$ satisfies a similar growth condition. Provided that the function $a'(t)$ does not converge to a non-zero constant as $t\to+\infty$, the above lemma can then be applied to the function $a'(t)$. 

When a function $a(t)$ is strongly non-polynomial and dominates the logarithmic function $\log t$, one can get a nice ordering relation for the growth rates of consecutive derivatives. This is the content of the following proposition.

\begin{proposition}\cite[Proposition A.2]{Tsinas}\label{P: Taylor expansion-preliminary}
Let $a\in \mathcal{H}$ be a function of polynomial growth that is strongly non-polynomial and also satisfies $a(t)\succ  \log t$. Then, for all sufficiently large $k\in \N$, we have \begin{equation*}
    1\prec \bigabs{ a^{(k)}(t)   }^{-\frac{1}{k}}\prec  \bigabs{ a^{(k+1)}(t)   }^{-\frac{1}{k+1}}\prec t.
\end{equation*} 
\end{proposition}

\begin{remark}\label{R: why we can use Lemma 4.1 for derivatives}
    The proof of Proposition \ref{P: Taylor expansion-preliminary} in \cite{Tsinas} establishes the fact that if $a$ satisfies the previous hypotheses, then the derivatives of $a$ always satisfy the conditions of Lemma~\ref{L: Frantzikinakis growth inequalities}.
\end{remark}

This proposition is the main tool used to show that a strongly non-polynomial function $a(t)$ can be approximated by polynomials in short intervals. Indeed, assume that a positive sub-linear function $L(t)$ satisfies \begin{equation}\label{E: growth relations}
    \bigabs{a^{(k)}(t) }^{-\frac{1}{k}}\prec L(t) \prec   \bigabs{a^{(k+1)}(t) }^{-\frac{1}{k+1}} 
\end{equation}for some sufficiently large $k\in \N$ (large enough so that the inequalities in Proposition \ref{P: Taylor expansion-preliminary} hold). In particular, this implies that $\lim\limits_{t\to+\infty }a^{(k+1)}(t)\to 0$.

Then, we can use the Taylor expansion around the point $N$ to write \begin{equation*}
    a(N+h)=a(N) +{ha'(N)}+\dots +\frac{h^ka^{(k)}(N)}{k!}+\frac{h^{k+1}a^{(k+1)} (\xi_{N,h}) }{(k+1)!}\  \text{for some } \xi_{N,h}\in [N,N+h]
\end{equation*}for every $0\leq h\leq L(N)$. However, we observe that \begin{equation*}
    \Bigabs{ \frac{h^{k+1}a^{(k+1)} (\xi_{N,h}  ) }{(k+1)!} }   \leq \frac{L(N)^{k+1} |a^{(k+1)} (N )|  }{(k+1)!}=o_N(1),       
\end{equation*}where we used the fact that $|a^{(k+1)}(t)|\to 0$ monotonically (since $a^{(k+1)}(t)\in \mathcal{H}$). Therefore, we have \begin{equation*}
   a(N+h)= a(N) +{ha'(N)}+\dots +\frac{h^ka^{(k)}(N)}{k!}+o_N(1),
\end{equation*}which implies that the function $a(N+h)$ is essentially a polynomial in $h$.

The final lemma implies that if the function $L(t)$ satisfies certain growth assumptions, then a strongly non-polynomial function $a(t)$ will be approximated by a polynomial of some degree $k$.

\begin{proposition}\label{P: Taylor expansion main}
Let $a\in \mathcal{H}$ be a strongly non-polynomial function of polynomial growth, such that $a(t)\succ \log t$. Assume that $L(t)$ is a positive sub-linear function, such that $1\prec L(t)\ll t^{1-\e}$ for some $\e>0$. Then, there exists a non-negative integer $k$ depending on the function $a(t)$ and $L(t)$, such that \begin{equation*}
   \bigabs{a^{(k)}(t)}^{-\frac{1}{k}}\prec L(t)\prec \bigabs{a^{(k+1)}(t)}^{-\frac{1}{k+1}},
\end{equation*}where we adopt the convention that $\bigabs{a^{(k)}(t)}^{-\frac{1}{k}}$ denotes the constant function 1, when $k=0$.
\end{proposition}

\begin{proof}

We split the proof into two cases depending on whether $a$ is sub-fractional or not.

Assume first that $a(t)\ll t^{\delta}$ for all $\delta>0$. We will establish the claim for $k=0$. This means that functions that are sub-fractional become essentially constant when restricted to intervals of the form $[N, N+L(N)]$. The left inequality is obvious. Furthermore, since $a(t)\prec t^{\e}$, Lemma \ref{L: Frantzikinakis growth inequalities} implies that $$a'(t)\prec \frac{1}{t^{1-\e}}\ll \frac{1}{L(t)},$$ which yields the desired result. 

Assume now that $a(t)\succ t^{\delta}$ for some $\delta>0$. Observe that, in this case, we have that \begin{equation*}
      \bigabs{a^{(k)}(t)}^{-\frac{1}{k}} \prec \bigabs{a^{(k+1)}(t)}^{-\frac{1}{k+1}}
\end{equation*}for $k$ large enough, due to Proposition \ref{P: Taylor expansion-preliminary}.
We also consider the integer $d$, such that $t^d\prec a(t)\prec t^{d+1}$. This number exists because the function $a$ is strongly non-polynomial.

If $L(t)\prec \bigabs{a^{(d+1)}(t)}^{-\frac{1}{d+1}}$, then the claim holds for $k={d}$, since $\bigabs{a^{(d)}(t)}^{-\frac{1}{d}}\prec 1\prec L(t)$. 

It suffices to show that there exists $k\in \N$, such that $L(t)\prec \bigabs{a^{(k+1)}(t)}^{-\frac{1}{k+1}}$, which, in turn, follows if we show that \begin{equation}\label{E: large derivatives surpass sub-linear powers}
        t^{1-\e}\prec \bigabs{a^{(k+1)}(t)}^{-\frac{1}{k+1}}
    \end{equation}for some $k\in \N$. We can rewrite the above inequality as $a^{(k+1)}(t)\prec t^{(k+1)(\e-1)}$. However, since the function $a(t)$ is strongly non-polynomial and $a(t)\succ \log t$, the functions $a^{(k)}(t)$ satisfy the hypotheses of Lemma \ref{L: Frantzikinakis growth inequalities} (see also Remark \ref{R: why we can use Lemma 4.1 for derivatives}). 
   Therefore, iterating the aforementioned lemma, we deduce that $$a^{(k+1)}(t)\ll \frac{a(t)}{t^{k+1}}.$$
    Hence, it suffices to find $k$ such that $a(t)\ll t^{(k+1)\e}$ and such a number exists, because the function $a(t)$ has polynomial growth.
\end{proof}

\begin{remark}
    The condition $L(t)\prec t^{1-\e}$ is necessary. For example, if $a(t)=t\log t$ and $L(t)=\frac{t}{\log t}$, then for any $k\in \N$, we can write \begin{equation*}
        (N+h)\log (N+h)=N\log N+\dots +\frac{C_1h^k}{N^{k-1}} +\frac{C_2h^{k+1}}{\xi_{N,h}^k}
    \end{equation*}for every $0\leq h\leq \frac{N}{\log N}$ and some numbers $C_1,C_2\in \R$. However, there is no positive integer $k$ for which the last term in this expansion can be made to be negligible since $\frac{N}{\log N}\succ N^{\frac{k}{k+1}}$ for all $k\in \N$. Essentially, in order to approximate the function $t\log t$ in these specific short intervals, one would be forced to use the entire Taylor series instead of some appropriate cutoff.  
\end{remark}

\subsection{Eliminating the error terms in the approximations}

In the previous subsection, we saw that any Hardy field function can be approximated by polynomials in short intervals using the Taylor expansion. Namely, if $a(t)$ diverges and $L(t)\to+\infty$ is a positive function, such that \begin{equation}\label{E: ti na valw twra edw?}
      \bigabs{a^{(k)}(t)}^{-\frac{1}{k}}\prec L(t)\prec \bigabs{a^{(k+1)}(t)}^{-\frac{1}{k+1}}
\end{equation}then, for any $0\leq h\leq L(N)$, we have \begin{equation*}
    a(N+h)= a(N+h)=a(N) +\dots +\frac{h^ka^{(k)}(N)}{k!}+\frac{h^{k+1}a^{(k+1)} (\xi_{N,h}) }{(k+1)!}=p_N(h)+\theta_N(h)
\end{equation*}$  \text{for some } \xi_{N,h}\in [N,N+h]$, where we denote $$p_N(h)=a(N) +\dots +\frac{h^ka^{(k)}(N)}{k!}.$$
Observe that our growth assumption on $L(t)$ implies that the term $\theta_N(h)$ is bounded by a quantity that converges to 0, as $N\to+\infty$. Therefore, for large values of $N$, we easily deduce that $$\floor{a(N+h)}=\floor{p_N(h)}+\e_{N,h},$$where $\e_{N,h}\in \{-1,0,1\}$. In order to be able to apply Proposition \ref{P: Gowers norm bound on variable polynomials}, we will need to eliminate the error terms $\e_{N,h}$. We will consider three distinct cases, which are tackled using somewhat different arguments.

\subsubsection{The case of fast-growing functions}
Firstly, we establish the main proposition that will allow us to remove the error terms in the case of functions that contain a "non-polynomial part" which does not grow too slowly. We will need a slight strengthening of the growth conditions in \eqref{E: ti na valw twra edw?}, which, as we saw previously, are sufficient to have a Taylor approximation in the interval $[N, N+L(N)]$.

\begin{proposition}\label{P: remove error term for fast functions}
    Let $A>0$ and let $a(t)$ be a $C^{\infty}$ function defined for all sufficiently large $t\in \R$. Assume $L(t)$ is a positive sub-linear function going to infinity and let $k$ be a positive integer, such that   \begin{equation}\label{E: strong domination conditions}
       1\lll  \bigabs{a^{(k)}(t)}^{-\frac{1}{k}}\lll L(t)\lll \bigabs{a^{(k+1)}(t)}^{-\frac{1}{k+1}} 
    \end{equation}and such that the function $a^{(k+1)}(t)$ converges to 0 monotonically.
     Then, for $N$ large enough, we have that, for all $0\leq c\leq d<1$, \begin{equation}\label{E: discrepancy of Hardy sequences on short intervals}
       \dfrac{\big|\{n\in [N, N+L(N)]\colon a(n)\in [c,d]\}\big|}{L(N)}=|d-c|+O_A(L(N)\log^{-A} N).\footnote{One can actually get a small power saving here, with an exponent that depends on $k$ and the implicit fractional powers in the growth relations of \eqref{E: strong domination conditions}, though this will not be any more useful for our purposes.}
    \end{equation}
    Consequently, for all $N$ sufficiently large, we have that \begin{equation*}
        \floor{a(N+h)} =\floor{ a(N) +{ha'(N)}+\dots +\frac{h^ka^{(k)}(N)}{k!}   }
    \end{equation*}for all, except at most $O_A(L(N)\log^{-A}(N))$ values of integers $h\in [N,N+L(N)]$.
\end{proposition}



\begin{proof}

Our hypothesis on $L(t)$ implies that there exist $\e_1,\e_2>0$ such that \begin{equation}\label{E: powersavinggrowthinequality}
L(t)\bigabs{a^{(k)}(t)}^{\frac{1}{k}} \gg t^{\e_1} \text{ and } \ L(t)\bigabs{a^{(k+1)}(t)}^{\frac{1}{k+1}}\ll t^{-\e_2}. 
\end{equation}In addition, the leftmost inequality implies that there exists $\e_3>0$, such that $a^{(k)}(t)\ll t^{-\e_3}$. 
 Using the Taylor expansion around the point $N$, we can write \begin{equation}\label{E: Taylor approximation for C apeiro function}
        a(N+h)= a(N) +{ha'(N)}+\dots +\frac{h^ka^{(k)}(N)}{k!}  +\frac{h^{k+1}a^{(k+1)}(\xi_h)  }{(k+1)!},\  \text{ for some }  \xi_h\in [N,N+h],
    \end{equation}for every $h\in [0,L(N)]$. We denote $$p_N(h)=a(N) +\dots +\frac{h^ka^{(k)}(N)}{k!}$$ and $$\theta_N(h)=\frac{h^{k+1}a^{(k+1)}(\xi_h)  }{(k+1)!}.$$
The function $a^{(k+1)}(t)$ converges to 0 monotonically due to our hypothesis. Therefore, for sufficiently large $N$, \begin{equation}\label{E: definition of theta_N}
    \max_{0\leq h\leq L(N)}^{} |\theta_N(h) |\leq \Bigabs{\frac{a^{(k+1)}(N)}{(k+1)!} }(L(N))^{k+1}=\theta_N,
\end{equation}and the quantity $\theta_N$ is strongly dominated by the constant 1 due to \eqref{E: powersavinggrowthinequality}. More precisely, we have that $\theta_N\ll N^{-(k+1)\e_2}$.

Let $A>0$ be any constant. We study the discrepancy of the finite polynomial sequence $$p_N(h),\ \text{where } 0\leq h\leq{L(N)}.$$
We shall establish that we have \begin{equation*}
    \D_{[c,d]}\big( p_N(h)  \big)\ll_A \log^{-A}N
\end{equation*}for any choice of the interval $[c,d]\subseteq [0,1]$.
To this end,
we apply Theorem \ref{T: Erdos-Turan} for the finite sequence $(p_N(h))_{0\leq h\leq L(N)}$ to deduce that \begin{equation}\label{E: discrepancy bound}
    \D_{[c,d]}\Big(\big (p_N(h)\big)_{0\leq h\leq L(N)}  \Big)\leq \frac{C}{\floor{\log^A N}}+C\sum_{m=1}^{\floor{\log^A N}} \frac{1}{m}\Bigabs{\underset{0\leq h \leq {L(N)}}{\E} \  e(mp_N(h))},
\end{equation}where $C$ is an absolute constant.
 We claim that for every $1\leq m\leq \floor{\log^A N}$, we have that \begin{equation}\label{E: log^A saving }
    \Bigabs{\underset{0\leq h \leq {L(N)}}{\E} \  e(mp_N(h))}\leq \frac{1}{\log ^A N},
\end{equation}provided that $N$ is sufficiently large.
Indeed, assume for the sake of contradiction that there exists $1\leq m_0\leq \floor{\log^A N}$, such that\begin{equation}\label{E: application of Erdos-Turan}
    \Bigabs{\underset{0\leq h\leq {L(N)}}{\E} \  e(m_0p_N(h))}>\frac{1}{\log ^A N}.
\end{equation} The leading coefficient of $m_0p_N(h)$ is equal to $$\frac{m_0a^{(k)} (N)}{k!}.$$ Then, Lemma \ref{L: Weyl-type estimate} implies that there exists a constant $C_k$ (depending only on $k$) an integer $q$ satisfying $|q|\leq \log^{C_k A} N$ and such that \begin{equation*}
    \Bignorm{q\cdot\frac{m_0a^{(k)} (N)}{k!} }_{\T} \leq \frac{ \log^{C_kA}  N}{\floor{L(N)}^k}.
\end{equation*} The number $qm_0$ is bounded in magnitude by $\log^{(C_k+1)A}(N)$, so that $$q\cdot\frac{m_0a^{(k)} (N)}{k!}\ll \log^{(C_k+1)A} N\cdot N^{-\e_3}=o_N(1).$$ Therefore, for large values of $N$, we can substitute the circle norm of the fraction in \eqref{E: application of Erdos-Turan} with the absolute value, which readily implies that \begin{equation*}
     \Bigabs{q\cdot\frac{m_0a^{(k)} (N)}{k!} }\leq \frac{ \log^{C_kA} N}{\floor{L(N)}^k}\implies \floor{L(N)}^k\big|a^{(k)}(N)\big|\leq k!\log^{C_kA} N.
\end{equation*}However, this implies that $L(t)$ cannot strongly dominate the function $\big(a^{(k)}(t)\big)^{-\frac{1}{k}}$, which is a contradiction due to our hypothesis.

We have established that for every $1\leq m\leq \floor{\log^A N}$ and large $N$, inequality \eqref{E: log^A saving } holds. Substituting this in \eqref{E: discrepancy bound}, we deduce that \begin{equation*}
    \D_{[c,d]}\Big(\big(p_N(h)\big)_{0\leq h\leq {L(N)}}\Big)\leq  \frac{C}{\floor{\log^A N}}+C\sum_{m=1}^{\floor{\log^A N}} \frac{1}{m\log^A N},
\end{equation*}which implies that \begin{equation*}
     \D_{[c,d]}\Big(\big(p_N(h)\big)_{0\leq h\leq {L(N)}}\Big)\ll \frac{A\log \log N}{\log^A N}.
\end{equation*}In particular, since $A$ was arbitrary, we get \begin{equation}\label{E: logarithmic power savings}
    \D_{[c,d]}\Big(\big(p_N(h)\big)_{0\leq h\leq {L(N)}}\Big)\ll_A \frac{1}{\log^A N}.
 \end{equation}
This establishes the first part of the proposition.

The second part of our statement follows from an application of the bound on the discrepancy of the finite polynomial sequence $(p_N(h))$.
Indeed, we consider the set $$S_N=[0,\theta_N]\cup [1-\theta_N,1),$$where we recall that $\theta_N$ was defined in \eqref{E: definition of theta_N} and decays faster than a small fractional power.  Then, if $\{p_N(h)\}\notin S_N$, we have $\floor{p_N(h) +\theta_N(h)}=\floor{p_N(h)}$, as can be seen by noticing that the error term in \eqref{E: Taylor approximation for C apeiro function} is bounded in magnitude by $\theta_N$.  Now, we estimate the number of integers $h\in [0, L(N)]$ for which $\{p_N(h)\}\in S_N$. 

    Using the definition of discrepancy and the recently established bounds, we deduce that \begin{equation*}
        \dfrac{\bigabs{\{h\in [0,{L(N)}]\colon \{p_N(h)\}\in [0,\theta_N]  \} }}{{L(N)}{}}  -\theta_N\ll_A \frac{1}{\log^A N}
    \end{equation*}for every $A>0$. Since the number $\theta_N$ is dominated by $ N^{-(k+1)\e_2}$, this implies that \begin{equation*}
         \bigabs{\{h\in [0,L(N)]\colon \{p_N(h)\}\in [0,\theta_N]  \} }\ll_A \frac{L(N)}{\log ^A N}.
    \end{equation*} An entirely similar argument yields the analogous relation for the interval $[1-\theta_N, 1)$. Therefore, the number of integers in $[0,L(N)]$ for which $\{p_N(h)\}\in S_N$ is at most $O_A({L(N)}\log^{-A} N)$.

In conclusion, since $\floor{a(N+h)}=\floor{p_N(h)}$ for all integers not in $S_N$, we have that the number of integers which does not satisfy this last relation is $O_A(L(N)\log ^{-A} N)$, which yields the desired result.
\end{proof}

The above proposition asserts that, for almost all values of $h\in [0,L(N)]$, we can write $\floor{a(N+h)}=\floor{p_N(h)}$. The logarithmic power saving in the statement will be helpful since we are dealing with averages weighted by the sequence $\La_{w,b}(n)-1$, which has size comparable to $\log N$ on the interval $[N, N+L(N)]$. Furthermore, notice that we did not assume that $a$ is a Hardy field function in the proof. Thus, the conditions in this proposition can be used to prove a comparison result for more general iterates.

\subsubsection{The case of slow functions}
Unfortunately, the previous proposition cannot deal with functions whose only possible Taylor approximations involve only a constant term. This case will emerge when we have sub-fractional functions (see Definition \ref{D: growthdefinitions}) since, as we have already remarked, these functions have a polynomial approximation of degree 0 in short intervals (assuming that $L(t)\ll t^{1-\e}$).
To cover this case, we will need the following proposition which is practically of a qualitative nature.

\begin{proposition}\label{P: remove error term for slow functions}
    Let $a(t)\in \mathcal{H}$ be a sub-fractional function such that $a(t)\succ \log t$. Assume $L(t)$ is a positive sub-linear function going to infinity and such that $L(t)\ll t^{1-\delta}$, for some $\delta>0$. Then, for every $0<\e<1$, we have the following: for all $R\in \N$ sufficiently large  we have $\floor{a(N+h)}=\floor{a(N)}$ for every $h\in [0,L(N)]$, for all, except at most $\e R$ values of $N\in [1,R]$.
\end{proposition}


\begin{proof}
Observe that for any $h\in [0,L(N)]$, we have \begin{equation}\label{E: approximation of subfractional}
    a(N+h)=a(N)+ha'(\xi_h)
\end{equation}for some $\xi_h\in [N,N+h]$. In addition, since $a'(t)$ converges to 0 monotonically, we have $$|ha'(\xi_h)|\leq L(N)a'(N)\ll N^{1-\delta} a'(N) \lll 1,$$
where the last inequality follows from Lemma \ref{L: Frantzikinakis growth inequalities} and the assumption that $a(t)$ is sub-fractional.
  In particular, there exists a positive real number $q$, such that $|ha'(\xi_h)|\ll N^{-q}$, for all $h\in [0, L(N)]$.\footnote{We do not actually need this quantity to converge to zero faster than some power of $N$. The same argument applies if this quantity simply converges to zero.} 

   The sequence $a(n)$ is equidistributed mod 1 by Theorem \ref{T: Boshernitzan}, since it dominates the function $\log t$. Now, suppose that $\e>0$, and choose a number $R_0$ such that ${R_0^{-2q}}<\e/2$. Then, for $R\geq R_0$, the number of integers $N\in [R_0,R]$ such that $\{a(N)\}\in [\frac{\e}{2},1-\frac{\e}{2}]$ is $$(R-R_0)(1-\e+o_R(1))$$ due to the fact that $a(n)$ is equidistributed. For these values of $N$, we have that $$\{a(N)\}\notin [0,N^{-2q}]\cup [1-N^{-2q}, 1],$$ which implies that for all $h\in [0,L(N)]$, we have that $\floor{a(N+h)}=\floor{a(N)}$, as can be derived easily by \eqref{E: approximation of subfractional} and the fact that the error term is $O(N^{-q})$. If we consider the integers $N$ in the interval $[1,R_0]$ as well, then the number of ``bad values'' (that is, the numbers $N$ for which we do not have $\floor{a(N+h)}=\floor{a(N)}$ for every $h\in [0,L(N)]$) is at most $$R_0+(R-R_0)(\e+o_R(1)).$$ 
    Finally, choosing $R$ sufficiently large, we get that this number is smaller than $2\e R$ and the claim follows.
\end{proof}

In simplistic terms, what we have established is that if we restrict our attention to short intervals $[N, N+L(N)]$ for the natural numbers $N$, such that $\{a(N)\}\in [\e,1-\e]$, then we can just write $\floor{a(N+h)}=\floor{a(N)}$ for all $h\in [0, L(N)]$. Due to the equidistribution of $a(n)$ mod 1 (which follows from Theorem \ref{T: Boshernitzan}), this is practically true for almost all $N$, if we take $\e$ sufficiently small.

\subsubsection{The case of polynomial functions}

The final case is the case of functions of the form $p(t)+x(t)$, where $p$ is a polynomial
with real coefficients and $x(t)$ is a sub-fractional function. The equidistribution of the corresponding sequence will be affected only by the polynomial $p$ when restricted to short intervals. Nonetheless, the techniques of Proposition \ref{P: remove error term for fast functions} cannot be employed, because we cannot establish quantitative bounds on the exponential sums uniformly over all real polynomials. 
Therefore, we will use the following proposition, which allows us to calculate the integer parts in this case. Unlike the previous two propositions which can be bootstrapped to give a similar statement for several functions, we establish this one for several functions from the outset. We do not need to concern ourselves with rational polynomials, since these can be trivially reduced to the case of integer polynomials by passing to arithmetic progressions.

\begin{proposition}\label{P: remove error terms for polynomial functions}
 Let $k,d$ be positive integers, let $0<\e<1/2$ be a real number and let $w\in \N$. We define $W=\prod_{p\in \P\colon p\leq w}p$ and let $1\leq b\leq W$ be any integer with $(b,W)=1$.
    Suppose that $a_1,\dots, a_k\in \mathcal{H}$ are functions of the form $p_i(t)+x_i(t)$, where $p_i$ are polynomials of degree at most $d$ and with at least one irrational  non-constant coefficient, while $x_i(t)$ are sub-fractional functions.
    Finally, assume that $L(t)$ is a positive sub-linear function going to infinity and such that \begin{equation*}
        t^{\frac{5}{8}}\lll L(t)\lll t.\footnote{See the notational conventions for the definition of $\lll$.}
    \end{equation*}
    
    Then, for every $r$ sufficiently large in terms of $w$, $\frac{1}{\e}$, we have that there exists a subset $\mathcal{B}_{r,\e}$ of integers in the interval $[r,r+L(r)]$ with at most $O_k(\e L(r))$ elements, such that for all integers $n\in [r,r+L(r)]\setminus\mathcal{B}_{r,\e}$, we have   \begin{equation*}
    \floor{p_i(n)+x_i(n)}=\floor{p_i(n)+x_i(r)}.
\end{equation*}
Furthermore, the set $\mathcal{B}_{r,\e}$ satisfies \begin{equation}\label{E: bound on weighted sums of B_r}
   \frac{1}{L(r)} \sum_{r\leq n\leq r+L(r)} \La_{w,b}(n){\bf 1}_{\mathcal{B}_{r,\e}}(n)\ll_{k,d} \e +o_w(1)\log \frac{1}{\e}+o_r(1).
\end{equation}
\end{proposition}

\begin{remark}
    The $o_r(1)$ term depends on the fixed parameters $w,\e$. However, in our applications, we will send $r\to+\infty$, then we will send $w\to+\infty$,  and then $\e\to 0$. We shall reiterate this observation in the proof of Theorem \ref{T: the main comparison}.
    On the other hand, the $o_w(1)$ term is the same as the one in \ref{L: discorrelation of W-tricked with polynomial phases} and depends on the degree $d$ of the polynomials, which will be fixed in applications. 
\end{remark}

\begin{proof}[Proof of Proposition \ref{P: remove error terms for polynomial functions}]
    Fix an index $1\leq i\leq k$ and consider a sufficiently large integer $r$. Using the mean value theorem and the fact that $|x_i'(t)|$ decreases to 0 faster than all fractional powers by Lemma \ref{L: Frantzikinakis growth inequalities}, we deduce that \begin{equation*}
\max_{0\leq h\leq L(r)} |x_{i}(r+h)-x_i(r)|\leq L(r)|x'_i(r)|\lll 1.
\end{equation*}In particular, there exists $\delta_0>0$ depending only on the functions $a_1,\dots, a_k$ and $L(t)$, such that \begin{equation}\label{E: locally constant function}
    \max_{0\leq h\leq L(r)} |x_{i}(r+h)-x_i(r)|\ll r^{-\delta_0}
\end{equation}for all $1\leq i\leq k$. Thus, we observe that if $\{p_i(n)+x_i(r)\}\in (\e,1-\e)$ and $r$ is large enough in terms of $1/\e$, then we have that \begin{equation*}
    \floor{p_i(n)+x_i(n)}=\floor{p_i(n)+x_i(r)}.
\end{equation*}Naturally, we consider the set \begin{equation}
    \mathcal{B}_{i,r,\e}=\{n\in [r,r+L(r)]\colon \{p_i(n)+x_i(r)\}\in [0,\e]\cup [1-\e,1)\}
\end{equation}and take $\mathcal{B}_{r,\e}=\mathcal{B}_{1,r,\e}\cup\dots\cup\mathcal{B}_{k,r,\e} $.
Now, we observe that the polynomial sequence $p_i$ is well-distributed modulo 1, since it has at least one non-constant irrational coefficient. Therefore, if $r$ is large enough, we have that the set $\mathcal{B}_{i,r,\e}$ has less than $3\e L(r)$ elements (say). Using the union bound, we conclude that the set $\mathcal{B}_{r,\e}$ has $O(\e k L(r))$ elements.
This shows the first requirement of the proposition.

We have to establish \eqref{E: bound on weighted sums of B_r}. We shall set $M=\floor{\e^{-1}}$ for brevity so that $r$ is assumed to be very large in terms of $M$.
Since the polynomials $p_i$ have at least one non-constant irrational coefficient, we can use Weyl's criterion for well-distribution (see, for instance, \cite[Theorem 5.2, Chapter 1]{Kuipers-Niederreiter}) to conclude that \begin{equation*}
    \max_{1\leq m\leq M} \Bigabs{\E_{r\leq n\leq r+L(r)}  e\big(m(p_i(n)+x_i(r)  ) \big)}=o_r(1),
\end{equation*}for all $r$ sufficiently large in terms of $M$, as we have assumed to be the case.\footnote{ A bound that is uniform over all $m\in \N$ is in general false, so we have to restrict $m$ to a finite range.} On the other hand, Lemma \ref{L: discorrelation of W-tricked with polynomial phases} implies that \begin{equation*}
    \max_{1\leq m\leq M} \Bigabs{ \E_{r\leq n\leq r+L(r)} \big(\La_{w,b}(n)-1 \big) e\big(m(p_i(n)+x_i(r)  ) \big)}=o_w(1)
\end{equation*}for $r$ sufficiently large in terms of $w$. Combining the last two bounds, we deduce that 
\begin{equation}\label{E: W-tricked Lambda exponential sums}
     \max_{1\leq m\leq M} \Bigabs{ \E_{r\leq n\leq r+L(r)} \La_{w,b}(n) e\big(m(p_i(n)+x_i(r)  ) \big)}=o_w(1)+o_r(1).
\end{equation}
Since we have estimates on the exponential sums weighted by $\La_{w,b}(n)$, we can now make the passage to \eqref{E: bound on weighted sums of B_r}. To this end, we apply Theorem \ref{T: Erdos-Turan} for the probability measure \begin{equation*}
    \nu(S)=\frac{\sum\limits_{r\leq n\leq r+L(r)} \La_{w,b}(n)\delta_{\{p_i(n)+x_i(r)\}}(S)   }{\sum\limits_{r\leq n\leq r+L(r)}\La_{w,b}(n)  }.\footnote{ The denominator is non-zero if $r$ is large enough.}
\end{equation*}Setting $$S_r=\sum_{r\leq n\leq r+L(r)}\La_{w,b}(n)$$ for brevity, we conclude that \begin{multline}
    \frac{\sum\limits_{r\leq n\leq r+L(r)} \La_{w,b}(n)\delta_{\{p_i(n)+x_i(r)\}}\big([0,\e]\cup[1-\e,1)\big)   }{S_r }\ll 2\e+ \frac{1}{M} +\\
    \sum_{m=1}^{M}\frac{1}{m}\Bigabs{\frac{1}{S_r}\sum_{r\leq n\leq r+L(r)}  \La_{w,b}(n) e\big(m(p_i(n)+x_i(r)  ) \big)  },
\end{multline}
where the implied constant is absolute.
Applying the bounds in \eqref{E: W-tricked Lambda exponential sums} and recalling the definition of $\mathcal{B}_{i,r,\e}$, we conclude that \begin{multline}\label{E: forza aekara}
    \sum\limits_{r\leq n\leq r+L(r)} \La_{w,b}(n){\bf 1}_{\mathcal{B}_{i,r,\e}}(n)\ll \Big(\e+\frac{1}{M}\Big)S_r+\sum_{m=1}^M \frac{L(r)}{m}(o_w(1)+o_r(1))\\
    \ll \e S_r+ L(r)\big(o_w(1)+o_r(1)\big)\log \frac{1}{\e},
\end{multline}since $M=\floor{\e^{-1}}$.
Finally, we bound $S_r$ by applying Corollary \ref{C: Brun-Titchmarsh inequality for von Mangoldt sums} to conclude that \begin{multline}\label{E: sieve upper bound on the S_r}
    S_r=\frac{\phi(W)}{W}\sum_{\underset{n\equiv b\;(W)}{Wr+b\leq n\leq Wr+b+WL(r)}} \La(n)\leq \frac{\phi(W)}{W}\Big( \frac{2WL(r)\log r}{\phi(W)\log\big(\frac{L(r)}{W}\big)} +\\
    O\Big(\frac{L(r)}{\log r}\Big)+
    O(r^{1/2}\log r)\Big)\ll L(r)(1+o_r(1)),
\end{multline}where we used the fact that $L(r)\gg t^{5/8}$ to bound the first fraction by an absolute constant. Applying this in \eqref{E: forza aekara}, we conclude that \begin{equation*}
    \frac{1}{L(r)} \sum_{r\leq n\leq r+L(r)} \La_{w,b}(n){\bf 1}_{\mathcal{B}_{i,r,\e}}(n)\ll \e (1+o_r(1))+\big(o_w(1)+o_r(1)\big)\log \frac{1}{\e}.
\end{equation*}

Finally, we recall that $\mathcal{B}_{r,\e}=\mathcal{B}_{1,r,\e}\cup\dots\cup\mathcal{B}_{k,r,\e}$ and use the union bound to get\begin{equation*}
    \frac{1}{L(r)} \sum_{r\leq n\leq r+L(r)} \La_{w,b}(n){\bf 1}_{\mathcal{B}_{r,\e}}(n)\ll_k \e++o_w(1)\log \frac{1}{\e}+o_r(1),
\end{equation*}provided that $r$ is very large in terms of $1/\e, w$. This is the desired conclusion.
\end{proof}

\subsection{Simultaneous approximation of Hardy field functions}

In view of Proposition~\ref{P: remove error term for fast functions}, we would like to show that we can find a function $L(t)$ such that the growth rate condition of the statement is satisfied for several functions in $\mathcal{H}$ simultaneously.
This is the content of the following lemma. We will only need to consider the case where the functions dominate some fractional power, since for sub-fractional functions, we have Propositions \ref{P: remove error term for slow functions} and \ref{P: remove error terms for polynomial functions} that can cover them adequately. We refer again to our notational conventions in Section \ref{Section-Introduction} for the notation $\lll$.

\begin{proposition}\label{P: Taylor expansion final000}
    Let $\ell\in \N$  and suppose $a_1,\dots,a_{\ell}\in \mathcal{H}$ are strongly non-polynomial functions of polynomial growth that are not sub-fractional. Then, for all $0<c<1$, there exists a positive sub-linear function $L(t)$, such that $t^c\ll L(t)\ll t^{1-\e}$ for some $\e>0$ and such that, for all $1\leq i\leq \ell,$ there exist positive integers $k_i$, which satisfy \begin{equation*}
   1\lll \bigabs{a_i^{(k_i)}(t)}^{-\frac{1}{k_i}}\lll L(t)\lll \bigabs{a_i^{(k_i+1)}(t)}^{-\frac{1}{k_i+1}}.
\end{equation*}
 Furthermore, the integers $k_i$ can be chosen to be arbitrarily large, provided that $c$ is sufficiently close to 1.
\end{proposition}

\begin{proof}
    We will use induction on $\ell$. For $\ell=1$, it suffices to show that there exists a positive integer $k$, such that  the function $\bigabs{a^{(k+1)}(t)}^{-\frac{1}{k+1}}$ strongly dominates the function $\bigabs{a^{(k)}(t)}^{-\frac{1}{k}}$. Then, we can pick the function $L(t)$ to be the geometric mean of these two functions to get our claim.\footnote{It is straightforward to check that if $f\lll g$, then $f\lll \sqrt{fg}\lll g$, assuming, of course, that the square root is well-defined (e.g. when the functions $f,g$ are eventually positive).}

Firstly, note that if we pick $k$ sufficiently large, then we can ensure that  $(a^{(k)}(t))^{-\frac{1}{k}}\gg t^{c}$, which would also imply the lower bound on the other condition imposed on the function $L(t)$.
To see why this last claim is valid, observe that the derivatives of $a$ satisfy the assumptions of Lemma \ref{L: Frantzikinakis growth inequalities}, so that we have $a^{(k)}(t)\ll t^{-k}a(t) $. Thus, if $d$ is a positive integer, such that $t^{d}$ grows faster than $a(t)$ and we choose $k>\frac{d}{c}-1$, we verify that our claim holds. 

Secondly, we will show that for all $k\in \N$, we have $$\bigabs{a^{(k)}(t)}^{-\frac{1}{k}}\ll t^{1-\e}$$ for some $0<\e<1$, as this relation (with $k+1$ in place of $k$) will yield the upper bound on the growth of the function $L(t)$ that we chose above. For the sake of contradiction, we assume that this fails and use the lower bound from Lemma \ref{L: Frantzikinakis growth inequalities}, to deduce that \begin{equation*}
    t^{k(\e-1)}\gg a^{(k)}(t) \succ \frac{a(t)}{t^{k}\log^{2k} t}
\end{equation*}for every $0<\e<1$. This, implies that $a(t)\ll t^{k\e}\log^{2k}t$ for all small $\e$, which contradicts the hypothesis that $a(t)$ is not sub-fractional.
 We remark in passing that this argument also indicates that the integer $k$ can be made arbitrarily large by choosing $c$ to be sufficiently close to $1$, as the last claim in our statement suggests.

In order to complete the base case of the induction, we show that for all sufficiently large $k$, we have $$\bigabs{a^{(k)}(t)}^{-\frac{1}{k}}\lll\bigabs{a^{(k)}(t)}^{-\frac{1}{k+1}}.$$
Equivalently, we prove that \begin{equation}\label{E: fractionally away}
        \frac{ \bigabs{a^{(k+1)}(t)}^{-\frac{1}{k+1}}}{ \bigabs{a^{(k)}(t)}^{-\frac{1}{k}}}\gg t^{\delta}
    \end{equation}for some $\delta>0$ that will depend on $k$. Choose a real number $0<q<1$ (the value of $q$ depends on $k$), such that  $\bigabs{a^{(k)}(t)}^{-\frac{1}{k}}\ll t^{1-q}$, which can be done as we demonstrated above.
In order to establish \eqref{E: fractionally away}, we combine the inequality $a^{(k)}(t)\gg ta^{(k+1)}(t)$ with the inequality $\bigabs{a^{(k)}(t)}^{-\frac{1}{k}}\ll t^{1-q}$, which after some computations gives the desired result for $\delta=q/(k+1)$. This completes the base case.

    Assume  that the claim has been established for the integer $\ell$. Now, let $a_1,\dots, a_{\ell+1}$ be functions that satisfy the hypotheses of the proposition. Our induction hypothesis implies that there exists a function $L(t)$ with $t^c \ll L(t)\ll t^{1-\e}$  and integers $k_1,\dots,k_{\ell}$, such that \begin{equation*}
        \bigabs{a_i^{(k_i)}(t)}^{-\frac{1}{k_i}}\lll L(t)\lll \bigabs{a_i^{(k_i+1)}(t)}^{-\frac{1}{k_i+1}},\;\;1\leq i\leq \ell.
    \end{equation*} Due to Proposition \ref{P: Taylor expansion main}, there exists a positive integer $s$, such that \begin{equation}\label{E: prec equation}
        \bigabs{a_{\ell+1}^{(s)}(t)}^{-\frac{1}{s}}\prec L(t)\prec \bigabs{a_{\ell+1}^{(s+1)}(t)}^{-\frac{1}{s+1}}.
    \end{equation}
    Without loss of generality, we may assume that $c$ is sufficiently close to 1. This implies that the integer $s$ can be chosen to be sufficiently large as well, so that the relation $\bigabs{a_{\ell+1}^{(s)}(t)}^{-\frac{1}{s}}\lll \bigabs{a_{\ell+1}^{(s+1)}(t)}^{-\frac{1}{s+1}}$ holds, as we established in the base case of the induction.

    If each function strongly dominates the preceding one in \eqref{E: prec equation}, then we are finished. Therefore, assume that $L(t)$ is not strongly dominated by the function $\bigabs{a_{\ell+1}^{(s+1)}(t)}^{-\frac{1}{s+1}}$ (the other case is similar). Note that for every $1\leq i\leq \ell$, we have that \begin{equation*}
    \bigabs{a_i^{(k_i)}(t)}^{-\frac{1}{k_i}}    \lll \bigabs{a_{\ell+1}^{(s+1)}(t)}^{-\frac{1}{s+1}}.
    \end{equation*}Indeed, since the function $L(t)$ strongly dominates the function $\bigabs{a_i^{(k_i)}(t)}^{-\frac{1}{k_i}}$ (by the induction hypothesis) and $L(t)$ grows slower than the the function $ \bigabs{a_{\ell+1}^{(s+1)}(t)}^{-\frac{1}{s+1}}$, this claim follows immediately. Among the functions $a_1,\dots,a_{\ell+1}$, we choose a function for which the growth rate of $\bigabs{a_i^{(k_i)}(t)}^{-\frac{1}{k_i}} $ is maximized.\footnote{ In the case $i=\ell+1$, we are referring to the function $\bigabs{a_i^{(s)}(t)}^{-\frac{1}{s}}$.} Assume that this happens for the index $i_0\in \{1,\dots,\ell+1\}$ and observe that the function $\bigabs{a_{\ell+1}^{(s+1)}(t)}^{-\frac{1}{s+1}}$ strongly dominates $\bigabs{a_{i_0}^{(k_{i_0})}(t)}^{-\frac{1}{k_{i_0}}} $, because the first function grows faster than $L(t)$ and $L(t)$ strongly dominates the latter (in the case $i_0={\ell+1}$, this follows from the fact that $\bigabs{a_{\ell+1}^{(s)}(t)}^{-\frac{1}{s}}\lll \bigabs{a_{\ell+1}^{(s+1)}(t)}^{-\frac{1}{s+1}}$).
    
    Define the function ${\wt{L}(t)}$ to be the geometric mean of the functions $\bigabs{a_{i_0}^{(k_{i_0})}(t)}^{-\frac{1}{k_{i_0}}} $ and $\bigabs{a_{\ell+1}^{(s+1)}(t)}^{-\frac{1}{s+1}} $. 
     Observe that this function grows slower than the function $L(t)$, since it is strongly dominated by the function $\bigabs{a_{\ell+1}^{(s+1)}(t)}^{-\frac{1}{s+1}},$ while the original function $L(t)$ is not. Due to its construction, we deduce that the function ${\wt{L}(t)}$ satisfies \begin{equation*}
         \bigabs{a_{\ell+1}^{(s)}(t)}^{-\frac{1}{s}}\lll {\wt{L}(t)}\lll \bigabs{a_{\ell+1}^{(s+1)}(t)}^{-\frac{1}{s+1}}
    \end{equation*}and \begin{equation*}
         \bigabs{a_i^{(k_i)}(t)}^{-\frac{1}{k_i}}\lll {\wt{L}(t)}
    \end{equation*}for all $1\leq i\leq \ell$. This is a simple consequence of the fact that $\wt{L}(t)$ strongly dominates the function $\bigabs{a_{i_0}^{(k_{i_0})}(t)}^{-\frac{1}{k_{i_0}}} $ and the index $i_0$ was chosen so that the growth rate of the associated function is maximized.
    In addition, the function $L(t)$ grows faster than the function ${\wt{L}(t)}$, which implies that \begin{equation*}
        {\wt{L}(t)}\prec L(t) \lll \bigabs{a_i^{(k_i+1)}(t)}^{-\frac{1}{k_i+1}}
    \end{equation*}for all $1\leq i\leq \ell$. The analogous relation in the case $i=\ell+1$ is also correct, as we pointed out previously. Therefore, the function ${\wt{L}(t)}$ satisfies all of our required properties and the induction is complete.

    Finally, the assertion that the integers $k_i$ can be made arbitrarily large follows by enlarging $c$ appropriately and the fact that given a fixed $k_i\in \N$, the function $\bigabs{a_i^{(k_i+1)}(t)}^{-\frac{1}{k_i+1}}$ cannot dominate all powers $t^c$ with $c<1$, as we displayed in the base case of the induction.
 \end{proof}

We can actually weaken the hypothesis that the functions are strongly non-polynomial. The following proposition is more convenient to use and its proof is an immediate consequence of Proposition \ref{P: Taylor expansion final000}.
\begin{proposition}\label{P: Taylor expansion final}
      Let $\ell\in \N$ and suppose $a_1,\dots,a_{\ell}\in \mathcal{H}$ are  functions of polynomial growth, such that $|a_i(t)-p(t)|\ggg 1$, for all real polynomials $p(t)$ and every $i\in \{1,\dots, \ell\}$. Then, for all $0<c<1$, there exists a positive sub-linear function $L(t)$, such that $t^c\prec L(t)\ll t^{1-\e}$ for some $\e>0$ and such that there exist positive integers $k_i$, which satisfy \begin{equation*}
   1\lll \bigabs{a_i^{(k_i)}(t)}^{-\frac{1}{k_i}}\lll L(t)\lll \bigabs{a_i^{(k_i+1)}(t)}^{-\frac{1}{k_i+1}}.
\end{equation*}
\end{proposition}

\begin{proof}
    Each of the functions $a_i$ can be written in the form $p_i(t)+x_i(t)$, where $p_i$ is a polynomial with real coefficients and $x_i\in \mathcal{H}$ is strongly non-polynomial. The hypothesis implies that the functions $x_i$ are not sub-fractional.
    If $k$ is large enough, then we have $a_i^{(k)}(t)=x_i^{(k)}(t)$ for all $t\in \R$. The conclusion follows from Proposition \ref{P: Taylor expansion final000} applied to the functions $x_i(t)$, where the corresponding integers $k_i$ are chosen large enough so that the equality $a_i^{(k_i)}(t)=x_i^{(k_i)}(t)$ holds.
\end{proof}

\section{The main comparison}\label{Section-The main comparison}

In this section, we will establish the main proposition that asserts that averages weighted by the W-tricked von-Mangoldt function are morally equal to the standard Ces\`{a}ro averages over $\N$. In order to do this, we will use the polynomial approximations for our Hardy field functions and we will try to remove the error terms arising from these approximations using Propositions \ref{P: remove error term for fast functions}, \ref{P: remove error term for slow functions} and \ref{P: remove error terms for polynomial functions}. Firstly, we will use a lemma that allows us to pass from long averages over the interval $[1,N]$ to shorter averages over intervals of the form $[N,N+L(N)]$. This lemma is similar to \cite[Lemma 3.3]{Tsinas}, the only difference being the presence of the unbounded weights.

\begin{lemma}\label{L: long averages to short averages}
    Let $(A_{n})_{n\in \N}$ be a sequence in a normed space, such that $\norm{A_{n}}\leq 1$ and let $L(t)\in\mathcal{H}$ be an (eventually) increasing sub-linear function, such that $L(t)\gg t^{\e}$ for some $\e>0$. Suppose that $w$ is a fixed natural number. Then, we have \begin{equation*}
       \Bignorm{\underset{1\leq r\leq R}{\E} \big(\La_{w,b}(r)-1\big)A_r}\leq   \underset{1\leq r\leq R}{\E}  \ \Bignorm{\underset{r\leq n\leq r+L(r)}{\E} \big(\La_{w,b}(n)-1\big)  A_n}+o_R(1),
    \end{equation*}uniformly for all $1\leq b\leq W$ with $(b,W)=1$.
\end{lemma}

\begin{proof}
    Using the triangle inequality, we deduce that \begin{equation*}
     \underset{1\leq r\leq R}{\E} \ \Bignorm{ \underset{r\leq n\leq r+L(r)}{\E}  \big(\La_{w,b}(n)-1\big)  A_{n}               }\geq \Bignorm{\underset{1\leq r\leq R}{\E}    \big(  \underset{r\leq n\leq r+L(r)}{\E}   \big(\La_{w,b}(n)-1\big) A_{n} \big) }.
 \end{equation*}Therefore, our result will follow if we show that \begin{equation*}
    \Bignorm{ \underset{1\leq r\leq R}{\E}    \Big( \underset{r\leq n\leq r+L(r)}{\E}    \big(\La_{w,b}(n)-1\big)A_{n}\Big) -\underset{1\leq r\leq R}{\E}   \big(\La_{w,b}(r)-1\big)A_{r}       }=o_R(1).
 \end{equation*}
 
 Let $u$ denote the inverse of the function $t+L(t)$, which is well-defined for sufficiently large $t$ due to monotonicity. 
 Furthermore, it is straightforward to derive that $\lim\limits_{t\to+\infty} u(t)/t=1$ from the fact that $t+L(t)$ also grows linearly. Now, we have \begin{multline*}
     \underset{1\leq r\leq R}{\E}    \Big(\   \underset{r\leq n\leq r+L(r)}{\E}    \big(\La_{w,b}(n)-1\big)A_{n} \Big) = \frac{1}{R} \Big( \sum_{n=1}^{R}p_R(n)\big(\La_{w,b}(n)-1\big)A_{n} + \\
     \sum_{n=R+1}^{R+L(R)} p_R(n)\big(\La_{w,b}(n)-1\big)A_{n}\Big)
 \end{multline*}for some real numbers $p_R(n)$, which denote the number of appearances of $A_n$ in the previous expression (weighted by the term $1/L(r)$ that appears on each inner average). Assuming that $n$ (and thus $R$) is sufficiently large, so that $u(n)$ is positive, we can calculate $p_R(n)$ to be equal to \begin{equation*}
     p_R(n)=   \frac{1}{L(\floor{u(n)})+1}+\cdots +\frac{1}{L(n)+1} +o_n(1),
 \end{equation*}since the number $A_{n}$ appears on the average $\underset{r\leq n\leq r+L(r)}{\E}$ if and only if $u(n)\leq r\leq n$. Note that $p_R(n)$ is actually independent of $R$ (for $n$ large enough) and therefore, we will denote it simply as $p(n)$ from now on. 
We have that \begin{equation}\label{E: p(n)limit}
    \lim_{n\to +\infty } p(n)=1.
\end{equation}
This follows exactly as in the proof of Lemma 3.3 in \cite{Tsinas}, so we omit its proof here.

Now, we show that \begin{equation}\label{E: to kommati poy jefeygei}
    \frac{1}{R}  \sum_{n=R+1}^{R+L(R)} p(n)\big(\La_{w,b}(n)-1\big)A_{n}=o_R(1).
\end{equation}Bounding $p(n)$ trivially by 2 (since its limit is equal to 1) and $\norm{A_n}$ by $1$, we infer that it is sufficient to show that \begin{equation*}
    \frac{1}{R} \sum_{n=R+1}^{R+L(R)} \big|\La_{w,b}(n)-1\big|=o_R(1).
\end{equation*}Using the triangle inequality and the fact that $L(r)\prec r$, this reduces to
    \begin{equation*}
    \frac{1}{R} \sum_{n=R+1}^{R+L(R)} \La_{w,b}(n)=o_R(1).
\end{equation*}To establish this, we apply Corollary \ref{C: Brun-Titchmarsh inequality for von Mangoldt sums} to conclude that \begin{multline*}
     \frac{1}{R} \sum_{n=R+1}^{R+L(R)} \frac{\phi(W)}{W}\La(Wn+b)=\frac{1}{R}\sum_{\underset{n\equiv b\;(W)}{WR+R+b\leq n\leq WR+R+b+WL(r)}} \La(n)\leq\\
     \frac{\phi(W)}{WR}\Big( \frac{2WL(R)\log R}{\phi(W)\log\big(\frac{L(R)}{W}\big)} +O\big(\frac{L(R)}{\log (WR+R+b)}\big) +O(R^{1/2}\log R)\Big)=o_R(1).
\end{multline*} This follows from the fact that $L(R)\prec R$ and that the quantity $\log R/\log(L(R))$ is bounded by the hypothesis $L(R)\gg R^{\e}$.

In view of this, it suffices to show that \begin{equation*}
    \Bignorm{ \frac{1}{R}\sum_{n=1}^{R}p(n)\big(\La_{w,b}(n)-1\big)A_{n}-\frac{1}{R}\sum_{n=1}^{R}\big(\La_{w,b}(n)-1\big)A_{n} }=o_R(1).
\end{equation*}

 We have \begin{equation*}
   \Bignorm{ \frac{1}{R}\sum_{n=1}^{R}p(n)\big(\La_{w,b}(n)-1\big)A_{n}-\frac{1}{R}\sum_{n=1}^{R}\big(\La_{w,b}(n)-1\big)A_{n} }\leq  \frac{1}{R}\sum_{n=1}^{R}|p(n)-1||\La_{w,b}(n)-1|, 
\end{equation*}by the triangle inequality. Now, given $\e>0$, we can bound this by \begin{equation*}
    \frac{1}{R}\sum_{n=1}^{R}\e\big(\La_{w,b}(n)+1\big)+o_R(1),
\end{equation*}where the $o_R(1)$ term reflects the fact that the bound for $|p(n)-1|\leq \e$ is valid for large values of $n$ only. It suffices to bound the term \begin{equation*}
     \frac{\e}{R}\sum_{n=1}^{R}\La_{w,b}(n),
\end{equation*}since the remainder is simply $O(\e)$. However, using Corollary \ref{C: Brun-Titchmarsh inequality for von Mangoldt sums} (or the prime number theorem in arithmetic progressions), we see that this term is also $O(\e)$, exactly as we did above. Sending $\e\to 0$, we reach the desired conclusion.
\end{proof}

We restate here our main theorem for convenience. 

\begin{customthm}{1.1}
    Let $\ell,k$ be positive integers and, for all $1\leq i\leq \ell,\ 1\leq j\leq k$, let $a_{ij}\in \mathcal{H}$ be functions of polynomial growth such that \begin{equation}
        |a_{ij}(t) -q(t) |\succ \log t\ \text{ for every polynomial }  q(t)\in \Q[t],
        \end{equation} or \begin{equation}
           \lim\limits_{t\to+\infty} |a_{ij}(t)-q(t)|=0\  \text{ for some polynomial }  q(t)\in \Q[t]+\R.
        \end{equation}Then, for any measure-preserving system $(X,\X,\m,T_1,\dots,T_k)$ of commuting transformations and functions $f_1,\dots,f_{\ell}\in L^{\infty}(\m)$, we have  \begin{equation}
         \lim_{w\to+\infty} \  \limsup\limits_{N\to+\infty}\max_{\underset{(b,W)=1}{1\leq b\leq W}} \Bignorm{\frac{1}{N}\sum_{n=1}^{N} \ \big(\La_{w,b}(n) -1\big) \prod_{j=1}^{\ell}\big(  \prod_{i=1}^{k} T_i^{\floor{a_{ij}(Wn+b)}}       \big)f_j   }_{L^2(\m)}=0.
        \end{equation}
\end{customthm}

\begin{proof}

    We split this reduction into several steps. For a function $a\in \mathcal{H}$, we will use the notation $a_{w,b}(t)$ to denote the function $a(Wt+b)$ and we will need to keep in mind that the asymptotic constants must not depend on $W$ and $b$. As is typical in these arguments, we shall rescale the functions $f_1,\dots, f_{\ell}$ so that they are all bounded by 1. 
\subsection*{Step 1: A preparatory decomposition of the functions}
Each function $a_{ij}$ can be written in the form \begin{equation*}
    a_{ij}(t)=g_{ij}(t)+p_{ij}(t)+q_{ij}(t)
\end{equation*}where $g_{ij}(t)$ is a strongly non-polynomial function (or identically zero), $p_{ij}(t)$ is either a polynomial with at least one non-constant irrational coefficient or a constant polynomial, and, lastly, $q_{ij}(t)$ is a polynomial with rational coefficients. Observe that there exists a fixed positive integer $Q_0$ for which all the polynomials $q_{ij}(Q_0n+s_0)$ have integer coefficients except possibly the constant term, for all $0\leq s_0\leq Q_0$. These non-integer constant terms can be absorbed into the polynomial $p_{ij}(t)$. Therefore, splitting our average into the arithmetic progressions $(Q_0n+s_0)$, it suffices to show that \begin{equation*}
     \lim_{w\to+\infty} \   \limsup\limits_{N\to+\infty}\ \max_{\underset{(b,W)=1}{1\leq b\leq W}} \Bignorm{\frac{1}{N}\sum_{n=1}^{N} \big(\La_{w,b}(Q_0n+s_0)-1\big) \prod_{j=1}^{\ell}\big(  \prod_{i=1}^{k} T_i^{\floor{a_{ij,w,b}(Q_0n+s_0)}}       \big)f_j   }_{L^2(\m)}=0
\end{equation*}for all $s_0\in\{0,\dots,Q_0-1\}$. Observe that each one of the functions $a_{ij,w,b}(Q_0t+s_0)$ satisfies either \eqref{E: far away from rational polynomials} or \eqref{E: essentially equal to a polynomial}. Since the polynomials $q_{ij,w,b}(Q_0n+s_0)$ have integer coefficients, we can rewrite the previous expression as \begin{multline}\label{E: Step 1 final expression}
     \lim_{w\to+\infty} \   \lim\limits_{N\to+\infty}\ \max_{\underset{(b,W)=1}{1\leq b\leq W}} \Bignorm{\frac{1}{N}\sum_{n=1}^{N} {\bf 1}_{s_0\;( Q_0)}(n)\big(\La_{w,b}(n)-1\big)\\
     \prod_{j=1}^{\ell}\big(  \prod_{i=1}^{k} T_i^{\floor{g_{ij,w,b}(n) +p_{ij,w,b}(n) }+q_{ij,w,b}(n)  }       \big)f_j   }_{L^2(\m)}=0.
\end{multline}

\subsection*{Step 2: Separating the iterates}

Define the sets \begin{equation}\label{E: set S_1}
  S_1=  \{(i,j)\in [1, k]\times [1, \ell] \colon g_{ij}(t)\ll t^{\delta}  \text{ for all } \delta>0 \ \text{and } p_{ij} \text{ is non-constant}\},
\end{equation}and 
\begin{equation}\label{E: set S_2}
    S_2=\{(i,j)\in [1, k]\times [1, \ell] \colon g_{ij}(t)\ll t^{\delta}  \text{ for all } \delta>0 \ \text{and } p_{ij} \text{ is constant}\},
\end{equation}whose union contains precisely the pairs $(i,j)$, for which  $g_{ij}(t)$ is sub-fractional.

Our first observation is that if a pair $(i,j)$ belongs to $S_2$, then the function $a_{ij}(t)$ has the form $g_{ij}(t)+q_{ij}(t)$, where $g_{ij}$ is sub-fractional and $q_{ij}$ is a rational polynomial. Thus, \eqref{E: far away from rational polynomials} and \eqref{E: essentially equal to a polynomial} imply that we either have that $g_{ij}(t)\succ \log(t)$ or $g_{ij}(t)$ converges to a constant, as $t\to+\infty$. The constant can be absorbed into the constant polynomial $p_{ij}$. In view of this, we will subdivide $S_2$ further into the following two sets:\begin{align}\label{E: sets S_2',S_2''}
    &S'_2=\{(i,j)\in S_2\colon g_{ij}(t)\succ \log t\},\\
  \tag*{}  &S''_2=\{(i,j)\in S_2\colon g_{ij}(t)\prec  1\}.
\end{align}

Observe that iterates corresponding to pairs $(i,j)$ that do not belong to the union $S_1\cup S'_2\cup S''_2$ have an expression inside the integer part that has the form $g(t)+p(t)$, where $g$ is a strongly non-polynomial function that is not sub-fractional. In particular, these functions satisfy the hypotheses of Proposition~\ref{P: Taylor expansion final}. Furthermore, functions that correspond to the set $S_1$ have the form $p(t)+x(t)$, where $p$ is an irrational polynomial and $x$ is sub-fractional, while functions in $S'_2$ are sub-fractional functions that dominate $\log t$. We will use Proposition~\ref{P: remove error terms for polynomial functions} and Proposition~\ref{P: remove error term for slow functions} for these two collections respectively. Finally, observe that if $(i,j)\in S_2''$, then for $n$ sufficiently large, we can write $$\floor{a_{ij}(Q_0n+s_0)}= q_{ij}(Q_0n+s_0)+\floor{c_{ij}}+e_{ij,Q_0n+s_0},$$
where $e_{ij, Q_0n+s_0}\in \{0,-1\}$ and $c_{ij}$ is a constant term arising from the constant (in this case) polynomial $p_{ij}$. The error term $e_{ij, Q_0n+s_0}$ actually exists only if $c_{ij}$ is an integer. 
In particular, we have $e_{ij,Q_0n+s_0}=0$ for all large enough $n$ when $g_{ij}(t)$ decreases to 0 and $e_{ij,Q_0n+s_0}=-1$ if $g_{ij}(t)$ increases to 0. Therefore, if we redefine the polynomials $q_{ij}(t)$ accordingly so that both $\floor{c_{ij}}$ and the error term $e_{ij, Q_0n+s_0}$ (which is independent of $s_0$) is absorbed into the constant term, we may assume without loss of generality that for all $n$ sufficiently large, we have \begin{equation*}
    \floor{g_{ij}(Q_0n+s_0)+p_{ij}(Q_0n+s_0)}+q_{ij}(Q_0n+s_0)=q_{ij}(Q_0n+s_0).
\end{equation*}We will employ this relation to simplify the iterates in \eqref{E: Step 1 final expression}, where $n$ will be replaced by $Wn+b$.

We rewrite the limit in \eqref{E: Step 1 final expression} as \begin{multline}\label{E: }
      \lim_{w\to+\infty} \   \limsup\limits_{N\to+\infty}\ \max_{\underset{(b,W)=1}{1\leq b\leq W}} \Bignorm{\frac{1}{N}\sum_{n=1}^{N} {\bf 1}_{s_0\;( Q_0)}(n)\big(\La_{w,b}(n)-1\big)\\ \prod_{j=1}^{\ell}\Big(  \prod_{i\colon (i,j)\in S_1}^{} T_i^{\floor{g_{ij,w,b}(n) +p_{ij,w,b}(n) }+q_{ij,w,b}(n)  } \cdot 
      \prod_{i\colon(i,j)\in S'_2}^{} T_i^{\floor{g_{ij,w,b}(n) +p_{ij,w,b}(n) }+q_{ij,w,b}(n)  }\cdot\\
      \prod_{i\colon(i,j)\in S''_2}^{} T_i^{q_{ij,w,b}(n)  }\cdot 
      \prod_{i\colon (i,j)\not S_1\cup S'_2\cup S''_2}^{} T_i^{\floor{g_{ij,w,b}(n) +p_{ij,w,b}(n) }+q_{ij,w,b}(n)  }     \Big)f_j   }_{L^2(\m)}.
\end{multline}

\subsection*{Step 3: Passing to short intervals }

The functions $g_{ij}(t)+p_{ij}(t)$ with $(i,j)\in S_1$ satisfy the assumptions of Proposition \ref{P: remove error terms for polynomial functions}, while the functions $g_{ij}(t)+p_{ij}(t)$ with $(i,j)\notin S_1\cup S'_2\cup S''_2 $ satisfy the assumptions of Proposition \ref{P: Taylor expansion final} (thus, each one of them satisfies Proposition \ref{P: remove error term for fast functions} for some appropriately chosen values of the integer $k$ in that statement). Lastly, the functions of the set $S_2'$ satisfy the assumptions of Propositions \ref{P: remove error term for slow functions}.
It is straightforward to infer that, in each case, the corresponding property continues to hold when the functions $g_{ij}(t)+p_{ij}(t)$ are replaced by the functions $g_{ij,w,b}(t)+p_{ij,w,b}(t)$. This is a simple consequence of the fact that if $f\in \mathcal{H}$ has polynomial growth, then the functions $f$ and $f_{w,b}$ have the same growth rate.

Let $d_0$ be the maximal degree appearing among the polynomials $p_{ij}(t)$.
Then, we can find a sub-linear function $L(t)$ such that \begin{equation}\label{E: L(t) satisfies the condition of Polynomial proposition}
   t^{ \frac{5}{8}} \lll L(t)\lll t
\end{equation} and, such that there exists positive integers $k_{ij}$ for $(i,j)\notin S_1\cup S'_2\cup S''_2$, for which we have the growth inequalities \begin{equation}\label{E: inequalities that give taylor approximation0}
   \Big| g_{ij}^{(k_{ij})}(t)\Big|^{-\frac{1}{k_{ij}}}\lll L(t) \lll  \Big| g_{ij}^{(k_{ij}+1)}(t)\Big|^{-\frac{1}{k_{ij}+1}}.
\end{equation} Furthermore, we can assume that $k_{ij}$ are very large compared to the maximal degree $d_0$ of the polynomials $p_{ij}(t)$, by taking $L(t)$ to grow sufficiently fast. 
We remark that \eqref{E: inequalities that give taylor approximation0} also implies the inequalities \begin{equation}\label{E: inequalities that give taylor approximation}
      \Big| g_{ij,w,b}^{(k_{ij})}(t)\Big|^{-\frac{1}{k_{ij}}}\lll L(t) \lll  \Big| g_{ij,w,b}^{(k_{ij}+1)}(t)\Big|^{-\frac{1}{k_{ij}+1}}.
\end{equation}for any fixed $w,b$.

For the choice of $L(t)$ that we made above, we apply Lemma \ref{L: long averages to short averages} to infer that it suffices to show that \begin{multline}\label{E: final expression in Step 3}
      \lim_{w\to+\infty} \   \limsup\limits_{R\to+\infty}  \ \max_{\underset{(b,W)=1}{1\leq b\leq W}}\ \E_{1\leq r\leq R} \Bignorm{\E_{r\leq n\leq r+L(r)} {\bf 1}_{s_0\;( Q_0)}(n)\big(\La_{w,b}(n)-1\big)\\
      \prod_{j=1}^{\ell}\Big(  \prod_{i\colon (i,j)\in S_1}^{} T_i^{\floor{g_{ij,w,b}(n) +p_{ij,w,b}(n) }+q_{ij,w,b}(n)  } \cdot 
      \prod_{i\colon(i,j)\in S'_2}^{} T_i^{\floor{g_{ij,w,b}(n) +p_{ij,w,b}(n) }+q_{ij,w,b}(n)  }
      \cdot\\
      \prod_{i\colon(i,j)\in S''_2}^{} T_i^{q_{ij,w,b}(n)  }\cdot
      \prod_{i\colon (i,j)\notin S_1\cup S'_2\cup S''_2}^{} T_i^{\floor{g_{ij,w,b}(n) +p_{ij,w,b}(n) }+q_{ij,w,b}(n)  }     \Big)f_j   }_{L^2(\m)}
      =0.
\end{multline}

\subsection*{Step 4: Reducing to polynomial iterates and using uniformity bounds}

We now fix $w$ (thus $W$) and the integer $b$.
Suppose that $R$ is sufficiently large and consider the expression \begin{multline}\label{E: I could not come up with something good}
   \mathcal{J}_{w,b,s_0}(R):=  \E_{1\leq r\leq R} \Bignorm{\E_{r\leq n\leq r+L(r)} {\bf 1}_{s_0\;( Q_0)}(n)\big(\La_{w,b}(n)-1\big)\\
      \prod_{j=1}^{\ell}\Big(  \prod_{i\colon (i,j)\in S_1}^{} T_i^{\floor{g_{ij,w,b}(n) +p_{ij,w,b}(n) }+q_{ij,w,b}(n)  } \cdot 
      \prod_{i\colon(i,j)\in S'_2}^{} T_i^{\floor{g_{ij,w,b}(n) +p_{ij,w,b}(n) }+q_{ij,w,b}(n)  }\cdot \\
      \prod_{i\colon(i,j)\in S''_2}^{} T_i^{q_{ij,w,b}(n)  }\cdot
      \prod_{i\colon (i,j)\notin S_1\cup S'_2\cup S''_2}^{} T_i^{\floor{g_{ij,w,b}(n) +p_{ij,w,b}(n) }+q_{ij,w,b}(n)  }     \Big)f_j   }_{L^2(\m)}.
\end{multline}
We will apply Propositions \ref{P: remove error term for fast functions}, \ref{P: remove error term for slow functions} and \ref{P: remove error terms for polynomial functions} to replace the iterates with polynomials (with coefficients depending on $r$). Due to the nature of Proposition \ref{P: remove error term for slow functions} (namely, that it excludes a small set of $r\in [1, R]$), we let $\mathcal{E}_{R,w,b}$ denote a subset of $\{1,\dots, R\}$, which will be constructed throughout the proof and will have small size. We remark that the iterates corresponding to $S_2''$ have been dealt with (morally), so we will focus our attention on the other three sets. 

Let $d$ be the maximum number among the degrees among the polynomials $p_{ij},q_{ij}$ and the integers $k_{ij}$.
Let $\e>0$ be a small (but fixed) quantity and we assume that $r$ is large enough in terms of $1/\e$, i.e., larger than some $R_0=R_0(\e)$.
Observe that if $R$ is sufficiently large, then we have $ R_0\leq \e R$.  We include the ``small'' $r$ in the exceptional set $\mathcal{E}_{R,w,b}$, so that $\mathcal{E}_{R,w,b}$ now has at most $\e R$ elements.
 We will need to bound the expression $\mathcal{J}_{w,b,s_0}(R)$ for large $R$ uniformly in $b$.

\smallskip

{\em Throughout the rest of this step, we implicitly assume that all terms of the form $o_r(1)$ or $o_R(1)$ are allowed to depend on the parameters $w$ and $\e$ which will be fixed up until the end of Step 4. One can keep in mind the following hierarchy $\frac{1}{\e}\ll w\ll r$.}

\smallskip

$\underline{\text{Case 1}}:$ We first deal with the functions in $S'_2$. Fix an $(i,j)\in S'_2$ and consider the function $g_{ij,w,b}(n)+p_{ij,w,b}(n)$ appearing in the corresponding iterate. Observe that due to the definition of $S'_2$ in \eqref{E: sets S_2',S_2''}, the polynomial $p_{ij}(t)$ is constant, so that $p_{ij,w,b}(t)$ is also constant. In addition, the function $g_{ij}(t)$ is a sub-fractional function and dominates $\log t$. Therefore, the same is true for the function $g_{ij,w,b}(t)$.

We apply Proposition \ref{P: remove error term for slow functions}: for all except at most $\e R$ values of $r\in [1,R]$, we have that \begin{equation}\label{E: iterates of the set S'_2: completed}
    \floor{g_{ij,w,b}(n)+p_{ij,w,b}(n)}=\floor{g_{ij,w,b}(r)+p_{ij,w,b}(r)} \ \text{for all } n\in [r,r+L(r)].
\end{equation}For each $(i,j)\in S'_2$, we include the ``bad'' values of $r$ to the set $\mathcal{E}_{R,w,b}$, so that the set $\mathcal{E}_{R,w,b}$ now has at most $(k\ell+1)\e R$ elements.

$\underline{\text{Case 2}}:$ Now, we turn our attention to functions on the complement of the set $S_1\cup S'_2\cup S''_2$. The functions $g_{ij}$ satisfy \eqref{E: inequalities that give taylor approximation} and recall that we have chosen $k_{ij}$ to be much larger than the degrees of the $p_{ij}$, so that the derivative of order $k_{ij}$ of our polynomial vanishes. In conclusion, we may conclude that $g_{ij}(t)+p_{ij}(t)$
satisfies the assumptions of Proposition \ref{P: remove error term for fast functions} for the integer $k_{ij}$ (and the sub-linear function $L(t)$ that we have already chosen).

Given $A>0$, we infer that for all but $O_A(L(r)\log^{-A} r)$ values of $n\in [r,r+L(r)]$, we have   \begin{equation}\label{E: iterates of the complement set: completed}
    \floor{g_{ij,w,b}(n)+p_{ij,w,b}(n) }=\floor{\wt{p}_{ij,w,b,r}(n)},
\end{equation}where $\wt{p}_{ij,w,b,r}(n)$ is the polynomial \begin{equation*}
    \sum_{l=0}^{k_{ij}} \frac{(n-r)^{l} g_{ij,w,b}^{(l)}(r) }{l!} +p_{ij,w,b}(n).
\end{equation*}Additionally, the polynomials $\wt{p}_{ij,w,b,r}$ satisfy \begin{equation}\label{E: strong non-concentration of polynomial approximations- strongly non-polynomial case}
     \frac{\bigabs{\{n\in [r,r+L(r)]\colon \{\wt{p}_{ij,w,b,r}(n)\}\in [1-\delta, 1)\}}   }{L(r)}=\delta+O_{A}(\log^{-A} r)
\end{equation}
for any $\delta<1$.
Practically, this last condition signifies that the polynomials $\wt{p}_{ij,w,b,r}$ satisfy the equidistribution condition in Proposition \ref{P: Gowers norm bound on variable polynomials}, which we shall invoke later.

$\underline{\text{Case 3}}:$ Finally, we deal with the case of the set $S_1$. Proposition \ref{P: remove error terms for polynomial functions} suggests that there is a subset $\mathcal{B}_{w,b,r,\e}$ of $[r,r+L(r)]$ of size $O_{k,\ell}(\e L(r))$, such that for every $n\in [r,r+L(r)]\setminus \mathcal{B}_{w,b,r,\e}$, we have 
\begin{equation}\label{E: iterates of the set S_1: completed}
    \floor{p_{ij,w,b}(n)+g_{ij,w,b}(n)}=\floor{p_{ij,w,b}(n)+g_{ij,w,b}(r)}.
\end{equation}
Additionally, the set $\mathcal{B}_{w,b,r,\e}$ satisfies \begin{equation}\label{E: bound on the size of B_r}
     \frac{1}{L(r)} \sum_{r\leq n\leq r+L(r)} \La_{w,b}(n){\bf 1}_{\mathcal{B}_{w,b,r,\e}}(n)\ll_{k,\ell,d} \ \e +o_w(1)\log \frac{1}{\e}+o_r(1).
\end{equation}
We emphasize that the asymptotic constant in \eqref{E: bound on the size of B_r} depends only on $k,l,d$, so that the constant is the same regardless of the choice of the parameters $w,b$.

First of all, we apply 
\eqref{E: iterates of the set S'_2: completed} to simplify the expression for  $\mathcal{J}_{w,b,s_0}(R)$. Namely, for any $r\notin \mathcal{E}_{R,w,b}$, we have that the inner average in the definition of $\mathcal{J}_{w,b,s_0}(R)$ is equal to \begin{multline*}
     \Bignorm{\E_{r\leq n\leq r+L(r)} {\bf 1}_{s_0\;( Q_0)}(n)\big(\La_{w,b}(n)-1\big)
      \prod_{j=1}^{\ell}\Big(  \prod_{i\colon (i,j)\in S_{1}}^{} T_i^{\floor{g_{ij,w,b}(n) +p_{ij,w,b}(n) }+q_{ij,w,b}(n)  }\cdot \\
      \prod_{i\colon(i,j)\in S'_2}^{} T_i^{\floor{g_{ij,w,b}(r) +p_{ij,w,b}(r) }+q_{ij,w,b}(n)  }\cdot \prod_{i\colon(i,j)\in S''_2}^{} T_i^{q_{ij,w,b}(n)  }\cdot\\
      \prod_{i\colon (i,j)\notin S_1\cup S'_2\cup S''_2}^{} T_i^{\floor{g_{ij,w,b}(n) +p_{ij,w,b}(n) }+q_{ij,w,b}(n)  }     \Big)f_j   }_{L^2(\m)}.
\end{multline*}
Thus, we have replaced the iterates of the set $S'_2$ with polynomials in the averaging parameter $n$.

Secondly,
we use \eqref{E: iterates of the complement set: completed} to deduce that for all, except at most $O_A(k\ell L(r)\log^{-A} r)$
values of $n\in [r, r+L(r)]$, the product of transformations appearing in the previous relation can be written as \begin{multline}\label{E: first reduction of the product on the iterates}
      \prod_{j=1}^{\ell}\Big(  \prod_{i\colon (i,j)\in S_{1}}^{} T_i^{\floor{g_{ij,w,b}(n) +p_{ij,w,b}(n) }+q_{ij,w,b}(n)  } 
      \prod_{i\colon(i,j)\in S'_2}^{} T_i^{\floor{g_{ij,w,b}(r) +p_{ij,w,b}(r) }+q_{ij,w,b}(n)  }\cdot\\
      \prod_{i\colon(i,j)\in S''_2}^{} T_i^{q_{ij,w,b}(n)  }\cdot
      \prod_{i\colon (i,j)\notin S_1\cup S'_2\cup S''_2}^{} T_i^{\floor{\wt{p}_{ij,w,b}(n)}+q_{ij,w,b}(n)  }     \Big)f_j.
\end{multline}
The contribution of the exceptional set can be at most \begin{equation*}
   k\ell \log(Wr+WL(r)+b)\cdot O_A(\log^{-A}r),
\end{equation*}since each $\La_{w,b}(n)$ is bounded by $\log(Wn+b)$. Therefore, if we choose $A\geq 2$, this contribution is $o_r(1)$ and we can rewrite the average over the corresponding short interval as \begin{multline}\label{E: equation before lifting two flows first time}
    \Bignorm{\E_{r\leq n\leq r+L(r)} {\bf 1}_{s_0\;( Q_0)}(n)\big(\La_{w,b}(n)-1\big)
     \prod_{j=1}^{\ell}\Big(  \prod_{i\colon (i,j)\in S_{1}}^{} T_i^{\floor{g_{ij,w,b}(n) +p_{ij,w,b}(n) }+q_{ij,w,b}(n)  } \\
      \prod_{i\colon(i,j)\in S'_2}^{} T_i^{\floor{g_{ij,w,b}(r) +p_{ij,w,b}(r) }+q_{ij,w,b}(n)  }\cdot \prod_{i\colon(i,j)\in S''_2}^{} T_i^{q_{ij,w,b}(n)  }\cdot\\
      \prod_{i\colon (i,j)\notin S_1\cup S'_2\cup S''_2}^{} T_i^{\floor{\wt{p}_{ij,w,b,r}(n)}+q_{ij,w,b}(n)  }     \Big)f_j  }_{L^2(\m)}
      +o_r(1).
\end{multline}Thus, we have reduced our iterates to polynomial form in this case as well.

Finally, we follow the same procedure for the set $S_1$. Namely, for all integers $n$ in the interval $[r,r+L(r)]$ such that $n\notin \mathcal{B}_{w,b,r,\e}$, we use \eqref{E: iterates of the set S_1: completed} to rewrite \eqref{E: first reduction of the product on the iterates} as  \begin{multline*}
       \prod_{j=1}^{\ell}\Big(  \prod_{i\colon (i,j)\in S_{1}}^{} T_i^{\floor{g_{ij,w,b}(r) +p_{ij,w,b}(n) }+q_{ij,w,b}(n)  } 
      \prod_{i\colon(i,j)\in S'_2}^{} T_i^{\floor{g_{ij,w,b}(r) +p_{ij,w,b}(r) }+q_{ij,w,b}(n)  }\cdot\\
      \prod_{i\colon(i,j)\in S''_2}^{} T_i^{q_{ij,w,b}(n)  }\cdot
      \prod_{i\colon (i,j)\notin S_1\cup S'_2\cup S''_2}^{} T_i^{\floor{\wt{p}_{ij,w,b}(n)}+q_{ij,w,b}(n)  }     \Big)f_j.
\end{multline*}
The contribution of the set $\mathcal{B}_{w,b,r,\e}$ on the average over the interval $[r,r+L(r)]$ can be estimated using the triangle inequality. More specifically, this contribution is smaller than \begin{equation*}
    \frac{1}{L(r)}\sum_{r\leq n\leq r+L(r)}{\bf 1}_{s_0\;(Q_0)}(n) \Bigabs{\La_{w,b}(n)-1}{\bf 1}_{\mathcal{B}_{w,b,r,\e}}(n).
\end{equation*}We bound the characteristic function ${\bf 1}_{s_0\;(Q_0)} $ trivially by 1, so that the above quantity is smaller than \begin{equation}\label{E: contribution of B_r on the average}
     \frac{1}{L(r)}\sum_{r\leq n\leq r+L(r)}\La_{w,b}(n){\bf 1}_{\mathcal{B}_{w,b,r,\e}}(n)+ \frac{1}{L(r)}\sum_{r\leq n\leq r+L(r)}{\bf 1}_{\mathcal{B}_{w,b,r,\e}}(n).
\end{equation}The second term contributes $O_{k,\ell}(\e)$, since $\mathcal{B}_{w,b,r,\e}$ has at most $O_{k,\ell}(\e L(r))$ elements. On the other hand, we have a bound for the first term already in \eqref{E: bound on the size of B_r}. Thus, the total contribution is $O_{k,\ell,d}(1)$ times the expression \begin{equation*}
    \e+o_w(1)\log \frac{1}{\e}+o_r(1).
\end{equation*}
In view of the above, we deduce that the average in \eqref{E: equation before lifting two flows first time} is bounded by $O_{k,\ell,d}(1)$ times 
\begin{multline}\label{E: final polynomial ergodic average }
     \Bignorm{\E_{r\leq n\leq r+L(r)} {\bf 1}_{s_0( Q_0)}(n)\big(\La_{w,b}(n)-1\big)
     \prod_{j=1}^{\ell}\Big(  \prod_{i\colon (i,j)\in S_{1}}^{} T_i^{\floor{g_{ij,w,b}(r) +p_{ij,w,b}(n) +q_{ij,w,b}(n)}  } \\
      \prod_{i\colon(i,j)\in S'_2}^{} T_i^{\floor{g_{ij,w,b}(r) +p_{ij,w,b}(r) +q_{ij,w,b}(n)}  }\cdot \prod_{i\colon(i,j)\in S''_2}^{} T_i^{\floor{q_{ij,w,b}(n)}  }\cdot\\
      \prod_{i\colon (i,j)\notin S_1\cup S'_2\cup S''_2}^{} T_i^{\floor{\wt{p}_{ij,w,b}(n)+q_{ij,w,b}(n)}  }     \Big)f_j  }_{L^2(\m)}
      + \e+o_w(1)\log \frac{1}{\e}+o_r(1).
      \end{multline}
Here, we moved the polynomials $q_{ij,w,b}$ back inside the integer parts, which we are allowed to do since they have integer coefficients.

 The polynomials in the iterates corresponding to $S_1, S'_2, S''_2$, and the complement of  $S_1\cup S'_2\cup S''_2$ fulfill the hypothesis of Proposition \ref{P: Gowers norm bound on variable polynomials}. To keep the number of parameters lower, we will apply this proposition for $\delta=\e$, where we have assumed that $\e$ is a very small parameter. Accordingly, we assume (as we may) that $w$ and $r$ are much larger than $\frac{1}{\e}$. To see why the hypotheses are satisfied, observe that for the first set, this follows from the fact that $p_{ij,w,b}$ has at least one non-constant irrational coefficient (since $p_{ij}$ is non-constant by the definition of $S_1$). 
 Therefore, the number of integers $n\in [r,r+L(r)]$ for which we have \begin{equation*}
     \{g_{ij,w,b}(r) +p_{ij,w,b}(n) +q_{ij,w,b}(n)\}\in (1-\e,1)
 \end{equation*}is smaller than $2\e L(r)$ for $r$ sufficiently large.
 At the same time, the result is immediate for the second and third sets, since the iterates involve polynomials with integer coefficients (except, possibly, their constant terms). For the final set, this claim follows from 
\eqref{E: strong non-concentration of polynomial approximations- strongly non-polynomial case}.

In view of the prior discussion, we conclude that there exists a positive integer $s$, that depends only on $d,k,\ell$, such that the expression   
in \eqref{E: final polynomial ergodic average } is bounded by \begin{multline}\label{E: final gowers norm bound}
    \e^{-k\ell} \bignorm{{\bf 1}_{s_0\;(Q_0)}\big(\La_{w,b}(n)-1\big)  }_{U^s(r,r+sL(r)]} 
      +\e^{-k\ell}o_w(1) + o_{\e}(1)(1+o_w(1))+\\
      \e+o_w(1)\log \frac{1}{\e}+o_r(1).
\end{multline}
Applying Lemma \ref{L: Gowers uniformity norms evaluated at arithmetic progressions}, we can bound the previous Gowers norm along the residue class $s_0\; (Q_0)$ as follows: \begin{equation}\label{E: gowers norm bound for Lambda in progressions}
  \bignorm{{\bf 1}_{s_0\;(Q_0)}\big(\La_{w,b}(n)-1\big)  }_{U^s(r,r+sL(r)]}\leq   \bignorm{ \La_{w,b}(n)-1  }_{U^s(r,r+sL(r)]}.
\end{equation}

In view of the arguments above, we conclude that, for every $r\notin \mathcal{E}_{R,w,b}$, the following inequality holds \begin{multline*}
    \Bignorm{\E_{r\leq n\leq r+L(r)} {\bf 1}_{s_0( Q_0)}(n)\big(\La_{w,b}(n)-1\big)\\
      \prod_{j=1}^{\ell}\Big(  \prod_{i\colon (i,j)\in S_1}^{} T_i^{\floor{g_{ij,w,b}(n) +p_{ij,w,b}(n) }+q_{ij,w,b}(n)  } \cdot 
      \prod_{i\colon(i,j)\in S'_2}^{} T_i^{\floor{g_{ij,w,b}(n) +p_{ij,w,b}(n) }+q_{ij,w,b}(n)  }\cdot \\
      \prod_{i\colon(i,j)\in S''_2}^{} T_i^{q_{ij,w,b}(n)  }\cdot
      \prod_{i\colon (i,j)\notin S_1\cup S'_2\cup S''_2}^{} T_i^{\floor{g_{ij,w,b}(n) +p_{ij,w,b}(n) }+q_{ij,w,b}(n)  }     \Big)f_j   }_{L^2(\m)}\ll_{k,\ell,d}\\\
    \e^{-k\ell}\bignorm{ \big(\La_{w,b}(n)-1\big)  }_{U^s(r,r+sL(r)]}+\e+\big(\e^{-k\ell}+\log \frac{1}{\e}+o_{\e}(1)\big)o_w(1)+o_{\e}(1)+o_r(1).
\end{multline*}We apply this estimate to the double average defining $\mathcal{J}_{w,b,s_0}(R)$ in \eqref{E: I could not come up with something good}. This estimate holds for every $r\notin \mathcal{E}_{R,w,b}$ and, thus, we need an estimate for the values of $r$ in this exceptional set. In order to achieve this, we recall that the set $\mathcal{E}_{R,w,b}$ has at most $(2k\ell+1)\e R$ elements. For each $r\in \mathcal{E}_{R,w,b}$, we use the triangle inequality to bound the average over the corresponding short interval by \begin{equation*}
    \frac{1}{L(r)} \sum_{\underset{n\equiv s_0\; ( \, Q_0)}{r\leq n\leq r+L(r)}} (\La(Wn+b) +1).
\end{equation*}
We bound the characteristic function of the residue class $n\equiv s_0\;(Q_0)$ trivially by 1 and
apply Corollary \ref{C: Brun-Titchmarsh inequality for von Mangoldt sums} to conclude that this expression is $O(1)+o_r(1)$, using similar estimates as the ones used in the proof of Proposition \ref{P: remove error terms for polynomial functions} (see \eqref{E: sieve upper bound on the S_r}). Therefore, the contribution of the set $\mathcal{E}_{R,w,b}$ is at most $O_{k,\ell}(\e)+o_r(1)$.
Combining all of the above, we arrive at the estimate\begin{multline}\label{E: penultimate bound}
\mathcal{J}_{w,b,s_0}(R)\ll_{d,k,\ell}  \  \e^{-k\ell} \Big( \E_{1\leq r\leq R}\bignorm{ \big(\La_{w,b}(n)-1\big)  }_{U^s(r,r+sL(r)]}\Big) +\e^{-k\ell}o_w(1)+\\
o_{\e}(1)(1+o_w(1))+o_R(1).
\end{multline}

We restate \eqref{E: final expression in Step 3}  here. Namely, we want to show that \begin{equation*}
       \limsup\limits_{R\to+\infty}\ \max_{\underset{(b,W)=1}{1\leq b\leq W}}\ \mathcal{J}_{w,b,s_0}(R)=o_w(1).
\end{equation*}
Applying \eqref{E: penultimate bound}, we conclude that for a fixed $w$, we have\begin{multline*}
    \limsup\limits_{R\to+\infty} \max_{\underset{(b,W)=1}{1\leq b\leq W}}\ \mathcal{J}_{w,b,s_0}(R)\ll_{d,k,\ell} \e^{-k\ell} \Big( \lim_{R\to+\infty}\E_{1\leq r\leq R}\ \max_{\underset{(b,W)=1}{1\leq b\leq W}}\ \bignorm{ \big(\La_{w,b}(n)-1\big)  }_{U^s(r,r+L(r)]}\Big)+\\
     \e^{-k\ell}o_w(1)+o_{\e}(1)(1+o_w(1)).
\end{multline*} 
Due to Theorem \ref{T: Gowers uniformity in short intervals}, we have that \begin{equation*}
 \max_{\underset{(b,W)=1}{1\leq b \leq W}}    \bignorm{ \big(\La_{w,b}(n)-1\big)  }_{U^s(r,r+L(r)]}=o_w(1)
\end{equation*}for every sufficiently large $r$.
Thus, we conclude that \begin{equation*}
     \limsup\limits_{R\to+\infty}\ \max_{\underset{(b,W)=1}{1\leq b\leq W}}\ \mathcal{J}_{w,b,s_0}(R)\ll_{d,k,\ell}   \e^{-k\ell}o_w(1)+ o_{\e}(1)(1+o_w(1)).
\end{equation*}

\subsection*{Step 5: Putting all the bounds together}

We restate here our conclusion. We have shown that for all fixed integers $w$ and real number $0<\e<1$, we have \begin{multline}\label{E: Final Estimate}
    \limsup\limits_{R\to+\infty} \lim\limits_{N\to+\infty}\ \max_{\underset{(b,W)=1}{1\leq b\leq W}} \Bignorm{\frac{1}{N}\sum_{n=1}^{N} {\bf 1}_{s_0( Q_0)}(n)\big(\La_{w,b}(n)-1\big)
     \prod_{j=1}^{\ell}\big(  \prod_{i=1}^{k} T_i^{\floor{a_{ij,w,b}(n) }  }       \big)f_j   }_{L^2(\m)}\\
     \ll_{d,k,\ell} \e^{-k\ell}o_w(1)+o_{\e}(1)(1+o_w(1)),
\end{multline}where we recall that $d$ was the maximum among the integers $k_{ij}$ and the degrees of the polynomials $p_{ij},q_{ij}$ (all of these depend only on the initial functions $a_{ij}$). Sending $w\to+\infty$, we deduce that the limit in \eqref{E: Step 1 final expression} (in view of \eqref{E: Final Estimate}) is smaller than a constant (depending on $k,\ell,d$) multiple of $o_{\e}(1)$. 
Sending $\e\to 0$, we conclude that the original limit is $0$, which is the desired result.
\end{proof}

\section{Proofs of the remaining theorems}\label{Section-Proofs of remaining theorems}

We finish the proofs of our theorems in this section.

\subsection{Proof of the convergence results}
\begin{proof}[Proof of Theorem \ref{T: criterion for convergence along primes}]
   Let $(X,\X,\m,T_1,\dots,T_k)$ be the system and 
   $a_{ij}\in \mathcal{H}$ the functions in the statement.
    In view of Lemma \ref{L: indicator of primes to von-Mangoldt}, it suffices to show that
    the averages \begin{equation*}
        A(N):= \frac{1}{N}\sum_{n=1}^{N} \La(n)\big(  \prod_{i=1}^{k} T_i^{\floor{a_{i1}(n)}}       \big)f_1\cdot\ldots \cdot  \big(  \prod_{i=1}^{k} T_i^{\floor{a_{i\ell}(n)}}       \big)f_{\ell}
    \end{equation*}converge in $L^2(\m)$. For a fixed $w \in \N$, we define $W=\prod_{p\leq w, p\in\P}p$ as usual and let $b\in \N$. We define 
    \begin{equation*}
        B_{w,b}(N):= \frac{1}{N}\sum_{n=1}^{N} \La(n)\big(  \prod_{i=1}^{k} T_i^{\floor{a_{i1}(Wn+b)}}       \big)f_1\cdot\ldots \cdot  \big(  \prod_{i=1}^{k} T_i^{\floor{a_{i\ell}(Wn+b)}}       \big)f_{\ell}.
    \end{equation*}
    
    Let $\e>0$.
    Using Theorem \ref{T: the main comparison}, we can find $w_0\in \N$ (which yields a corresponding $W_0$) such that \begin{equation}\label{E: comparison between the A, B}
\Bignorm{A(W_0N)-\frac{1}{\phi(W_0)}\sum_{\underset{(b,W_0)=1}{1\leq b\leq W_0} }B_{w_0,b}(N)}_{L^2(\mu)}=O(\varepsilon)
\end{equation}for all $N$ sufficiently large. Our hypothesis implies that the sequence of bounded functions $B_{w_0,b}(N)$ is a Cauchy sequence in $L^2(\m)$, which, in conjunction with \eqref{E: comparison between the A, B}, implies that the sequence $A(W_0N)$ is a Cauchy sequence. In particular, we have \begin{equation*}
\norm{A(W_0M)-A(W_0N)}_{L^2(\mu)}=O(\varepsilon),
\end{equation*}for all $N,M$ sufficiently large.
Finally, since
\begin{equation*}
\norm{A(W_0N+b)-A(W_0N)}_{L^2(\mu)}=o_N(1),
\end{equation*} for all $1\leq b\leq W_0$, we conclude that $A(N)$ is a Cauchy sequence, which implies the required convergence.

Furthermore, if the sequence $B_{w,b}(N)$ converges to the function $F$ in $L^2(\m)$ for all $w,r\in \N$, then \eqref{E: comparison between the A, B} implies that $\norm{A(W_0N)-F}_{L^2(\m)}=O(\e)$, for all large enough $N$. Repeating the same argument as above, we infer that $A(N)$ converges to the function $F$ in norm, as we desired.
\end{proof}

\begin{proof}[Proof of Theorem~\ref{T: convergence of Furstenberg averages}]
Let $a\in \mathcal{H}$ satisfy either \eqref{E: far away from real multiples of integer polynomials} or \eqref{E: equal to a real multiple of integer polynomial}, $k\in\N,$ $(X,\X,\m,T)$ be any measure-preserving system, and functions $f_1,\dots,f_k\in L^{\infty}(\m).$ 
Observe that in either case, the function $a$ satisfies \eqref{E: far away from rational polynomials} or \eqref{E: essentially equal to a polynomial}. In addition, when $a(t)$ satisfies
either of the two latter conditions, then the function $a(Wt+b)$ satisfies the same condition, for all $W,b\in \N$.

Using \cite[Theorem 2.1]{Fra-Hardy-singlecase},\footnote{There is a slight issue here, in that we would need the assumption that the function $a(Wn+b)$ belongs to $\mathcal{H}$ in order to apply Theorem 2.2 from \cite{Fra-Hardy-singlecase}, However, the proof in \cite{Fra-Hardy-singlecase} only requires some specific growth conditions on the derivatives of the function $a(Wn+b)$ (specifically those outlined in equation 26 of that paper), which follow naturally from the assumption that $a\in \mathcal{H}$.  } we have that, for all $W,b \in \N,$, the averages \begin{equation*}
    \frac{1}{N}\sum_{n=1}^{N} T^{\floor{a(Wn+b)}}f_1\cdot\ldots\cdot T^{k\floor{a(Wn+b)}}f_k
\end{equation*}converge in $L^2(\m)$.
aWe conclude that the two conditions of Theorem \ref{T: criterion for convergence along primes} are satisfied, which shows that the desired averages converge.

In particular, if $a$ satisfies condition \eqref{E: far away from real multiples of integer polynomials}, we can invoke \cite[Theorem~2.2]{Fra-Hardy-singlecase} to conclude that the limit of the averages \begin{equation*}
      \frac{1}{N}\sum_{n=1}^{N} T^{\floor{a(Wn+b)}}f_1\cdot\ldots\cdot T^{k\floor{a(Wn+b)}}f_k
\end{equation*}
is equal to the limit (in $L^2(\mu)$) of the averages
\begin{equation*}
        \frac{1}{N}\sum_{n=1}^{N} T^nf_1\cdot\ldots\cdot T^{kn} f_k.
    \end{equation*}
   Again, Theorem \ref{T: criterion for convergence along primes} yields the desired conclusion.
\end{proof}

\begin{proof}[Proof of Theorem \ref{T: jointly ergodic case}]
    We work analogously as in the proof of Theorem \ref{T: convergence of Furstenberg averages}. The only difference is that in this case, we use \cite[Theorem 1.2]{Tsinas} to deduce that, for all  $W\in \N,$ $b\in\N$ positive integers $W$ and $b$, the averages \begin{equation*}
        \frac{1}{N}\sum_{n=1}^{N} T^{\floor{a_1(Wn+b)}}f_1\cdot\ldots\cdot T^{\floor{a_k(Wn+b)}}f_k
    \end{equation*}
     converge in $L^2(\m)$ to the product $\wt{f}_1\cdot \ldots \cdot \wt{f}_{k}$. The result follows from Theorem \ref{T: criterion for convergence along primes}.
\end{proof}

\begin{proof}[Proof of Theorem~\ref{T: nikos result to primes}] 
  The proof follows identically as the one of Theorem~\ref{T: jointly ergodic case}
by using \cite[Theorem 2.3]{Fra-Hardy-multidimensional}
 instead of \cite[Theorem 1.2]{Tsinas}.\end{proof}

\subsection{Proof of the recurrence results}

We recall Furstenberg's Correspondence Principle for $\Z^d$-actions \cite{Furstenberg-book}, for the reader's convenience. 

\begin{customthm}{F}[Furstenberg's Correspondence Principle] Let $d\in \N$ and $E\subseteq \Z^d.$ There exists a system $(X,\X,\mu,T_1,\ldots,T_d)$ and a set $A\in \X$ with $\bar{d}(E)=\mu(A),$ such that 
\begin{equation*}
\bar{d}\big(E\cap(E+{\bf{n}}_1)\cap\dots\cap(E-{\bf{n}}_k)\big)\geq \mu\left(A\cap\prod_{i=1}^d T_i^{-n_{i,1}}A\cap\dots\cap \prod_{i=1}^d T_i^{-n_{i,k}}A\right),
\end{equation*}
for all $k\in \N$ and ${\bf{n}}_j=(n_{1,j},\dots,n_{d,j})\in\Z^d,$ $1\leq j\leq k.$  
\end{customthm}

In view of the correspondence principle, the corollaries in Section \ref{Section-Introduction} follow easily.

\begin{proof}[Proof of Theorem \ref{T: multiple recurrence for szemeredi type patterns}]
(a) We apply Theorem \ref{T: convergence of Furstenberg averages} for the functions $f_1=\dots=f_k={\bf 1}_{A}$. Since convergence in $L^2(\m)$ implies weak convergence, integrating along $A$ the relation \begin{equation*}
      \lim\limits_{N\to+\infty} \frac{1}{\pi(N)} \sum_{p\in \P\colon p\leq N} T^{\floor{a(p)}}{\bf 1}_A\cdot \ldots\cdot T^{k\floor{a(p)}} {\bf 1}_A = \\
      \lim\limits_{N\to+\infty} \frac{1}{N}\sum_{n=1}^{N} T^n{\bf 1}_A\cdot\ldots\cdot T^{kn}{\bf 1}_{A},
\end{equation*}and applying Furstenberg's multiple recurrence theorem we infer that \begin{equation*}
       \lim\limits_{N\to+\infty} \frac{1}{\pi(N)} \sum_{p\in \P\colon p\leq N} \m\big(A\cap T^{-\floor{a(p)}}A\cap \dots\cap T^{-{k\floor{a(p)}}}A \big)> 0,
\end{equation*}which is the desired result.

(b) We write $a(t)=cq(t)+\e(t)$, where $q(t)\in \Z[t],\ q(0)=0,\ c\in \R$ and $\e(t)$ is a function that converges to $0$, as $t\to+\infty$. Using \cite[Proposition 3.8]{koutsogiannis-closest}, we have that there exists $c_0$ depending only on $\m(A)$, the degree of $q$ and $k$, such that \begin{equation*}
    \liminf\limits_{N\to+\infty} \frac{1}{N} \sum_{n=1}^{N} \m(A\cap T^{-[[cq(n)]]}A\cap \dots\cap T^{-k[[cq(n)]]}A)\geq c_0.
\end{equation*}

Now, we consider two separate cases. If $c$  is rational with denominator $Q$ in lowest terms, then for $t$ sufficiently large, we have $|\e(t)|\leq (2Q)^{-1}$. Therefore, we immediately deduce that \begin{equation*}
    [[cq(t)+\e(t)]]=[[cq(t)]].
\end{equation*} Thus, we conclude that \begin{equation}\label{E: recurrence lower bound for polynomials}
    \liminf\limits_{N\to+\infty} \frac{1}{N} \sum_{n=1}^{N} \m(A\cap T^{-[[cq(n)+\e(n)]]}A\cap \dots\cap T^{-k[[cq(n)+\e(n)]]}A)\geq c_0.
\end{equation}

If $c$ is irrational, then the polynomial $cq(t)$ is uniformly distributed mod 1. Given $\delta>0$, we consider the set $S:=\{n\in \N \colon \{cq(n)\}\in [\delta,1-\delta] \}$, which has density $1-2\delta$. Therefore, we have \begin{multline*}
 \Bigabs{   \frac{1}{N} \sum_{n=1}^{N} \m(A\cap T^{-[[cq(n)+\e(n)]]}A\cap \dots\cap T^{-k[[cq(n)+\e(n)]]}A)-\\
    \frac{1}{N}\sum_{n=1}^{N}  \m(A\cap T^{-[[cq(n)]]}A\cap \dots\cap T^{-k[[cq(n)]]}A) }\leq 2\delta+o_N(1).
\end{multline*}Sending $\delta\to 0^{+}$, we derive \eqref{E: recurrence lower bound for polynomials} in this case as well.

Notice that since $c_0$ depends only on the degree of $q$, we have that \begin{equation*}
     \liminf\limits_{N\to+\infty} \frac{1}{N} \sum_{n=1}^{N} \m(A\cap T^{-[[cq(Rn)+\e(Rn)]]}A\cap \dots\cap T^{-k[[cq(Rn)+\e(Rn)]]}A)\geq c_0,
\end{equation*}for all positive integers $R$.
Now, we apply Theorem \ref{T: the main comparison} with $b=1$ and the functions $a(\cdot -1)$, where we recall that $a(t)=cq(t)+\e(t)$ to obtain that for some sufficiently large $w$, we have \begin{equation*}
        \liminf\limits_{N\to+\infty} \frac{1}{N}\sum_{n=1}^{N} \La_{w,1}(n)  \m\big(A\cap T^{-\floor{a(Wn)}}A\cap \dots\cap T^{-k\floor{a(Wn)}}A\big) \geq c_0/2,
    \end{equation*}where $W$ is defined as usual in terms of $w$.
Finally, we observe that we can replace the function $\La(n)$ in the previous relation with the function $\La(n){\bf 1}_{\P}(n)$ since the contribution of the prime powers (i.e. with exponent $\geq 2$) is negligible on the average. Therefore, we conclude that 
\begin{equation*}
        \liminf\limits_{N\to+\infty} \frac{1}{N}\sum_{n=1}^{N} \La_{w,1}(n){\bf 1}_{\P}(Wn+1)  \m\big(A\cap T^{-\floor{a(Wn)}}A\cap \dots\cap T^{-k\floor{a(Wn)}}A\big) \geq c_0/2,
    \end{equation*}which implies the desired result.
Analogously, we reach the expected conclusion for the set $\P+1$ instead of $\P-1$.
\end{proof}

\begin{proof}[Proof of Theorem \ref{T: multiple recurrence in the jointly ergodic case}]
   Similarly to the proof of Theorem \ref{T: multiple recurrence for szemeredi type patterns}, we apply Theorem \ref{T: jointly ergodic case} for the functions $f_1=\cdots=f_k={\bf1}_{A}$. 
  We deduce that \begin{equation}\label{E: vale oti thes edw pera}
        \lim\limits_{N\to+\infty} \frac{1}{\pi(N)} \sum_{p\in \P\colon p\leq N} \m\big(A\cap T^{-\floor{a_1(p)}}A \cap \dots \cap T^{-\floor{a_k(p)}}A \big)=\int {\bf 1}_A\cdot \big( \E({\bf1}_A| \mathcal{I}(T))\big)^{k}\, d\m.
    \end{equation}However, using that the function ${\bf 1}_A$ is non-negative and H\"{o}lder's inequality, we get
$$\int {\bf 1}_A\cdot \big( \E({\bf1}_A| \mathcal{I}(T))\big)^{k}\, d\m\geq \Big(\int \E({\bf 1}_A|\mathcal{I}(T))\, d\m \Big )^{k+1}=\big(\m(A)  \big)^{k+1},$$
and the conclusion follows.
\end{proof}

\begin{proof}[Proof of Theorem \ref{T: multidimensional recurrence for primes}]
    The proof is similar to the proof of Theorem \ref{T: multiple recurrence in the jointly ergodic case}.
    The only distinction is made in \eqref{E: vale oti thes edw pera}, namely we have \begin{multline*}
          \lim\limits_{N\to+\infty} \frac{1}{\pi(N)} \sum_{p\in \P\colon p\leq N} \m\big(A_0\cap T_1^{-\floor{a_1(p)}}A_1 \cap \dots \cap T_k^{-\floor{a_k(p)}}A_k \big)=\\ \int {\bf 1}_{A_0}\cdot \E({\bf1}_{A_1}| \mathcal{I}(T_1))\cdot \ldots\cdot \E({\bf1}_{A_k}| \mathcal{I}(T_k))\, d\m,
    \end{multline*}
       where the sets $A_0, A_1, \dots, A_k$ satisfy the hypothesis.
Since each function $\E({\bf1}_{A_i}| \mathcal{I}(T_i))$ is $T_i$-invariant, we deduce that the integral on the right-hand side is larger than \begin{equation*}
    \int f\cdot  \E(f|\mathcal{I}(T_1))\cdot\ldots\cdot \E(f|\mathcal{I}(T_k))\, d\m,
\end{equation*}where $f={\bf 1}_{A_0\cap T^{\ell_1} A_1\cap \dots \cap T^{\ell_k}A_k}$.
However, since the function $f$ is non-negative, \cite[Lemma~1.6]{Chu-2-commuting}
    implies that \begin{equation*}
     \int f\cdot  \E(f|\mathcal{I}(T_1))\cdot\ldots\cdot \E(f|\mathcal{I}(T_k))\, d\m\geq  \left(\int f\;d\mu\right)^{k+1}= \m(A)^{k+1},
    \end{equation*}and the conclusion follows.
\end{proof}

\subsection{Proof of the equidistribution results in nilmanifolds}

In this final part of this section, we offer a proof for Theorem \ref{T: criterion for pointwise convergence along primes-nil version}. The main tool is the approximation of Lemma \ref{L: approximation by nilsequences}.

\begin{proof}[Proof of Theorem \ref{T: criterion for pointwise convergence along primes-nil version}]
    Let $X$ and $g_1,\dots, g_k,x_1,\dots, x_k$ be as in the statementthe section?, we offer a proof for Theorem 1.12. The main tool is the approximation of and let $s$ denote the nilpotency degree of $X$. It suffices to show that, for any continuous functions $f_1,\dots, f_s$ on $X$, we have the following:\begin{equation*}
        \lim\limits_{N\to+\infty} \frac{1}{\pi(N)} \sum_{p\in\P \colon p\leq N}f_1(g_1^{\floor{a_1(p)}}x_1)\cdot \ldots\cdot f_k(g_k^{\floor{a_k(p)}}x_k) =\int_{Y_1}f_1 \, dm_{Y_1}\cdot\ldots\cdot \int_{Y_k} f_k\, d m_{Y_k},
    \end{equation*}where $Y_i=\overline{(g_i^{\Z}x_i )}$ for all admissible values of $i$.
We rewrite this in terms of the von Mangoldt function as \begin{equation}\label{E: what we want to show}
      \lim\limits_{N\to+\infty} \frac{1}{N} \sum_{n=1}^{N}\La(n) f_1(g_1^{\floor{a_1(n)}}x_1)\cdot \ldots\cdot f_k(g_k^{\floor{a_k(n)}}x_k) =\int_{Y_1}f_1 \, dm_{Y_1}\cdot\ldots\cdot \int_{Y_k} f_k\, d m_{Y_k},
\end{equation}where the equivalence of the last two relations is a consequence of Lemma \ref{L: indicator of primes to von-Mangoldt}.

Our equidistribution assumption implies that for all $W,b\in \N$, we have \begin{equation}\label{E: what the shit assumption gives}
     \lim\limits_{N\to+\infty} \frac{1}{N} \sum_{n=1}^{N} f_1(g_1^{\floor{a_1(Wn+b)}}x_1)\cdot \ldots\cdot f_k(g_k^{\floor{a_k(Wn+b)}}x_k) =\int_{Y_1}f_1 \, dm_{Y_1}\cdot\ldots\cdot \int_{Y_k} f_k\, d m_{Y_k}.
\end{equation}

 We write $Y_i=G_i/\G_i$ for some nilpotent Lie groups $G_i$ with discrete and co-compact subgroups $\G_i$ and denote $Y=Y_1\times \dots\times Y_k$.
    Define the function $F:Y\to \C$ by $F(y_1,\dots, y_k)=f_1(y_1)\cdot\ldots\cdot f_k(y_k)$ and rewrite \eqref{E: what we want to show} as \begin{equation}\label{E: modified what we want to show}
     \lim\limits_{N\to+\infty} \frac{1}{N} \sum_{n=1}^{N} \La(n) F(\wt{g}_1^{\floor{a_1(n)}}\cdot\ldots\cdot \wt{g}_k^{\floor{a_k(n)}} \wt{x})
     =\int_{Y}F \, dm_{Y},
\end{equation}where $\wt{g_i}$ is the element on the nilpotent Lie group $G_1\times\dots\times  G_k$ whose $i$-th coordinate is equal to $g_i$ and the rest of its entries are the corresponding identity elements. Lastly, $\wt{x}$ is the point $(x_1,\dots, x_k)\in Y$.
Similarly, we rewrite \eqref{E: what the shit assumption gives} as  \begin{equation}\label{E: modified what the shit assumption gives}
       \lim\limits_{N\to+\infty} \frac{1}{N} \sum_{n=1}^{N}  F(\wt{g}_1^{\floor{a_1(Wn+b)}}\cdot\ldots\cdot \wt{g}_k^{\floor{a_k(Wn+b}} \wt{x})
     =\int_{Y}F \, dm_{Y}.
\end{equation}Therefore, we want to prove \eqref{E: modified what we want to show} under the assumption that \eqref{E: modified what the shit assumption gives} holds for all $W,r\in \N$.

We use the notation  \begin{equation*}
    A(N):=\frac{1}{N} \sum_{n=1}^{N} \La(n) F(\wt{g}_1^{\floor{a_1(n)}}\cdot\ldots\cdot \wt{g}_k^{\floor{a_k(n)}} \wt{x}),
\end{equation*}and \begin{equation*}
    B_{W,b}(N):=\frac{1}{N} \sum_{n=1}^{N}  F(\wt{g}_1^{\floor{a_1(Wn+b)}}\cdot\ldots\cdot \wt{g}_k^{\floor{a_k(Wn+b)}} \wt{x})
\end{equation*}for convenience.

Let $\e>0$.
Observe that the sequence $\psi({\bf n})=F(\wt{g}_1^{n_1}\cdot\ldots\cdot \wt{g}_k^{n_k} \wt{x})$
is an $s$-step nilsequence in $k$-variables.
We apply  Lemma \ref{L: approximation by nilsequences} to deduce that there exists a system $(X',\X',\m,S_1,\dots, S_k)$ and functions $G_1,\dots, G_s\in L^{\infty}(\m)$ such that \begin{equation*}
    \Bigabs{F(\wt{g}_1^{n_1}\cdot\ldots\cdot \wt{g}_k^{n_k} \wt{x})-\int \prod_{j=1}^{s} \Big(\prod_{i=1}^{k} S_i^{\ell_jn_i}\Big)G_j \,d\m }\leq \e
\end{equation*}for all $n_1,\dots, n_k \in \Z$, where $\ell_j=(s+1)!/j$.

Thus, if we define \begin{equation*}
    A'(N):=\frac{1}{N} \sum_{n=1}^{N} \La(n) \int \prod_{j=1}^{s+1} \Big(\prod_{i=1}^{k} S_i^{\ell_j\floor{a_i(n)}}\Big)G_j \,d\m,
\end{equation*}and \begin{equation*}
    B'_{W,b}(N)=\frac{1}{N} \sum_{n=1}^{N} \int \prod_{j=1}^{s+1} \Big(\prod_{i=1}^{k} S_i^{\ell_j\floor{a_i(Wn+b)}}\Big)G_j \,d\m,
\end{equation*}we deduce that $|B_{W,b}(N)-B'_{W,b}(N)|\leq \e$, for all $N\in \N$, whereas $|A(N)-A'(N)|\leq \e (1+o_N(1))$, by the prime number theorem.

The functions $a_1,\dots, a_k $ satisfy the assumptions of Theorem \ref{T: the main comparison}. Thus, we deduce that if we pick $w_0$ (which provides a corresponding $W_0$) sufficiently large and apply the Cauchy-Schwarz inequality, we will get \begin{equation}\label{E: application of Theorem 1.1 in nilproof}
   \max_{\underset{(b,W_0)=1}{1\leq b\leq W}} \Bigabs{\frac{1}{N}\sum_{n=1}^{N}   \big(\La_{w_0,b}(n)-1\big) \int \prod_{j=1}^{s+1} \Big(\prod_{i=1}^{k} S_i^{\ell_j\floor{a_i(W_0n+b)}}\Big)G_j \,d\m   }\leq \e
\end{equation}for every sufficiently large $N\in \N$. In addition, we use \eqref{E: modified what the shit assumption gives}, the inequality $|B_{W_0,b}(N)-B'_{W_0,b}(N)|\leq \e$ and the triangle inequality to infer that for $N$ large enough, we have \begin{equation}\label{E: approximation of B(N) by the integral}
    \Bigabs{B'_{W_0,b}(N) -\int_{Y}F \, dm_{Y}}\leq 2\e, 
\end{equation}for all $1\leq b\leq W_0$ coprime to $W_0$.

Observe that \eqref{E: application of Theorem 1.1 in nilproof} implies that for all $N$ sufficiently large, we have \begin{equation*}
    \Bigabs{A'(W_0N)-\frac{1   }{\phi(W_0)} \sum_{\underset{(b,W_0)=1}{1\leq b\leq W_0}}^{}B'_{W_0,b}(N)    }\leq 2\e, 
\end{equation*}and we can combine this with \eqref{E: approximation of B(N) by the integral}  to conclude that \begin{equation*}
     \Bigabs{A'(W_0N)- \int_{Y}F \, dm_{Y} }\leq 4\e
\end{equation*}for all $N$ sufficiently large. 
Since $|A'(N)-A(N)|\leq \e (1+o_N(1))$, we finally arrive at the inequality \begin{equation*}
    \Bigabs{A(W_0N) -\int_{Y}F \, dm_{Y}}\leq 6\e,
\end{equation*}for all large enough $N\in \N$. Since $|A(W_0N)-A(W_0N+b)|=o_N(1)$ for all $1\leq b\leq W$, we conclude that \begin{equation*}
    \Bigabs{A(N) -\int_{Y}F \, dm_{Y}}\leq 7\e,
\end{equation*}for all sufficiently large $N\in \N$. Sending $\e\to 0$, we deduce \eqref{E: modified what we want to show}, which is what we wanted to show.
\end{proof}

\begin{proof}[Proof of Proposition Corollary \ref{C: equidistribution nilmanifolds}]
    The result follows readily from Theorem \ref{T: criterion for pointwise convergence along primes-nil version}. The first hypothesis of the criterion is satisfied, since each of the functions $a_i(t)$ satisfies \eqref{E: t^e away from polynomials}, while condition (b) follows from \cite[Theorem 1.1]{tsinas-pointwise} and our assumption that $a_i(Wt+b)$ belongs to $\mathcal{H}$ .
\end{proof}

\section{More general iterates}\label{Section: more general iterates}

In this last section of the article, we discuss how the hypotheses that the functions $a_i(t)$ in the iterates belong to a Hardy field $\mathcal{H}$ can be weakened. The starting point is Proposition \ref{P: remove error term for fast functions}, which was established for general smooth functions, subject to some growth inequalities on the derivative of some particular order (the integer $k$ in the statement). Unfortunately, one cannot generalize theorems such as Theorem \ref{T: jointly ergodic case}, which involve several functions to a more general class. The main obstruction is that in order to obtain the simultaneous Taylor 
expansions, one needs to find a function $L(t)$ (the length of the short interval) that satisfies a growth relation for all functions at the same time, which is non-trivial to perform, because we do not know how the derivatives of one function might grow relative to the derivatives of another function. Nonetheless, this is not feasible in the case of one function, such as Theorem~\ref{T: convergence of Furstenberg averages}, which leads to Szemer\'{e}di-type results.

We have the following proposition.

\begin{proposition}\label{P: comparison for more general iterates}
       Let $a(t)$ be a function, defined for all sufficiently large $t$ and satisfying $|a(t)|\to+\infty$, as $t\to+\infty$.
       Suppose there exists a positive integer $k$ for which $a$ is $C^{k+1},$ $a^{(k+1)}(t)$ converges to 0 monotonically, and such that\footnote{See the subsection with the notational conventions in Section \ref{Section-Introduction} for the notation $\lll$.}\begin{equation*}
           t^{5/8}\ll \bigabs{a^{(k)}(t)}^{-\frac{1}{k}}\lll\bigabs{a^{(k+1)}(t)}^{-\frac{1}{k+1}}\ll t.
       \end{equation*}
       
     Then, for any $\ell\in\N,$ measure-preserving system $(X,\X, \m, T_1,\dots, T_{\ell}),$ and functions $f_1,\dots,f_{\ell}\in L^{\infty}(\m)$, we have  \begin{equation*}\label{E: main average comparison tempered}
         \lim_{w\to+\infty} \  \limsup\limits_{N\to+\infty}\  \max_{\underset{(b,W)=1}{1\leq b\leq W}}\ \Bignorm{\frac{1}{N}\sum_{n=1}^{N} \big(\La_{w,b}(n) -1\big) T_1^{\floor{a(Wn+b)}}f_1\cdot\ldots\cdot  T_{\ell}^{\floor{a(Wn+b)}}       f_{\ell}  }_{L^2(\m)}=0.
        \end{equation*}
\end{proposition}
 We remark that any improvement in the parameter $5/8$ 
in Theorem \ref{T: Gowers uniformity in short intervals} will also lower the term $t^{5/8}$ on the leftmost part of the growth inequalities accordingly.
\begin{proof}[Sketch of the proof of Proposition \ref{P: comparison for more general iterates}]
We define $L(t)$ to be the geometric mean of the functions $\bigabs{a^{(k)}(t)}^{-\frac{1}{k}}$ and $\bigabs{a^{(k+1)}(t)}^{-\frac{1}{k+1}}$, which is well-defined for all $t$ sufficiently large. A standard computation implies the relation \begin{equation*}
         t^{5/8}\ll \bigabs{a^{(k)}(t)}^{-\frac{1}{k}}\lll L(t) \lll\bigabs{a^{(k+1)}(t)}^{-\frac{1}{k+1}}\ll t.
    \end{equation*}Regarding the parameter $w$ as fixed, it suffices to show that \begin{equation*}
        \limsup\limits_{N\to+\infty}\max_{\underset{(b,W)=1}{1\leq b\leq W}}\ \Bignorm{\frac{1}{N}\sum_{n=1}^{N} \big(\La_{w,b}(n) -1\big) T_1^{\floor{g(Wn+b)}}f_1\cdot\ldots\cdot  T_{\ell}^{\floor{g(Wn+b)}}       f_{\ell}  }_{L^2(\m)}=o_w(1).
    \end{equation*}This follows if we show that \begin{equation*}
          \limsup\limits_{N\to+\infty}\max_{\underset{(b,W)=1}{1\leq b\leq W}}\ \Bignorm{\E_{N\leq n\leq N+L(N)} \big(\La_{w,b}(n) -1\big) T_1^{\floor{a(Wn+b)}}f_1\cdot\ldots\cdot  T_{\ell}^{\floor{a(Wn+b)}}       f_{\ell}  }_{L^2(\m)}=o_w(1).
    \end{equation*}
    This derivation is very similar to the proof of \cite[Lemma 4.3]{Fra-Hardy-singlecase}, which was stated only for bounded sequences.
    This is proven by covering the interval $[1,N]$ with non-overlapping sub-intervals that have the form $[m,m+L(m)]$ (for $m$ large enough), where the term of the average on the last set of the covering is bounded as in \eqref{E: to kommati poy jefeygei}. 
    \footnote{In particular, this case is much simpler than the method used to establish Theorem \ref{T: the main comparison}, in that we do not have to consider the more complicated double averaging scheme. In addition, we do not need any assumptions on $L(t)$ other than it is positive and $L(t)\prec t$.}

    Using Proposition \ref{P: remove error term for fast functions} and the abbreviated notation $g_{W,b}(t)$ for the function $g(Wt+b)$, we deduce that we can write \begin{equation*}
        \floor{g_{W,b}(n) }=\floor{g_{W,b}(N)+\dots+\frac{(n-N)^kg^{(k)}_{W,b}(N)}{k!}}
    \end{equation*}for all except at most $O(L(N)\log^{-100}N)$ values of $n\in [N,N+L(N)]$. Furthermore, we also have the equidistribution assumption of Proposition~\ref{P: remove error term for fast functions}, which implies that Proposition~\ref{P: Gowers norm bound on variable polynomials} is applicable for the polynomial $$g_{W,b}(N)+\dots+\frac{(n-N)^kg^{(k)}_{W,b}(N)}{k!}$$ appearing in the iterates. The conclusion then follows similarly as in the proof of Theorem~\ref{T: the main comparison}, so we omit the rest of the details.
\end{proof}

An application of the previous comparison is for the class of {\em tempered} functions, which we define promptly.

\begin{definition}
    Let $i$ be a non-negative integer. A real-valued function $g$ which is $(i+1)$-times continuously differentiable on $(t_0,\infty)$ for some $t_0\geq 0,$ is called a \emph{tempered function of degree $i$} (we write $d_g=i$), if the following hold:
\begin{enumerate}
\item[$(a)$] $g^{(i+1)}(t)$ tends monotonically to $0$ as $t\to\infty;$

\item[$(b)$]  $\lim_{t\to+\infty}t|g^{(i+1)}(t)|=+\infty.$
\end{enumerate}
Tempered functions of degree $0$ are called \emph{Fej\'{e}r} functions. 
\end{definition}

For example, consider the functions 
\begin{equation}\label{E: Examples tempered} g_1(t)=t^{1/25}(100+\sin\log t)^3, \; g_2(t)=t^{1/25}, \; g_3(t)=t^{17/2}(2+\cos\sqrt{\log t}).\end{equation}
We have that $g_1$ and $g_2$ are Fej\'er, $g_3$ is tempered of degree $8$ (which is not Hardy, see \cite{Bergelson-Haland}). 
Every tempered function of degree $i$ is eventually monotone and it grows at least as fast as $t^i\log t$ but slower than $t^{i+1}$ (see \cite{Bergelson-Haland}), so that, under the obvious modification of Definition \ref{D: growthdefinitions}, tempered functions $\mathcal{T}$ are strongly non-polynomial. Also, for every tempered function $g,$ we have that $(g(n))_{n\in\N}$ is equidistributed mod 1.\footnote{ For Fej\'er functions this is a classical result due to Fej\'er (for a proof see \cite{Kuipers-Niederreiter}). The general case follows inductively by van der Corput's difference theorem.}

In general, it is more restrictive to work with tempered functions than working with Hardy field ones. To see this, notice that ratios of tempered functions need not have limits, in contrast to the Hardy field case. For example, the functions $g_1$ and $g_2$ in \eqref{E: Examples tempered} are such that $g_1(t)/g_2(t)$ has no limit as $t\to+\infty$. This issue persists even when we are dealing with a single function, as ratios that involve derivatives of the same function may not have a limit either. Indeed, we can easily see that $g_1$ from \eqref{E: Examples tempered} (which was first studied in  \cite{DKS-pointwise}) has the property that $\frac{tg_1'(t)}{g_1(t)}$ does not have a limit as $t\to+\infty.$ The existence of the limit of the latter is important as it allows us to compare (via L' H\^opital's rule) growth rates of derivatives of functions with comparable growth rates.

In order to sidestep the aforementioned problematic cases, we restrict our study to the following subclass of tempered functions (see also \cite{Bergelson-Haland}, \cite{Koutsogiannis-tempered}).

Let $\mathcal{R}:=\Big\{g\in C^\infty(\mathbb{R}^+):\;\lim_{t\to+\infty}\frac{tg^{(i+1)}(t)}{g^{(i)}(t)}\in \mathbb{R}\;\;\text{for all}\;\;i\in\mathbb{N}\cup\{0\}\Big\};$

$\mathcal{T}_i:=\Big\{g\in\mathcal{R}:\;\exists\;i<\alpha< i+1,\;\lim_{t\to+\infty}\frac{tg'(t)}{g(t)}=\alpha,\;\lim_{t\to+\infty}g^{(i+1)}(t)=0\Big\};$ 

and $\mathcal{T}:=\bigcup_{i=0}^\infty \mathcal{T}_i.$ For example, $g_2\in \mathcal{T}_0$ and $g_3\in\mathcal{T}_8$ ($g_2, g_3$ are those from \eqref{E: Examples tempered}).


Notice that while the class of Fej\'er functions contain sub-fractional functions, $\mathcal{T}_0$ does not as, according to \cite[Lemma~6.4]{DKS-hardy}, if $g\in \mathcal{T}$ with $\lim_{t\to+\infty}\frac{tg'(t)}{g(t)}=\alpha,$ then for every $0<\beta<\alpha$ we have $t^\beta\prec g(t).$ 

We will prove a convergence result for the class $\mathcal{T}$ through an application of Proposition~\ref{P: comparison for more general iterates}.

\begin{lemma}
    Let $g$ be a function in $\mathcal{T}$ and $0<c<1$. Then,
for all large enough positive integers $k$, we have 
 \begin{equation*}
   t^c\prec \bigabs{g^{(k)}(t)}^{-\frac{1}{k}}\lll  \bigabs{g^{(k+1)}(t)}^{-\frac{1}{k+1}} \prec t.
\end{equation*}
\end{lemma}
\begin{proof}
     Since $g(t)\prec t^{d_g+1}$ and $0<c<1,$ we have $g(t)\prec t^{k(1-c)}$ for all large enough $k\in \N$, which implies
$$\frac{g^{(k)}(t)}{t^{-ck}}=\frac{g(t)}{t^{k(1-c)}}\cdot \prod_{i=1}^k \frac{tg^{(i)}(t)}{g^{(i-1)}(t)}\to 0.$$
Hence, $g^{(k)}(t)\prec t^{-ck}$ or, equivalently, $t^c\prec \bigabs{g^{(k)}(t)}^{-\frac{1}{k}}.$

For the aforementioned $k$'s, let $0<q<1$ so that $t^{kq}\prec g(t).$ Since $\lim_{t\to+\infty}\frac{tg'(t)}{g(t)}\notin \N,$
$$\frac{t^{k(q-1)}}{g^{(k)}(t)}=\frac{t^{kq}}{g(t)}\cdot \prod_{i=1}^k \frac{g^{(i-1)}(t)}{tg^{(i)}(t)}\to 0,$$
so $t^{k(q-1)}\prec g^{(k)}(t).$ As $\lim_{t\to+\infty}\frac{tg^{(k+1)}(t)}{g^{(k)}(t)}\in\R\setminus\{0\},$ we get $g^{(k+1)}(t)\ll t^{-1}g^{(k)}(t)$, so, if we let $\delta=\frac{q}{k+1},$ we have
$$\frac{\bigabs{g^{(k+1)}(t)}^{-\frac{1}{k+1}}}{\bigabs{g^{(k)}(t)}^{-\frac{1}{k}}}\gg \frac{t^{\frac{1}{k+1}}\bigabs{g^{(k)}(t)}^{-\frac{1}{k+1}}}{\bigabs{g^{(k)}(t)}^{-\frac{1}{k}}}=t^{\frac{1}{k+1}}\bigabs{g^{(k)}(t)}^{\frac{1}{k(k+1)}}\succ t^{\frac{1}{k+1}}\cdot t^{\frac{q-1}{k+1}}=t^\delta,$$
completing the proof of the lemma (the rightmost inequality follows by \cite{DKS-hardy}). 
\end{proof}

Using Proposition \ref{P: comparison for more general iterates} and \cite[Theorem~2.2]{Fra-Hardy-singlecase} we get the following result. More precisely, we use the fact here that \cite[Theorem~2.2]{Fra-Hardy-singlecase} holds for a single function $a$ which has the property that, for some $k\in\N,$ $a$ is $C^{k+1},$ $a^{(k+1)}(t)$ converges to $0$ monotonically, $1/t^k\prec a^{(k)}(t),$ and $|a^{(k)}(t)|^{-1/k}\prec |a^{(k+1)}(t)|^{-1/(k+1)}$ (see comments in \cite[Subsection 2.1.5]{Fra-Hardy-singlecase}). We omit its proof as it is identical to the one of Theorem~\ref{T: convergence of Furstenberg averages}.

\begin{theorem}\label{T: convergence of Furstenberg averages tempered}
    Let $g\in \mathcal{T}.$ For any $k\in\N,$ measure-preserving system $(X,\X, \m, T),$ and functions $f_1,\dots,f_{k}\in L^{\infty}(\m)$, we have 
     \begin{equation}\label{E: aek ole tempered}
        \lim_{N\to+\infty} \frac{1}{\pi(N)}\sum_{p\in \mathbb{P}\colon p\leq N} T^{\floor{g(p)}}f_1\cdot\ldots\cdot T^{k\floor{g(p)}}f_k=\lim_{N\to+\infty} \frac{1}{N}\sum_{n=1}^{N} T^nf_1\cdot\ldots\cdot T^{kn} f_k,
    \end{equation} where the convergence takes place in $L^2(\m)$.
\end{theorem}

As in the Hardy field case, we have the corresponding recurrence result.

\begin{theorem}
    \label{T: recurrence tempered}
Let $g\in \mathcal{T}.$ For any $k\in\N,$ measure-preserving system $(X,\X, \m,T),$  and set $A$ with positive measure, we have 
     \begin{equation*}
        \lim\limits_{N\to+\infty} \frac{1}{\pi(N)} \sum_{p\in \P\colon p\leq N} \m(A\cap T^{-\floor{g(p)}}A\cap \dots\cap T^{-{k\floor{g(p)}}}A )> 0.
\end{equation*}
\end{theorem}

The latter implies the following corollary, which guarantees arbitrarily long arithmetic progressions, with steps coming from the class of tempered functions evaluated at primes.

\begin{corollary}\label{C: Szemeredi tempered}
Let $g\in \mathcal{T}.$ For any set $E\subseteq \N$ of positive upper density, and $k\in \N,$ we have
\begin{equation*}
        \liminf\limits_{N\to+\infty} \frac{1}{\pi(N)}\sum_{p\in \P\colon p\leq N} \bar{d}\big(E\cap (E-\floor{g(p)})\cap \dots\cap (E-k\floor{g(p)})\big)>0.
    \end{equation*}
\end{corollary}

\begin{comment*}
In Theorem~\ref{T: convergence of Furstenberg averages tempered}, and, thus, in Theorem~\ref{T: recurrence tempered} and Corollary~\ref{C: Szemeredi tempered}, the floor function can be replaced with either the function $\lceil\cdot\rceil$ or the function  $[[\cdot ]]$. Furthermore, in each of these results, one can alternatively evaluate the sequences along the affine shifts $ap+b,$ for $a, b\in \R$ with $a\neq 0.$
 \end{comment*}

As we saw, the comparison method provides results along primes through the corresponding results for averages along $\N$, though in the case of tempered functions, we do not have a comparison result of the same strength as Theorem \ref{T: the main comparison}. Nonetheless, it is expected that convergence results along $\N$ for iterates which are comprised of multiple tempered functions (or even combinations of tempered and Hardy field functions) can be transferred to the prime setting. Even in the case of averages along $\N$, the convergence results are still not established under the most general expected assumptions. For a single function and commuting transformations, a result in this direction was proven in \cite{DKS-hardy}.
We note that \cite[Theorem~6.1]{DKS-hardy} reflects the complexity of the assumptions we have to impose on the growth rates of functions to deduce such results. This analysis is beyond the scope of this paper.

 \bibliography{final}
\bibliographystyle{plain}

\end{document}